 \documentclass{amsart}
 \usepackage[foot]{amsaddr}
\usepackage[margin = 1in]{geometry}     
\usepackage[mathscr]{euscript}
\usepackage{mathtools}

\usepackage{graphicx}
\usepackage{dsfont}
\usepackage{booktabs}
\usepackage{pst-node}
\usepackage{tikz-cd} 
\usepackage{amssymb}
\usepackage[utf8]{inputenc}
\usepackage[T1]{fontenc}
\usepackage{bm}
\usepackage{amsthm}
\usepackage{tikz, xcolor}
\usetikzlibrary{positioning,arrows}
\usetikzlibrary{decorations.markings}
\tikzstyle{circledvertex} = [draw, black, shape=circle, minimum size=8pt, inner sep=1pt]
\tikzstyle{invisivertex} = [black, shape=rectangle, minimum size=0pt, inner sep=2pt]
\tikzstyle{point}=[draw, black, fill,shape=circle, minimum size=4pt, inner sep=0pt]
\tikzstyle{label} = [draw, violet, shape=circle, minimum size=8pt, inner sep=1pt]
\tikzstyle{glabel} = [draw, teal, shape=circle, minimum size=4pt, inner sep=1pt]
\tikzstyle{graylabel} = [draw, gray, shape=circle, minimum size=4pt, inner sep=1pt]

\tikzstyle{BVertex}  = [draw, black, fill, shape=circle, minimum size=3pt, inner sep=0pt]

\usepackage{graphicx}	
\usepackage{subcaption}
\newtheorem{theorem}{Theorem}[section]
\newtheorem{lemma}[theorem]{Lemma}
\newtheorem{prop}[theorem]{Proposition}
\newtheorem{cor}[theorem]{Corollary}

\newtheorem{definition}[theorem]{Definition}

\newtheorem{example}[theorem]{Example}
\newtheorem{remark}[theorem]{Remark}
\numberwithin{equation}{subsection}

\usepackage{tikz-cd}
\usepackage{bbm}
\usepackage{amssymb}
\usepackage{hyperref}
\usepackage{cleveref}
\DeclareMathOperator{\Z}{\mathbb{Z}}

\DeclareMathOperator{\R}{\mathbb{R}}
\DeclareMathOperator{\C}{\mathbb{C}}
\DeclareMathOperator{\Q}{\mathbb{Q}}
\DeclareMathOperator{\PP}{\mathbb{P}}
\DeclareMathOperator{\SSS}{\mathbb{S}}

\DeclareMathOperator{\hilb}{Hilb}

\DeclareMathOperator{\lie}{Lie}
\DeclareMathOperator{\type}{type}
\DeclareMathOperator{\sh}{sh}
\DeclareMathOperator{\indeg}{in_{deg}}
\DeclareMathOperator{\inn}{in}

\DeclareMathOperator{\I}{\mathcal{I}}
\DeclareMathOperator{\K}{\mathcal{K}}

\DeclareMathOperator{\conf}{Conf}

\DeclareMathOperator{\pt}{pt}

\newcommand{\triv}{\mathds{1}}

\newcommand{\G}{\mathcal{G}}
\newcommand{\J}{\mathcal{J}}
\newcommand{\gr}{\mathfrak{gr}}
\newcommand{\repsgn}{\chi^{\emptyset,(1,1)}}

\newcommand{\repone}{\chi^{(1,1),\emptyset}}

\newcommand{\zminus}{z_{1\overline{2}}}
\newcommand{\zone}{z_{1}}
\newcommand{\ztwo}{z_{2}}
\newcommand{\zbad}{z_{\overline{1}2}}
\newcommand{\zbadbad}{z_{\overline{12}}}
\newcommand{\zi}{z_{i}}
\newcommand{\zj}{z_{j}}
\newcommand{\zk}{z_{k}}
\newcommand{\zijminus}{z_{i\overline{j}}}
\newcommand{\zijbad}{z_{\overline{i}j}}
\newcommand{\zijbadbad}{z_{\overline{ij}}}
\newcommand{\zikminus}{z_{i\overline{k}}}

\newcommand{\yi}{y_{0\overline{0}i}}
\newcommand{\yj}{y_{0\overline{0}j}}
\newcommand{\yijminus}{y_{0i\overline{j}}}
\newcommand{\yijbad}{y_{0\overline{i}j}}

\newcommand{\yk}{y_{0\overline{0}k}}
\newcommand{\yjkminus}{y_{0j\overline{k}}}
\newcommand{\yjkbad}{y_{0\overline{j}k}}

\newcommand{\yikminus}{y_{0i\overline{k}}}
\newcommand{\yikbad}{y_{0\overline{i}k}}

\newcommand{\rthree}{\R^{3} \setminus \{ 0 \}}
\newcommand{\ijminus}{i\overline{j}}

\DeclareMathOperator{\sgn}{sgn}

\DeclareMathOperator{\des}{des}
\DeclareMathOperator{\Des}{Des}

\DeclareMathOperator{\ind}{Ind}

\newcommand{\bconfhom}{H^{*}\mathcal{Z}_{n}}
\newcommand{\bconfhomn}[1]{H^{*}\mathcal{Z}_#1}
\newcommand{\bconfhomlift}{H^{*}\mathcal{Y}_{n+1}}
\newcommand{\bconfhomliftn}[1]{H^{*}\mathcal{Y}_{#1}}
\newcommand{\bconf}{\mathcal{Z}_{n}}
\newcommand{\bconfn}[1]{\mathcal{Z}_{#1}}
\newcommand{\bconflift}{\mathcal{Y}_{n+1}}
\newcommand{\bconfliftn}[1]{\mathcal{Y}_{#1}}
\DeclareMathOperator{\mrdes}{MRDes}

\usepackage{amsmath}

\def\frake{{\mathfrak{e}}}
\def\frakg{{\mathfrak{g}}}
\def\sgnpart{(\lambda^{+}, \lambda^{-})}
\def\vazrep{G_{\sgnpart}}
\def\vazidem{\frakg_{\sgnpart}}

\DeclareMathOperator{\Y}{\mathcal{Y}}
\DeclareMathOperator{\V}{\mathcal{V}}
\def\cc{{\mathscr{C}}}

\subjclass[2020]{05Exx, 20F55, 55R80, 55N91, 14N20, 52C35}


\author{Sarah Brauner}

\title{A Type $B$ analog of the Whitehouse representation}
\address{Department of Mathematics, University of Minnesota, Twin Cities, MN}
\email{braun622@umn.edu}

\begin{document}
\maketitle

\begin{abstract}
We give a Type $B$ analog of Whitehouse's lifts of the Eulerian representations from $S_n$ to $S_{n+1}$ by introducing a family of $B_{n}$-representations that lift to $B_{n+1}$. As in Type $A$, we interpret these representations combinatorially via a family of orthogonal idempotents in the Mantaci-Reutenauer algebra, and topologically as the  graded pieces of the cohomology of a certain $\Z_{2}$-orbit configuration space of $\R^{3}$. We show that the lifted $B_{n+1}$-representations also have a configuration space interpretation, and further parallel the Type $A$ story by giving analogs of many of its notable properties, such as connections to equivariant cohomology and the Varchenko-Gelfand ring.
\end{abstract}
\tableofcontents

\section{Introduction}\label{sec:intro}
Let $V$ be a representation of a finite group $H$; then $V$ is said to have a \emph{lift} to a group $G$ containing $H$ if there is a representation of $G$ that restricts to $V$. The goal of this paper is to $(1)$ identify a family of representations of the hyperoctahedral group $B_{n}$ that decompose the regular representation $\Q[B_{n}]$ and lift to $B_{n+1}$, and $(2)$ interpret these representations combinatorially and topologically. 
\subsection{Type A Motivation}\label{sec:typeamotivation}
Our work is inspired by the well-documented Type $A$ story of a family of $S_{n}$-representations lifting to representations of $S_{n+1}$ studied by Whitehouse \cite{Whitehouse}, Early--Reiner \cite{Early-Reiner}, Mathieu \cite{mathieu1996hidden}, Getzler--Kapranov \cite{getzler1995cyclic}, Moseley--Proudfoot--Young \cite{moseley2017orlik}, and others. These $S_{n}$-representations and their lifts arose from two distinct perspectives. The first is via a family of orthogonal idempotents $\{\frake_{k} \}_{0 \leq k \leq n-1}$ known as the \emph{Eulerian idempotents}. The $\frake_k$ lie in \emph{Solomon's descent algebra} $\Sigma[S_{n}]$, the subalgebra of $\Q[S_{n}]$ generated by sums of permutations $\sigma=(\sigma_1,\ldots,\sigma_n)$ with the same descent set
\[ \Des(\sigma_{1}, \cdots, \sigma_{n}) := \{ i \in [n-1]: \sigma_{i}> \sigma(i+1) \}. \]
The Eulerian idempotents have been extensively researched in the world of algebraic combinatorics, and generate the \emph{Eulerian representations} $E^{(k)}_{n}:=\frake_k \Q[S_{n}]$, which lift to a family of $S_{n+1}$-representations called the \emph{Whitehouse representations} $F_{n+1}^{(k)}$ \cite{Whitehouse}.

The second viewpoint comes from the study of the configuration space
\[ \conf_{n}(\R^{d}):= \{ (x_1, \cdots, x_n) \in \R^{dn}: x_i \neq x_j \}. \] 
As a ring, the cohomology $H^{*}\conf_{n}(\R^{d})$ has an elegant description due to Arnol'd \cite{arnold} (for $d=2$), F. Cohen \cite{cohen} (for $d\geq 2$), and Varchenko-Gelfand (for $d=1$) \cite{varchenkogelfand}. When $d$ is even, $H^{*}\conf_{n}(\R^{d})$ is the Type $A$ \emph{Orlik-Solomon} algebra. The relevant scenario here will be when $d$ is odd, which can be split into two cases:
\begin{itemize}
    \item $d=1$: The space $\conf_n(\R)$ is the complement of the Braid arrangement, and its cohomology $H^{*}\conf_n(\R)$ is a disjoint union of $n!$ contractible pieces. As a representation, $H^{*}\conf_n(\R)$ is the regular representation $\Q[S_n]$. Varchenko--Gelfand gave a presentation of $H^{*}\conf_n(\R)$ in terms of certain combinatorial functions called \emph{Heaviside functions}; these functions impose an ascending filtration by degree with associated graded ring $\gr(H^{*}\conf_n(\R))$ \cite{varchenkogelfand}.
    For $0 \leq k \leq n-1$, denote by $\gr(H^{*}\conf_n(\R))_k$ the $k$-th graded piece of $\gr(H^{*}\conf_n(\R))$.
    \item $d \geq 3$ and odd: In this case, $H^{*}\conf_n(\R^d)$ is concentrated in degrees $0, \ d-1, \ 2(d-1), \cdots, \ (n-1)(d-1)$. The presentation for $H^{*}\conf_n(\R^d)$ and the representations carried on each graded piece are the same for all odd $d \geq 3$, and in this way one obtains a family of $n$ representations of $S_n$ which decompose $\Q [S_n]$.
\end{itemize}

When $d = 1$ and $d=3$, one obtains lifts from $S_n$ to $S_{n+1}$ of the representations carried by $H^*\conf_n(\R^d)$. Recall that $U(1)$, the unitary group, is homeomorphic to the 1-sphere $\SSS^1$ (e.g. the 1-point compactification of $\R$) and $SU_2$, the group of $2 \times 2$ unitary matrices of determinant $1$ over $\C$, is homeomorphic to $\SSS^{3}$ (e.g. the 1-point compactification of $\R^3$). Our lifts will come from the quotient spaces 
\begin{equation}\label{eq:aliftdef} \V_{n+1}^{1}:= \conf_{n+1}(U(1))/U(1),  \hspace{5em} \V_{n+1}^{3}:= \conf_{n+1}(SU_{2})/SU_{2}; \end{equation}
in both cases the quotient is by (left) diagonal multiplication.

As in the case of $\conf_n(\R)$, the cohomology of $\V_{n+1}^1$ has a presentation in terms of \emph{cyclic Heaviside functions} due to Moseley--Proudfoot--Young \cite{moseley2017orlik}, and again the cyclic Heaviside functions give rise to an ascending filtration with associated graded ring $\gr(H^{*}\V_{n+1}^{1}) = \bigoplus_{k=0}^{n-1} \gr(H^{*}\V_{n+1}^{1})_k$. The presentation of $H^{*}\V_{n+1}^{3}$ was computed by Early--Reiner in \cite{Early-Reiner}.

Though not obvious, these viewpoints---the Eulerian and Whitehouse representations on one hand and ``$d$ odd'' configuration space cohomology on the other---turn out to be equivalent and serve as a beautiful link between classical combinatorial objects and important topological ones. 
In particular they are connected by the following representation isomorphisms:
\begin{align}
     E_{n}^{(n-1-k)}  &\cong_{S_n} H^{2k}\conf_n(\R^3) \cong_{S_n} \gr(H^{*}\conf_n(\R))_k  \label{aiso1}\\
  F_{n+1}^{(n-1-k)} &\cong_{S_{n+1}} H^{2k}\V_{n+1}^3 \cong_{S_{n+1}}   \gr(H^{*}\V_{n+1}^{1})_k  \label{aiso2}.
\end{align} 
Note that each term in \eqref{aiso1} simultaneously lifts to the corresponding term in \eqref{aiso2}. The first isomorphism in \eqref{aiso1} was proved by the author in \cite{brauner2020eulerian} in the context of Coxeter groups, while the second is a result of Moseley \cite{moseley2017equivariant}. The isomorphisms in \eqref{aiso2} are due to Early--Reiner \cite{Early-Reiner} and Moseley--Proudfoot--Young \cite{moseley2017orlik}, respectively. In Section \ref{sec:propertywishlist}, we summarize the notable properties of the Eulerian and Whitehouse representations, including a recursion relating them, a 
description of $H^{*}\V_{n+1}^3$ as an induced representation, and connections to equivariant cohomology. Our main results provide a Type $B$ analog of each of these properties.

\subsection{Type $B$ analog}
Our goal is to construct an analog to both Type $A$ perspectives discussed above for Type $B$.
%\subsubsection{Type $B$ idempotents}
Perhaps the most obvious generalization of the Type $A$ descent algebra $\Sigma[S_n]$ is the Type $B$ descent algebra $\Sigma[B_n]$, with Coxeter length used to describe $\Des(\sigma)$ and Type $B$ Eulerian idempotents\footnote{In fact, in \cite{bergeron1992decomposition}, Bergeron--Bergeron--Howlett--Taylor defined analogous idempotents for any finite Coxeter group.} defined by Bergeron--Bergeron in \cite{Bergeron-Bergeron}. However, it turns out that the corresponding Eulerian representations of $B_{n}$ (studied by the author in \cite{brauner2020eulerian} for instance) do \emph{not} lift to $B_{n+1}$! 

Instead, the right analogy in the context of lifts is to replace Solomon's descent algebra by the \emph{Mantaci-Reutenauer algebra} $\Sigma'[B_{n}]$, a combinatorially defined subalgebra of $\Q[B_{n}]$ generalizing $\Sigma[S_{n}]$ and containing $\Sigma[B_{n}]$. It was introduced by Mantaci--Reutenauer in \cite{mantaci1995generalization} and has a basis indexed by signed integer compositions $(a_1, \cdots, a_\ell)$, where $a_i \in \Z \setminus \{0\}$ and $|a_1| + \cdots |a_\ell| = n$. The Mantaci-Reutenauer algebra has rich combinatorics and representation theory studied by Aguiar--Bergeron--Nyman \cite{aguiar2004peak},  Bonnaf\'e--Hohlweg \cite{bonnafe2006generalized} and Douglass--Tomlin \cite{douglass2018decomposition}, much of which generalizes properties of $\Sigma[S_n]$.

The role of the Eulerian idempotents will be played by a family of orthogonal idempotents $\{ \mathfrak{g}_{k} \}_{0 \leq k \leq n},$ obtained as a sum of idempotents $\vazidem$ in $\Sigma'[B_{n}]$ introduced by Vazirani \cite{vazirani}:
\begin{equation}
    \frakg_{k}:= \sum_{\substack{(\lambda^{+}, \lambda^{-}) \\ \ell(\lambda^{+}) = k}} \frakg_{(\lambda^{+}, \lambda^{-})}.
\end{equation}

The above analogies are quite natural in the following sense. Let $\tau: B_{n} \to S_{n}$ be the projection which forgets the signs of $\sigma \in B_{n}$. In \cite{aguiar2004peak}, Aguiar--Bergeron--Nyman study the properties of $\tau$ and show that it extends to a surjective algebra homomorphism  $\tau: \Sigma'[B_{n}] \to \Sigma[S_{n}]$. In Proposition \ref{sgnmap}, we show that $\tau(\frakg_0) = 0$ and $\tau(\frakg_k) = \frake_{k-1}$ for $0 < k \leq n$.

The $B_{n}$-representations of interest are then defined for $0 \leq k \leq n$ to be
\[ G_{n}^{(k)}:=\frakg_{k}\Q[B_{n}] . \] 

We obtain analogs of the  ``hidden'' action\footnote{The word hidden refers to the fact that there is not an obvious $S_{n+1}$ action on $\conf_n(\R^{d})$.} spaces $\V_{n+1}^{1}$ and $\V_{n+1}^{3}$ by considering the  $\Z_{2}$-orbit configuration spaces on $U(1)$ and $SU_2$ induced from the antipodal action on the sphere, and taking their respective quotients by the relevant diagonal action:
\begin{equation}\label{eq:bliftdef} \Y_{n+1}^{1}:= \conf_{n+1}^{\Z_{2}}(U(1))/U(1),  \hspace{5em} \Y_{n+1}^{3}:= \conf_{n+1}^{\Z_{2}}(SU_{2})/SU_{2}. \end{equation} 
Both $\Y_{n+1}^{1}$ and $\Y_{n+1}^{3}$ carry a natural action by $B_{n+1}$ (e.g. permutation and negation of each coordinate), and we show that there are $B_n$-equivariant homeomorphisms $\Y_{n+1}^d \cong \bconf^d$ for $d=1,3$, where
\[ \bconf^d:= \conf_{n}^{\langle \varphi \rangle}(\R^{d} \setminus \{ 0 \}) := \bigg \{ (x_1, \cdots, x_n) \in (\R^d \setminus \{ 0 \})^n: x_{i} \neq x_{j} \neq -\frac{x_{j}}{|x_j|^{2}} \bigg \}. \]
The strangeness in the definition of $ \bconf^d$ (first studied by Feichtner--Ziegler in \cite{feichtner2002orbit}) comes from the fact that the composition of the antipodal map $x \mapsto -x$ with stereographic projection $\pi: \SSS^d \to \R^d$ gives a map $\varphi$ sending $x$ to $ -x/|x|^2.$ (Here we take take $|x|$ to be the magnitude of $x \in \R^d \setminus \{ 0 \}$ and $\SSS^d$ to be embedded in $\R^{d+1}$ by placing its south pole at the origin of $\R^{d+1}$.)

Nonetheless, many of the properties of $\bconf^d$ mirror those of $\conf_n(
\R^{d})$ for odd $d$. In particular, we again have two cases:
\begin{itemize}
     \item $d=1$: In direct parallel with Type $A$, the spaces $\bconflift^1$ and $\bconf^1$ are disjoint unions of $2^{n}n!$ contractible pieces. In Theorems \ref{thm:lift1pres} and \ref{cor:bconfhom1}, we give a presentation for $H^{*}\Y_{n+1}^{1}$ and $\bconfhom^1$, respectively. Our presentation is in terms of \emph{signed cyclic Heaviside functions} and \emph{signed Heaviside functions}; again, these combinatorial functions impose an ascending filtration on both rings by degree, with corresponding associated graded rings     \begin{equation}\label{eq:bgradings} \gr( H^{*}\Y_{n+1}^{1}) = \bigoplus_{k=0}^{n}\gr( H^{*}\Y_{n+1}^{1})_k , \hspace{5 em} \gr( \bconfhom^1)=\bigoplus_{k=0}^{n}\gr( \bconfhom^1)_k. \end{equation}
 \item $d\geq 3$ and odd: Once more, the analysis of $\bconfhom^d$ is identical for all odd $d \geq 3$, so assume $d=3$. Feichtner-Ziegler prove that $\bconfhom^3$ is concentrated in degrees $0, 2, 4, \cdots, 2n$; in Theorem \ref{thm:3pres} and Corollary \ref{thm:bpres3lift} we give presentations for $\bconfhom^3$ and $H^{*}\Y_{n+1}^{3}$, respectively, that bear a strong resemblance to the F. Cohen \cite{cohen} presentation for $H^*\conf_n(\R^{3})$ and Early--Reiner \cite{Early-Reiner} presentation for $H^{*}\V_{n+1}^{3}$. 
\end{itemize}

 Our main results give an analogy of the Type $A$ story: 

\begin{theorem}\label{intro:mainthm}
For $0 \leq k \leq n$, there are $B_{n}$-representation isomorphisms:
\begin{equation}\label{eq:biso1} G_{n}^{(n-k)} \cong H^{2k} \bconf^3 \cong  \gr(H^{*}\bconf^1)_k,  \end{equation}
lifting to $B_{n+1}$-representation isomorphisms
\begin{equation}\label{eq:biso2}  H^{2k} \bconflift^3 \cong  \gr(H^{*}\bconflift^1)_k.  \end{equation}
\end{theorem}
We further
\begin{itemize}
    \item show that the circle group $U(1)$ acts on both $\bconf^3$ and $\bconflift^3$ (Proposition \ref{prop:U1acts}), and compute presentations for their $U(1)$-equivariant cohomologies (Theorem \ref{eqcohompres}); 
    \item  prove that for both $\bconfhom^1$ and $\bconfhomlift^1$ the associated graded ring with respect to the filtration induced by the $U(1)$-equivariant cohomology coincides with \eqref{eq:bgradings} (Theorem \ref{cor:filtrationscoincide});
    \item compute a recursion relating $\bconfhom^3$ and $\bconfhomliftn{n}^3$ (Corollary \ref{cor:recursion}); and
    \item show that as ungraded rings,
\[ \bconfhom^3 \cong_{B_{n}} \Q[B_n], \hspace{5em} \bconfhomlift^3 \cong_{B_{n+1}} \ind_{\langle c \rangle }^{B_{n+1}} \triv, \]
where $c$ is a Coxeter element of $B_{n+1}$ and $\triv$ is the trivial representation. (Theorem \ref{ungradedrep}). 
\end{itemize}
In addition to drawing upon and generalizing techniques of Berget \cite{berget2018internal}, Moseley \cite{moseley2017equivariant}, and  Moseley--Proudfoot--Young \cite{moseley2017orlik}, the primary novelties in our methodology will be to (1) study the ``lifted'' spaces $\bconflift^d$ to deduce information about the spaces $\bconf^d$ and (2) to utilize various filtrations in the signed Heaviside and signed cyclic Heaviside functions.  

\subsection{Outline of the paper}
The remainder of the paper proceeds as follows:
\begin{itemize}
    \item Section \ref{sec:background} first gives a more detailed description of the Type $A$ motivation including a ``property wishlist'' for the Type $B$ analog (\S \ref{sec:propertywishlist}). We then discuss a general framework for obtaining hidden action spaces from orbit configuration spaces (Proposition \ref{hiddenactionframework}). Our Type $B$ work will serve as the primary example of this framework. We conclude with a review of properties of $B_n$ and its representation theory.
    \item Section \ref{sec:d1} focuses on the spaces $\Y_{n+1}^{1}$ and $\bconf^1$, and their cohomology. We introduce the signed cyclic Heaviside functions and use their combinatorial properties to give presentations for $H^{*}\Y_{n+1}^{1}$ (Theorem \ref{thm:lift1pres}) and $\bconfhom^1$ (Theorem \ref{cor:bconfhom1}), as well as their respective associated graded rings (Corollaries \ref{cor:lifthom1gr} and \ref{cor:grbconfhom}.)
    \item Section \ref{sec:d3} then considers the $d=3$ case and the spaces $\Y_{n+1}^{3}$ and $\bconf^3$; we prove a recursion relating their cohomologies (Corollary \ref{cor:recursion}), compute the presentation of $\bconfhom^3$, and show that it has a bi-grading (Corollary \ref{cor:bigrading}).
    \item Section \ref{sec:equivcohom} turns to equivariant cohomology. We compute $H_{U(1)}^{*}\bconf^3$  (Theorem \ref{eqcohompres}), as well as the ungraded representations $\bconfhom^3$ and $\bconfhomlift^3$ (Theorem \ref{ungradedrep}). This allows us to fully understand the relationship between the $d=1$ and $d=3$ cases by showing the filtrations by cyclic Heaviside functions and equivariant cohomology coincide (Theorem \ref{thm:liftedfiltrationscoincide}).
    \item Section \ref{sec:MRalg} introduces the idempotents $\vazidem$ and $\frakg_k$, as well as the representations of $B_n$ they generate. From here we are able to conclude Theorem \ref{intro:mainthm} (Corollary \ref{cor:mainiso}) and more specifically, relate the $\vazidem$ to $\bconfhom^3$.
\end{itemize}

\section{Background}\label{sec:background}
\subsection{Type $A$ revisited}
We begin by fleshing out the Type $A$ motivation described in Section \ref{sec:typeamotivation}.

One way to define the Eulerian idempotents $\frake_k$ is via the generating function due to Garsia--Reutenauer \cite{garsia}:
\begin{equation}\label{eq:typeaeulerian}
\sum_{k=0}^{n-1} t^{k+1}\frake_{k} = \sum_{\sigma \in S_{n}} \binom{t-1+n-\des(\sigma)}{n}\sigma.
\end{equation}
For more equivalent definitions, see Aguiar--Mahajan \cite{aguiarmahajan}, Loday \cite{loday}, Gerstenhaber--Schack \cite{Gerstenhaber-Schack}, and Saliola \cite{saliola2009face}.
\begin{example} \rm When $n=3$, the Eulerian idempotents are 
\begin{align*}
& \frake_{0} = \frac{1}{6} \big( 2 - (12) - (23) - (123) - (132) + 2(13) \big), \\
& \frake_{1} = \frac{1}{2} \big( 1 - (13) \big), \\
& \frake_{2} = \frac{1}{6} \big(1 + (12) + (23) + (13) + (123) + (132) \big). \\
\end{align*}
\end{example}
\begin{definition}\label{def:eulrep} \rm
The \emph{$k$-th Eulerian representation} is the right ideal
\[ E_{n}^{(k)}:= \frake_k \Q[S_n]. \]
\end{definition}

The Whitehouse lifts $F_{n+1}^{(k)}$ of the Eulerian representations are obtained by introducing an idempotent $f_{n+1}^{(k)}$ in $\Q[S_{n+1}]$ as follows.
View $S_{n} \leq S_{n+1}$ as the subgroup fixing the element $n+1$, let $w_{n+1}$ be the $n+1$ cycle $(12\dots (n+1)) \in S_{n+1}$, and define
\[ \mathscr{W}_{n+1} := \frac{1}{n+1} \sum_{i=0}^{n} (w_{n+1})^{i} .\]
Whitehouse shows the element $f_{n+1}^{(k)} := \mathscr{W}_{n+1} e_{n}^{(k)}$
is an idempotent in $\Q[S_{n+1}]$, 
generating a family of representations
\[ F_{n+1}^{(k)} :=  f_{n+1}^{(k)} \Q[S_{n+1}]  \]
which we will call the \emph{Whitehouse representations}. She then proves that the $F_{n+1}^{(k)}$ are lifts of the $E_{n}^{(k)}$ \cite[Prop 1.4]{Whitehouse}. 

\begin{example}[$n = 3$]  
\rm
Denote by $S^\lambda$ the irreducible symmetric group representation indexed by the partition $\lambda$. Then the $S_{3}$ Eulerian representations and their $S_{4}$ lifts are
\[
\begin{array}{rlccll}
E_{3}^{(0)} &= S^{(2,1)} & & &  F_{4}^{(0)} &= S^{(2,2)} \\
    E_{3}^{(1)} &= S^{(2,1)} \oplus S^{(1,1,1)} & & & F_{4}^{(1)} &= S^{(2,1,1)}\\
    E_{3}^{(2)} &= S^{(3)}& && F_{4}^{(2)} &= S^{(4)}.
\end{array}
\]
Each $F_{4}^{(k)}$ restricts to the representation $E_{3}^{(k)}$ via the symmetric group branching rules.
\end{example}
\subsubsection{The spaces $\conf_n(\R^d)$ and their cohomology}
We begin with the case that $d =1$. Recall that the space $\conf_{n}(\R)$ is a disjoint union of $n!$ contractible pieces. Each piece is parametrized by a relative ordering of $x_{1}, \cdots, x_{n}$ in $\R$, and so $H^{*}\conf_{n}(\R, \Z) = H^{0}\conf_{n}(\R, \Z)$ can be understood as the space of $\Z$-valued functions on the set of connected components of $\conf_{n}(\R)$. 

Varchenko-Gelfand give a combinatorial set of generators for $H^{0}\conf_{n}(\R)$ called 
\emph{Heaviside functions}, defined by
\[ u_{ij}(x_{1}, \cdots, x_{n}) := \begin{cases} 1 & x_{i} < x_{j} \\
0 & x_{i} > x_{j}
\end{cases} \]
for $i \neq j \in [n] := \{ 1, \cdots, n \}$. A permutation $\sigma \in S_n$ naturally acts on $u_{ij}$:
\[ \sigma \cdot u_{ij} =  u_{\sigma(i)\sigma(j)}. \]
The space of such Heaviside functions forms a $\Z$-algebra, where the $u_{ij}$ add and multiply pointwise as functions, so that multiplication is given by
\[
    u_{ij}\cdot u_{k\ell}(x_{1}, \cdots, x_{n}) = \begin{cases} 1 & x_{i} < x_{j} \textrm{ and } x_{k} < x_{\ell} \\
    0 & \textrm{ otherwise.}
    \end{cases}
\]
This implies certain natural relations, for example that $u_{ij}^{2} = u_{ij}$. Similarly, one can deduce that $1-u_{ij} = u_{ji}$, so that $u_{ij}\cdot u_{jk} \cdot (1-u_{ik}) = 0$, since it is impossible that $x_{i} < x_{j} < x_{k}$ but $x_{i} > x_{k}$. This is the essential idea behind Theorem \ref{thm:vgring}.

\begin{theorem}\label{thm:vgring}{\rm (Varchenko--Gelfand \cite{varchenkogelfand}).}
The ring $H^{*}\conf_{n}(\R)$ has presentation $\Z[u_{ij}]/\mathcal{Q}$, where $\mathcal{Q}$ is generated by  
\[ (i) \hspace{.5em} u_{ij}^{2} = u_{ij}, \hspace{2.5em}(ii)\hspace{.5em} u_{ij} = (1-u_{ji}), \hspace{2.5em} (iii)\hspace{.5em} u_{ij}u_{jk}(1-u_{ik}) - (1-u_{ij})(1-u_{jk})u_{ik} = 0, \]
for distinct $i,j,k \in [n]$.
\end{theorem}
Call the ring $\Z[u_{ij}]/\mathcal{Q}$ the \emph{Varchenko--Gelfand ring.} The presentation in Theorem \ref{thm:vgring} imposes an ascending filtration on the Varchenko-Gelfand ring obtained from the natural degree grading on $\Z[u_{ij}]/\mathcal{Q}$: the $m^{th}$ layer in the filtration is the span of monomials in the variables $u_{ij}$ having degree at most $m$. The associated graded ring from this filtration has presentation $\Z[u_{ij}]/\inn_{\deg}(\mathcal{Q})$ with $\inn_{\deg}(\mathcal{Q})$ generated by\footnote{The notation $\inn_{\deg}(\mathcal{Q})$ refers to the fact that $\inn_{\deg}(\mathcal{Q})$ is an initial ideal for some degree (partial) ordering on monomials. This will be rigorously defined and discussed in \S \ref{sec:signedheaviside}} 
\[ (i) \hspace{.5em} u_{ij}^{2} = 0, \hspace{2.5em}(ii)\hspace{.5em} u_{ij} = -u_{ji}, \hspace{2.5em} (iii)\hspace{.5em} u_{ij}u_{jk} - u_{ij}u_{ik} - u_{jk}u_{ik} = 0. \]

When $d \geq 2$, the space $\conf_{n}(\R^{d})$ is no longer comprised of contractible, disjoint pieces but nonetheless has an elegant presentation due to F. Cohen ($d \geq 2$) and Ar'nold ($d=2$). 
\begin{theorem}\label{thm:cohen} {\rm (F. Cohen \cite{cohen}, Ar'nold\cite{arnold}).}
For $d \geq 2$, the ring $ H^{*}\conf_{n}(\R^{d})$ has presentation $\Z\langle u_{ij} \rangle/\mathcal{Q}'$ where $\mathcal{Q}'$ is generated by the relations $u_{ij}u_{k\ell} = (-1)^{d+1}u_{k\ell}u_{ij}$ and 
 \[(i)\hspace{.5em}u_{ij}^{2} = 0  \hspace{2em} (ii)\hspace{.5em}u_{ij} = (-1)^{d}u_{ji}, \hspace{2em} (iii)\hspace{.5em}u_{ij}u_{jk} + u_{jk}u_{ki} + u_{ki}u_{ij} = 0 \]
 for distinct $i,j,k,\ell \in [n]$.
\end{theorem}
\noindent As in the $d=1$ case, $S_n$ acts on the $u_{ij}$ by permuting coordinates.

 The generator $u_{ij}$ lies in $H^{d-1}\conf_{n}(\R^{d})$, which together with the relations in $\mathcal{Q}'$, implies that $H^{*}\conf_{n}(\R^{d})$ is concentrated in degrees $0, (d-1), 2(d-1), \cdots, (n-1)(d-1).$ Direct comparison of the above presentations shows that $\Z[u_{ij}]/\inn(\mathcal{Q}) \cong H^{*}\conf_n(\R^{d})$ for $d \geq 3$ and odd.

\begin{example}[$\lie_n$]\label{ex:lie} \rm
The top degree cohomology of $\conf_n(\R^{d})$ has a particularly nice description: 
\[ H^{(d-1)(n-1)}\conf_n(\R^{d}) \cong_{S_n} \begin{cases} \lie_n  & d \geq 3 \textrm{ and odd}\\
\lie_n \otimes \sgn & d \textrm{ even,}
\end{cases}\]
where $\lie_n$ is the \emph{multilinear component of the free Lie algebra on $n$ generators}; see Reutenauer \cite{reutenauer2001free} for more on the many wonderful properties of $\lie_n$. A result of Kraskiewicz--Weyman \cite{kraskiewiczalgebra} says that $\lie_n$ can be described as the induced representation 
\[ \lie_n = \ind_{\Z_n}^{S_n} \omega, \]
where $\Z_n$ is the cyclic group generated by an $n$-cycle in $S_n$ and $\omega$ is the character sending the generator of $\Z_n$ to $e^{2 \pi i/n}$. 

The relations in Theorem \ref{thm:cohen} imply that $\lie_n$ has basis 
\[ u_{1i_{1}}u_{i_{1}i_{2}}u_{i_{2}i_{3}}\cdots u_{i_{n-1}i_{n}} \]
for distinct $i_{1}, i_{2}, \cdots, i_{n} \in [n]$, and is therefore $(n-1)!$ dimensional. There are other notable bases for $\lie_n$, in particular using Lyndon words; see Barcelo \cite{barcelo1990action}.
\end{example}

We will return to $\lie_n$ in Section \ref{sec:MRalg}, where it will play a key role in describing $\bconfhom^3$.

\subsubsection{Hidden actions and lifts}
Assume that $d = 1,3$. The fact that $H^{*}\conf_n(\R^{d})$ lifts to an $S_{n+1}$-representation comes from the fact that the spaces $\conf_{n}(\R^d)$ have \emph{hidden} (e.g. non-obvious) actions of $S_{n+1}$ via the $S_n$-equivariant maps 
\begin{align}
\label{eq:ahomeo}    f^{d}_{A}: \V_{n+1}^d &\xrightarrow[]{\cong} \conf_n(\R^d)\\
      (p_0, \cdots, p_{n-1}, p_{n}) &\mapsto (\pi(p_{0}^{-1}p_1), \cdots, \pi(p_{0}^{-1}p_n))
\end{align}
where $\pi: \SSS^d \setminus \{\infty\} \to \R^d$ is the stereographic projection, and as in the introduction
\begin{equation*} \V_{n+1}^{1}:= \conf_{n+1}(U(1))/U(1),  \hspace{5em} \V_{n+1}^{3}:= \conf_{n+1}(SU_{2})/SU_{2}. \end{equation*}
Here we identify the element $1$ in $SU_{2}$ and $U(1)$ with $\infty$ (the point at infinity) in $\SSS^3$ and $\SSS^1$, so that $\pi$ is a homeomorphism between $SU_{2} \setminus \{ 1 \} \to \R^3$ and $U(1) \setminus \{ 1 \} \to \R$.

One can thus define an $S_{n+1}$-action on
 $\conf_n(\R^d)$ as the natural $S_{n+1}$-action by coordinate permutation on $\V_{n+1}^{d}$. Let 
 \[ [n]_{0}:= \{ 0, 1, \cdots, n \}.\] 
 We will think of $S_{n+1}$ as the group of permutations on the set $[n]_{0}$.  The homeomorphism \eqref{eq:ahomeo} implies that one recovers the standard $S_{n}$-action on $\conf_n(\R^d)$ by permuting only the last $n$ coordinates of $ \V_{n+1}^{d}$. We will discuss a more general framework for obtaining hidden actions in \S \ref{hiddenactions}.

The space $ \V_{n+1}^{1}$ can seem unwieldy, but it actually has an intuitive description;
 $ \V_{n+1}^{1}$ has representatives $(1, p_{1}, \cdots, p_{n})$ for $p_{i} \neq p_{j} \neq 1$ and like $\conf_{n}(\R)$, is comprised of $n!$ contractible pieces. Each disjoint piece of $\V_{n+1}^{1}$ is parametrized by a relative ordering of $p_{0}, \cdots, p_{n}$ around the circle. These disjoint pieces ($S_{n}$-equivariantly) biject with the pieces of $\conf_{n}(\R)$. To move from $\V_{n+1}^{1}$ to $\conf_{n}(\R)$, read the ordering of $p_{1}, \cdots, p_{n}$ around $U(1)$ counter-clockwise beginning after $p_{0}$. 
 
 When we move to cohomology, the Heaviside functions $u_{ij}$ also lift to \emph{cyclic Heaviside functions} $v_{ijk} \in  \V_{n+1}^{1}$, defined in \cite{moseley2017orlik} by Moseley--Proudfoot--Young as:
\[ v_{ijk}(p_{0}, \cdots, p_{n}) := \begin{cases} 1 & p_{i} < p_{j} < p_{k} \textrm{ in counter-clockwise order on } U(1)\\
0 & \textrm{ otherwise,}
\end{cases} \]
where now $i,j,k \in [n]_{0}$.

\begin{example} \rm
Consider the two representatives $\vec{q}$ and $\vec{r}$ of $\V_{3}^{1}$ and their images under $f_{A}^{1}$:
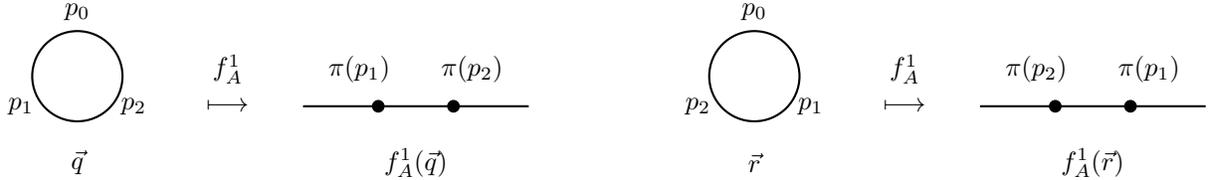
\begin{center}
\begin{tikzpicture}[thick]
 \draw[thick](-3,.4) circle(.6);
 \draw[thick](-3+9,.4) circle(.6);

 % State q2
\draw (-5+5,0)-- (-2+5,0);
 % State q1
\node[point] (a) at (-4+5,0){};
\node[] (aa) at (-4.25 +5,.5){$\pi(p_{1})$};

% State q0
\node[point] (b) at (-3+5,0){};
\node[] (bb) at (-2.75+5,.5){$\pi(p_{2})$};

\node[] (cc) at (-3.5+5,-.75){$f_{A}^{1}(\vec{q})$};

% State q2
\node[] (A) at (.25-4,0){$p_1$};
 
% State q1
\node[] (B) at (1.75-4,0){$p_2$};
 
% State q0
\node[] (C) at (1-4,1.25){$p_0$};

\node[] (cc) at (1-4,-.75){$\vec{q}$};

\node[] (x) at (-1,0){$\longmapsto$};
\node[] (xx) at (-1,.5){$f_{A}^{1}$};

% Transition q2 to q0
%\draw (A) to[bend left] node[left]{}(C) ;
 
% Transition q0 to q1
%\draw (C) to[bend left] node[right]{}(B);
 
% Transition q1 to q2
%\draw (B) to[bend left]node[below]{} (A);
 
% Initial state
\draw (5+4,0)-- (8+4,0);
 % State q1
\node[point] (a) at (6+4,0){};
\node[] (aa) at (5.75+4,.5){$\pi(p_{2})$};

% State q0
\node[point] (b) at (7+4,0){};
\node[] (bb) at (7.25+4,.5){$\pi(p_{1})$};

\node[] (cc) at (6.5+4,-.75){$f_{A}^{1}(\vec{r})$};

% State q2
\node[] (A) at (9.25-4,0){$p_2$};
 
% State q1
\node[] (B) at (10.75-4,0){$p_1$};
 
% State q0
\node[] (C) at (10-4,1.25){$p_0$};

 \node[] (cc) at (10-4,-.75){$\vec{r}$};

\node[] (x) at (8,0){$\longmapsto$};
\node[] (xx) at (8,.5){$f_{A}^{1}$};
% Transition q2 to q0
%\draw (A) to[bend left] node[left]{}(C) ;
 
% Transition q0 to q1
%\draw (C) to[bend left] node[right]{}(B);
 
% Transition q1 to q2
%\draw (B) to[bend left]node[below]{} (A);
 
% Initial state

\end{tikzpicture}
\end{center}
Note that $v_{123}(\vec{q}) =u_{12}(f_{A}^{1}(\vec{q})) =  1$, while $v_{123}(\vec{r}) =u_{12}(f_{A}^{1}(\vec{r})) = 0.$ On the other hand $v_{213}(\vec{q}) = u_{21}(f_{A}^{1}(\vec{q}))  = 0$ and $v_{213}(\vec{r}) = u_{21}(f_{A}^{1}(\vec{r})) =1$.
\end{example}

The $v_{ijk}$ again form a $\Z$-algebra and provide an elegant combinatorial description for the ring $H^* \V_{n+1}^{1}$.  As in the case of $H^{*}\conf_{n}(\R)$, the degree grading on $H^{*} \V_{n+1}^{1}$ from the $v_{ijk}$ imposes an ascending filtration with associated graded ring $\gr(H^{*} \V_{n+1}^{1})$.

For both $d = 1, 3$, an explicit presentation for $H^{*}\V_{n+1}^d$ can be recovered from the presentation in Theorem \ref{thm:vgring} via the induced isomorphism $f_{A}^{*}$ sending $u_{ij}$ to $v_{0ij}$, along with the additional relation due to Early--Reiner \cite[Thm 3]{Early-Reiner} (see also Moseley--Proudfoot--Young \cite[Rmk 2.9]{moseley2017orlik} and Matherne--Miyata--Proudfoot--Ramos \cite[Thm. A.4]{matherne2021equivariant}):
\begin{equation*}\label{eq:liftarelation}
v_{ijk} - v_{ij\ell} + v_{ik \ell} - v_{jk\ell} = 0.
\end{equation*}
\begin{remark}[Cyclic operads]\label{rmk:cyclicoperads}\rm
Kontsevich in \cite{kontsevich1993formal} showed there was a hidden action of $S_{n+1}$ on $\lie_n$. In \cite{getzler1995cyclic}, Getzler--Kapranov then formalized his techniques to introduce \emph{cyclic operads} more generally. Using constructions of $H_{*}\conf_n(\R^d)$ and Poisson operads (see for instance Sinha \cite{sinha2006homology}), one can apply the work of Getzler--Kapranov (\cite[Prop 3.11]{getzler1995cyclic}) to show that there is a cyclic operad structure on all graded pieces of the homology of $\conf_n(\R^3)$. It would be interesting to relate this cyclic operad structure to the bases of $H^*\conf_n(\R^3)$ and $H_{*}\conf_n(\R^3)$ more directly. 
\end{remark}
\subsubsection{Property Wishlist}\label{sec:propertywishlist}
We will finish the Type $A$ summary by providing a ``property wishlist'' which will serve as the inspiration and guiding motivation of our Type $B$ work. In particular, we highlight the following: 
\begin{itemize}
    \item There is an isomorphism of $S_n$-representations\footnote{In fact \eqref{eq:typeaeulerianconfig} holds for any odd $d \geq 3$ by replacing $H^{2k}\conf_{n}(\R^{d})$ with $H^{(d-1)k}\conf_{n}(\R^{d})$.} for $0 \leq k \leq n-1$:
\begin{equation}\label{eq:typeaeulerianconfig}
     E^{(n-1-k)}_{n} \cong_{S_{n}} H^{2k}\conf_{n}(\R^{3}).
\end{equation}
 This was first deduced by comparing a result of Sundaram and Welker for subspace arrangements \cite[Thm 4.4(iii)]{sundaramwelker} with descriptions of the characters of $E^{(k)}_{n}$ by Hanlon \cite{Hanlon}, and was later proved in the context of Coxeter groups by the author in \cite{brauner2020eulerian}.
\item Equation \eqref{eq:typeaeulerianconfig} lifts to an isomorphism of $S_{n+1}$-representations \cite[Prop. 2]{Early-Reiner}:
\begin{equation}\label{eq:typeawhitehouseconfig}
    F^{(n-1-k)}_{n+1} \cong_{S_{n+1}} H^{2k}(\V_{n+1}^{3}).
\end{equation}
\item There is a recursion\footnote{We think of this as a recursion in the sense that the formula relates the representation $E_{n}^{(k)}$ (which lifts to $F_{n+1}^{(k)}$) to the representation $F_{n}^{(k)}$ (which restricts to $E_{n-1}^{(k)}$).} relating the Eulerian and Whitehouse representations of $S_n$:
\begin{equation}
    E_{n}^{(k)} = F_{n}^{(k-1)} \oplus \left( S^{(n-1,1)} \otimes F_{n}^{(k)} \right),
\end{equation}
where $S^{(n-1,1)}$ is the irreducible reflection representation of $S_{n}$ \cite[Prop. 1]{Early-Reiner}.
\item There is an $S_{n+1}$-representation isomorphism
\[ \sum_{k=0}^{n-1} F_{n+1}^{(k)} \cong \ind_{\Z_{n+1}}^{S_{n+1}} \triv, \]
where $\Z_{n+1}$ is the cyclic group generated by an $(n+1)$-cycle in $S_{n+1}$ and $\triv$ is the trivial representation \cite[Prop 1.5]{Whitehouse}.
\item The circle group $U(1)$ acts on $\R^{3}$ by rotation around the $x$-axis, thereby inducing an action on $\conf_{n}(\R^{3})$. The filtration induced from the $U(1)$-equivariant cohomology $H^{*}_{U(1)}\conf_{n}(\R^{3})$ implies an isomorphism of $S_{n}$-modules for $0 \leq k \leq n-1$ \cite{moseley2017equivariant}:
\begin{equation}\label{eq:grtypea}
    \gr(H^{*}\conf_{n}(\R))_k \cong_{S_{n}} H^{2k}\conf_{n}(\R^{3}),
\end{equation}
where $ \gr(H^{*}\conf_{n}(\R))$ coincides with the associated graded ring coming from the filtration by Heaviside functions \cite{moseley2017equivariant}. 

\item Equation \eqref{eq:grtypea} also lifts to an $S_{n+1}$-module isomorphism for $0 \leq k \leq n-1$ \cite{moseley2017orlik}:
\begin{equation}\label{eq:liftedtypeagradediso}
    \gr(H^{*}\V_{n+1}^{1})_k \cong_{S_{n+1}} H^{2k}(\V_{n+1}^{3}).
\end{equation}
where again \eqref{eq:liftedtypeagradediso} comes from a $U(1)$ action on $\V^{3}_{n+1}$ and subsequent computation of $H^{*}_{U(1)}\V^{3}_{n+1}$. The grading on the left-hand-side also coincides with the associated graded ring coming from the filtration by cyclic Heaviside functions.
\end{itemize}
The remainder of this paper is devoted to obtaining an analogous Type $B$ statement for each of the properties listed above.

\subsection{Orbit configuration spaces and hidden actions}\label{hiddenactions}
We now introduce a general framework with which to construct spaces with hidden actions. Our Type $B$ spaces will emerge from this general framework (see Example \ref{ex:typebhiddenaction}).

\begin{definition}\label{def:orbconfig} \rm
For a group $H$ acting freely on a topological space $X$, the \emph{$n$-th ordered orbit configuration space} is 
\[ \conf_{n}^{H}(X) := \{(x_{1}, \dots, x_{n}) \in X^{n}: h\cdot x_{i} \cap h \cdot x_{j} = \emptyset \text{ for } i \neq j \textrm{ and any } h\in H \}. \]
\end{definition}

Orbit configuration spaces were first defined by Xicot{\'e}ncatl in \cite{xico}, and have proved integral to the study of universal covers of certain configuration spaces \cite{xico}, hyperplane arrangements associated to root systems by Bibby \cite{bibby2018representation} and Moci \cite{moci2008combinatorics}, and equivariant loop spaces by Xicot{\'e}ncatl \cite{xicotencatl2002product}. Computing the cohomology of orbit configuration spaces is an active area of study developed by Casto \cite{Casto}, Denham--Suciu \cite{denham2018local}, Feichtner--Ziegler \cite{feichtner2002orbit}, Bibby--Nadish  \cite{bibby2019generating}, \cite{bibby2018combinatorics} and others.

If $H$ is the trivial group, Definition \ref{def:orbconfig} recovers the classical configuration space of $X$. More generally any orbit configuration space $\conf_{n}^{H}(X)$ has a natural action by the wreath product $H \wr S_{n}$. Recall that the wreath product $H \wr S_{n}$ is the group whose elements are of the form 
\[ H \wr S_n = \{ (h_1, \cdots, h_n, \sigma): h_i \in H, \sigma \in S_n \}  \]
with multiplication defined by
\[  (h'_1, h'_2, \cdots, h'_n, \sigma') \cdot (h_1, h_2, \cdots, h_n, \sigma) = (h'_{\sigma(1)}h_1, h'_{\sigma(2)}h_2, \cdots, h'_{\sigma(n)}h_n, \sigma' \sigma  ). \]
For instance if $H = \Z_2$, this wreath product $\Z_2 \wr S_n$ is $B_n$, and if $H = \Z_r$ for $r \geq 2$, one obtains the complex reflection group $G(r,1,n)$ (sometimes referred to as a \emph{generalized symmtric group}).

Our goal is to study orbit configuration spaces equipped with a  hidden action. In other words, given an orbit configuration space $\conf_{n}^{H}(X)$, we would like to identify a space $Y$ on which $H \wr S_{n+1}$ acts, and an $H \wr S_{n}$-equivariant homeomorphism 
\[ \conf_{n}^{H}(X) \cong Y. \]
If such a homeomorphism exists, we say that $\conf_{n}^{H}(X)$ has a \emph{hidden action} by $H \wr S_{n+1}$. 
\begin{prop}\label{hiddenactionframework}
Let $X$ be a topological space with 
\begin{itemize}
    \item A transitive (left)-action by a group $G$ and 
    \item A free (left)-action by a group $H$ in the center of $G$. 
\end{itemize}
Then there is an $H \wr S_{n}$-equivariant homeomorphism
\[
 \conf_{n+1}^{H}(X)/G \quad \cong \quad  \conf_{n}^{H}(X-\mathcal{O}_{H}(x_{0}))/G_{x_{0}},
\]
where both quotients are by left-diagonal multiplication, $\mathcal{O}_{H}(x_{0})$ is the $H$-orbit of $x_{0} \in X$ and $G_{x_{0}}$ is the $G$-stabilizer of $x_{0}$.
\end{prop}
\begin{proof}

Since $H$ acts freely on $X$ it will also act freely on the space $X \setminus \mathcal{O}_{H}(x_{0}))$ obtained by removing the $H$-orbit of some point $x_{0}$ from $X$. The $G$-stabilizer of $x_{0}$, denoted $G_{x_{0}}$, then acts on $X \setminus \mathcal{O}_{H}(x_{0}))$ and since $H$ necessarily commutes with $G_{x_{0}},$ one also obtains a well-defined space $\conf_{n}^{H}(X-\mathcal{O}_{H}(x_{0}))/G_{x_{0}}.$

Fix a point $(p_{0},p_{1},\cdots, p_{n}) \in \conf_{n+1}^{H}(X)/G$; because $G$ acts transitively on $X$, there is some $g \in G$ sending $x_{0}$ to $p_{0}$. Thus in $\conf_{n+1}^{H}(X)/G$, 
\[ (p_{0},\cdots, p_{n}) \sim (x_{0},g^{-1}p_1,\ldots,g^{-1}p_n).\]
Then the maps back-and-forth giving inverse homeomorphisms are
\[
\begin{array}{rcl}
\conf^{H}_{n+1}(X)/G &\longrightarrow & \conf_{n}^{H}(X-\mathcal{O}_{H}(x_{0}))/G_{x_0}\\
(p_0,p_1,\ldots,p_n) & \longmapsto & (g^{-1}p_1,\ldots,g^{-1}p_n)\\
 & & \\
\conf_{n}^{H}(X-\mathcal{O}_{H}(x_{0}))/G_{x_0} &\longrightarrow & \conf^{H}_{n+1}(X)/G\\
(q_1,\ldots,q_n) & \longmapsto & 
(x_0,q_1,\ldots,q_n).
\end{array}
 \]
To check that the first map is well-defined, note that $p_{i}$ and $p_{j}$ are in the same $H$-orbit if and only if the same is true for $gp_{i}$ and $gp_{j}$. Thus $(g^{-1}p_{1}, \cdots g^{-1}p_{n})$ is indeed in $\conf_{n}^{H}(X-\mathcal{O}_{H}(x_{0}))/G_{x_{0}}$. Further, if $g_{1} x_{0} = p_{0}$, then for any $g_{2} \in G$ one has $g_{2}p_{0} = g_{2} g_{1} x_{0}$ and 
\[ \begin{array}{rcl}
(g_{2}p_{0}, \cdots, g_{2}p_{n}) & \longmapsto & (g_{1}^{-1}g_{2}^{-1}g_{2}p_{n}, \cdots, g_{1}^{-1}g_{2}^{-1}g_{2}p_{n}) = (g_{1}^{-1}p_{1}, \cdots, g_{1}^{-1}p_{n})\\
(p_{0}, \cdots, p_{n}) & \longmapsto &(g_{1}^{-1}p_{1}, \cdots, g_{1}^{-1}p_{n}).
\end{array} \]
One can similarly check that the second map is well-defined and that their composition gives the identity. The $H \wr S_{n}$ equivariance follows from the fact that $H$ commutes with $G$.
\end{proof}

We will be interested in the special case that $G$ acts simply transitively on $X$ so that $G_{x_{0}}$ is trivial, in which case one obtains the isomorphism
\[ \conf_{n}^{H}(X - \mathcal{O}_{H}(x_{0})) \cong \conf_{n+1}^{H}(X)/G. \]

\begin{example}\label{ex:typeahiddenaction} \rm
Take $H$ to be the trivial group. 
\begin{itemize}
    \item Let $X = G = SU_{2}$ and $x_0 = 1$, the identity of $SU_{2}$. Then Proposition \ref{hiddenactionframework} recovers the
   $S_{n}$-equivariant homeomorphism
\[ \conf_{n+1}(SU_{2})/ SU_{2} \cong \conf_{n}(\R^{3}), \]
where the map here is precisely \eqref{eq:ahomeo}, since $SU_{2} \setminus \{ 1 \} \cong \R^3$.
\item Let $X = G = U(1)$ and set $x_0 = 1 \in U(1)$, so that $U(1) \setminus \{ 1 \} \cong \R$. Then Proposition \ref{hiddenactionframework} (again via \eqref{eq:ahomeo}) gives
\[ \conf_{n+1}(U(1))/U(1)\cong \conf_{n}(\R). \]
\end{itemize}
\end{example}

One way of framing all subsequent work in this paper is to ask the question: what happens to Example \ref{ex:typeahiddenaction} when $H$ is replaced by $\Z_{2},$ acting via the antipodal map? 

\begin{example}\label{ex:typebhiddenaction} \rm
As suggested above, take $H = \Z_2$ acting as the antipodal map (e.g. by $-1$) on $SU_{2}$ and $U(1)$. Then
\begin{itemize}
    \item When $X = G = SU_{2}$ and $x_0 = 1$, the orbit of $1$ is $\pm 1$. Recall that for $x \in \R^{d} \setminus \{ 0 \}$, $\varphi(x) := -x/|x|^2$ and that $B_n \cong \Z_2 \wr S_n$. Proposition \ref{hiddenactionframework} then implies there is a
   $B_{n}$-equivariant homeomorphism
\begin{align*} \conf^{\Z_{2}}_{n+1}(SU_{2})/SU_{2} =: \bconflift^3 \cong_{B_{n}} \bconf^3 :=& \conf_{n}^{\langle \varphi \rangle}(\rthree)\\ =& \{ (x_1, \cdots, x_n) \in \rthree: x_i \neq x_j, \varphi(x_j) \}. \end{align*}

\item When $X = G = U(1),$ then Proposition \ref{hiddenactionframework} says that 
\begin{align*} \conf^{\Z_2}_{n+1}(U(1))/U(1) =:\bconflift^1 \cong_{B_{n}} \bconf^1 :=& \conf_{n}^{\langle \varphi \rangle}(\R \setminus \{ 0 \}) \\ =& \{ (x_1, \cdots, x_n) \in \R \setminus \{ 0 \}: x_i \neq x_j, \varphi(x_j) \}. \end{align*}
\end{itemize}
\end{example}

The space $\bconf^{d}$ arises naturally in the work of Feichtner--Ziegler in \cite{feichtner2002orbit} as part of the fiber sequence
\[ \bconf^{d} \longrightarrow \conf_{n+1}^{\Z_{2}}(\SSS^{d}) \longrightarrow \SSS^{d}, \]
where the last map sends a point $(x_1, \cdots, x_{n+1})$ to $x_{n+1}$.
They then use the above sequence to compute the cohomology of $\conf_{n+1}^{\Z_{2}}(\SSS^{d})$. In this sense, $\bconf^d$ is the natural analog of $\conf_{n}(\R^{d})$, which arises in the fiber sequence 
\[ \conf_{n}(\R^{d}) \longrightarrow \conf_{n+1}(\SSS^{d}) \longrightarrow \SSS^{d}. \]

\begin{remark} \rm
A priori, one might expect that to parallel the Type $A$ story, the analog of $\conf_n(\R^d)$ would be the more ``standard'' $\Z_{2}$-orbit configuration space 
\[ \conf_{n}^{\Z_{2}}(\R^{d}) = \{ (x_{1}, \cdots, x_{n}) \in \R^{dn}: x_{i} \neq \pm x_{j} \neq 0 \} \]
whose cohomology was studied by Xicot{\'e}ncatl in \cite{xico}, Moseley in \cite{moseley2017orlik} and related to the Type $B$ Eulerian idempotents in \cite{brauner2020eulerian} by the author. However, this is \emph{not} the case because there is no $B_n$-equivariant map between $ \conf_{n}^{\Z_{2}}(\R^{d})$ and $\bconflift^d$.

That being said, the space $\conf_{n}^{\Z_{2}}(\R^{d})$ and its cohomology will not be completely irrelevant. In particular, we will see that
\begin{itemize}
    \item some relations in $H^{*}\conf_{n}^{\Z_{2}}(\R^{3})$ will help us recover relations in $\bconfhom^3$;
    \item the spaces $H^{*}\conf_{n}^{\Z_{2}}(\R^{d})$ and $\bconfhom^d$ are isomorphic as \emph{vector spaces} but not as $B_n$ modules; and 
    \item when $d$ is odd, both $\bconfhom^d$ and $H^{*}\conf_{n}^{\Z_{2}}(\R^d)$ have total representation $\Q[B_n]$, despite carrying different representations on each graded piece.
\end{itemize} 
\end{remark}

\begin{example} \rm
It is helpful to see a concrete example of the space $\bconf^{d}$. Consider for instance 
\[ \bconfn{2}^{1} = \left \{ (x_{1}, x_{2}) \in \R^{2} \setminus \{ 0 \}: x_{1} \neq x_{2}, x_{1} \neq \varphi(x_{2}) = -\frac{x_{2}}{|x_{2}|^2} \right \},\]
(shown on the left below), compared to to the space
\[ \conf_{2}^{\Z_{2}}(\R):=  \left \{ (x_{1}, x_{2}) \in \R^{2} \setminus \{ 0 \}: x_{1} \neq x_{2}, x_{1} \neq -x_{2} \right \}, \]
which is the complement of the reflection arrangement associated to $B_2$ (shown on the right below).
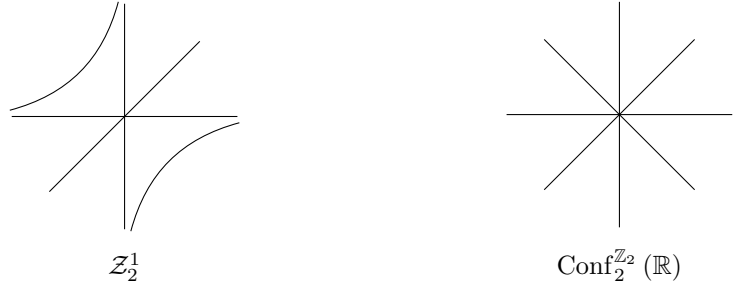
\begin{center}
\begin{tikzpicture}
 \draw (-1,-1)-- (1,1);
 \draw (0,-1.5) -- (0,1.5);
 \draw (-1.5,0) -- (1.5,0);
\node[] (B) at (-1.65,.05){};
 
% State q0
\node[] (C) at (-.05,1.65){}; 
\draw (C) to[bend left] node[right]{}(B);
 
 \node[] (D) at (1.65,-.05){};
 
% State q0
\node[] (E) at (.05,-1.65){}; 
\draw (D) to[bend right] node[right]{}(E);
 \node[] (F) at (0,-2){$\bconfn{2}^1$}; 
 \end{tikzpicture}
 \hspace{8em}
 \begin{tikzpicture}
 \draw (-1,-1)-- (1,1);
  \draw (1,-1)-- (-1,1);

 \draw (0,-1.5) -- (0,1.5);
 \draw (-1.5,0) -- (1.5,0);
\node[] (B) at (-1.65,.05){};
 
% State q0
\node[] (C) at (-.05,1.65){}; 
%\draw (C) to[bend left] node[right]{}(B);
 
 \node[] (D) at (1.65,-.05){};
 
% State q0
\node[] (E) at (.05,-1.65){}; 
\node[] (F) at (0,-2){$\conf^{\Z_2}_{2}{(\R)}$}; 
 \end{tikzpicture}
 \end{center}
Importantly, $\bconfn{2}^{1}$ is not the complement of linear subspaces (and the same is true of $\bconf^1$ in general). 
\end{example}
\begin{remark} \rm
In both Examples \ref{ex:typeahiddenaction} and \ref{ex:typebhiddenaction}, the existence of a hidden action at the level of topological spaces depends upon the sphere $\SSS^{d}$ having a realization as a Lie group. There is no such description for $\SSS^d$ in the case that $d$ is even. In \cite{gaiffi1996actions}, Giaffi obtains a hidden action of $S_{n+1}$ on the space $H^{*}\conf_{n}(\R^{2}),$ but he shows that this action is on the level of cohomology rather than topological spaces. It would be interesting to investigate whether the same is true in Type $B$---that is, determine whether there is some hidden action on $\bconfhomn{n}^{2}$ but not on $\bconfn{n}^{2}$.
\end{remark}
\subsection{The hyperoctahedral group}\label{sec:hyperoctahedralgroupbackground}
We have seen that $B_n$ is isomorphic to the wreath product $\Z_2 \wr S_n$. Now, we will view $B_n$ combinatorially as the group of signed permutations, and more abstractly as a Coxeter group. In what follows, we will discuss conventions for both perspectives, and then discuss the representation theory of $B_n$. 

\subsubsection{$B_n$ as the group of signed permutations}
Recall that $[n] = \{ 1, 2, \cdots, n\}$, and define the following sets:
\begin{itemize}
\item $[n]^{-} := \{ \overline{1}, \cdots, \overline{n} \}$, and
\item $[n]^{\pm} := \{ 1, \overline{1}, \cdots, n, \overline{n}\}$.
\end{itemize}
We adopt the convention that $\overline{i}$ behaves like $-i$, so that 
$\overline{\overline{i}} = i$ and $|\overline{i}|$ (e.g. the absolute value of $\overline{i}$) is $i$. We will first think of $B_n$ as a subgroup of the group of permutations of $[n]^{\pm}$:
\begin{definition}\rm
The \emph{hyperoctahedral group $B_n$} is the group of bijective maps from $[n]^{\pm}$ to $[n]^{\pm}$, subject to the condition that for any $\sigma \in B_n$, if $\sigma: i \mapsto j$, then $\sigma: \overline{i} \mapsto \overline{j}$ for $i,j \in [n]^{\pm}$.
\end{definition}

A permutation in $S_n$ can be written as a permutation matrix, in one-line notation, or in cycle notation. The same is true of signed permutations. That is, $\sigma \in B_n$ can be written 
\begin{itemize}
\item as a \emph{signed permutation matrix}; in this case, we realize $\sigma$ as an $n \times n$ matrix $M(\sigma)$ with $i,j$-th entry 
\[ M(\sigma)_{ij} = \begin{cases} 1 & \sigma: j \mapsto i\\
-1 & \sigma: j \mapsto \overline{i}\\
0 & \textrm{ otherwise;}
\end{cases} \]
\item in \emph{one-line notation}, meaning that $\sigma$ is written as $(\sigma_1, \cdots, \sigma_n)$ where $i \mapsto \sigma_i$ for $\sigma_i \in [n]^{\pm}$;  
\item in \emph{cycle notation}, meaning that $\sigma$ is factored as a product of positive and negative cycles. A negative cycle, written as $(\sigma_{1}\sigma_{2} \cdots \sigma_{\ell})^{-}$ for $\sigma_{i} \in [n]^{\pm}$ means that $\sigma$ sends $\sigma_{i} \mapsto \sigma_{i+1}$ for $1 \leq i \leq \ell-1$ and $\sigma_{\ell} \mapsto \overline{\sigma_{1}}$. On the other hand, a positive cycle $(\sigma_{1}\sigma_{2} \cdots \sigma_{\ell})$ sends $\sigma_{i} \mapsto \sigma_{i+1}$ for $1 \leq i \leq \ell-1$ and $\sigma_{\ell} \mapsto \sigma_{1}$.
\end{itemize}
We will make clear from context whether we are using one-line or cycle notation. 
\begin{example}\label{ex:signedperms} \rm
Consider $\sigma \in B_7$ sending 
\[ 
    1  \mapsto 2\hspace{2em}
    2  \mapsto \overline{3} \hspace{2em}
    3  \mapsto 1 \hspace{2em}
    4 \mapsto 7 \hspace{2em}
    5 \mapsto \overline{6} \hspace{2em}
    6 \mapsto 5 \hspace{2em}
    7  \mapsto 4. \]
Then $\sigma$ can be written
\begin{itemize}
\item as the matrix 
\[ \begin{pmatrix} 0 & 0 & 1 & 0 & 0& 0 & 0 \\
1 & 0 & 0 &  0 & 0& 0 & 0\\
0 & -1 & 0 & 0 & 0& 0 & 0\\
0 & 0 & 0 & 0 & 0& 0 & 1\\
0 & 0 & 0 & 0 & 0& 1 & 0\\
0 & 0 & 0 & 0 & -1& 0 & 0\\
0 & 0 & 0 & 1 & 0& 0 & 0
\end{pmatrix}
\]
    \item in one-line notation as $(2,\overline{3},1,7,\overline{6},5,4)$ and 
    \item in cycle notation as 
    \[ (12\overline{3})^{-} \,\, (47) \,\, (56)^{-}. \]
\end{itemize}
\end{example}

We will see in Section \ref{sec:MRalg} (for instance, Definition \ref{def:MRdes}) that one-line notation is useful for defining combinatorial statistics on $B_n$.

The primary use of cycle notation is to describe the conjugacy classes of $B_n$ via \emph{cycle type}. Signed permutations have cycle type described by \emph{signed partitions}, which are ordered pairs of partitions $\sgnpart$ where $|\lambda^{+}| + |\lambda^{-}| = n$. Note that we allow $\lambda^+$ or $\lambda^{-}$ to be $\emptyset$, the unique (empty) partition of zero.

 \begin{definition}\label{ex:conjclasses} \rm
 The signed permutation $\sigma \in B_n$ is said to have \emph{cycle type} 
 \[ \sgnpart = (\underbrace{\lambda_1, \cdots, \lambda_k}_{\lambda^{+}}, \underbrace{\lambda_{k+1}, \cdots, \lambda_\ell)}_{\lambda^{-}}  \] 
 if 
 $\sigma$ can be factored into cycles  
 \[ \sigma = c_1 \cdots c_k \cdot d_{1} \cdots d_\ell, \]
  where the $c_i$ are positive cycles of length $\lambda_i$ for $1 \leq i \leq k$ and the $d_j$ are negative cycles of length $\lambda_j$ for $k+1 \leq j \leq \ell$.

Define $\cc_{\sgnpart}$ to be the conjugacy class comprised of elements in $B_n$ of cycle type $\sgnpart$.
 \end{definition}
 For instance, the permutation $(12\overline{3})^{-}(47)(56)^{-}$ from Example~\ref{ex:signedperms} has cycle type $((2),(3,2))$. 
\subsubsection{$B_n$ as a Coxeter group}
Importantly, $B_n$ is also a \emph{Coxeter group}, meaning that it has a presentation of the form
\[ W = \langle r_1, \cdots, r_k: (r_i r_j)^{m_{ij}} = 1 \rangle,  \]
where $m_{ii}= 1$ and $m_{ij} \geq 2$ for $i \neq j$. The elements $r_1, \cdots, r_k$ are called \emph{simple reflections}. We will only briefly survey Coxeter groups here; for a comprehensive treatment, the reader should consult Bj\"orner--Brenti \cite{bjorner2006combinatorics}, Humphreys \cite{humphreys1990reflection}, or Kane \cite{kane2001reflection}.

When $W = B_n$, the simple reflections are (in cycle notation) the adjacent transpositions
\[ s_i := (i,i+1) \]
for $1 \leq i \leq n-1$ and the element
\[ t_n := (n)^{-}. \] 
More generally, it will be useful to think about the element $t_i = (i)^{-}$ for $1 \leq i \leq n$.

Any element of $B_n$ can be written as a product of the $s_i$ and $t_n$ (with repeats allowed), and $\sigma$ is said to be a \emph{reduced expression} if there is no way to rewrite $\sigma$ using fewer generators. In this case, the number of generators used to write $\sigma$ is the \emph{length} of $\sigma$. Every Coxeter group has a unique longest word, $w_0$. In the case of $B_n$, this is (in one-line notation)
\[ w_0 = (\overline{1}, \cdots, \overline{n}) = -1. \]
In cycle notation, $w_0 = (1)^{-}(2)^{-}\cdots (n)^{-} = t_1 t_2 \cdots t_n.$

We will also be interested in the \emph{Coxeter elements} of $B_n$. For any Coxeter group $W$, a Coxeter element $c$ is a product of each simple reflection, without repeats:
\[ c = r_{i_1} \cdots r_{i_k}. \] 
Coxeter elements are not unique but they are all conjugate to each other. 
In the case of $B_n$, Coxeter elements are of the form (in cycle notation)
   \[(i_1 i_2 i_3  \cdots i_n )^{-}.  \]
    for distinct $i_1, \cdots, i_n \in [n]$.
Thus all Coxeter elements have cycle type $(\emptyset, (n))$, and are in the conjugacy class $\cc_{\emptyset, (n)}$. 

\subsubsection{Representation theory of the hyperoctahedral group}\label{sec:reptheorytypeB}
This section will closely follow Stembridge's write-up in \cite{branchingrules}. Recall that $\tau: B_n \to S_n$ is the surjection which forgets the signs of permutations (e.g. $\tau(s_i) = s_i$ and $\tau(t_n) = 1$). Let $S = \{ s_{1}, \cdots, s_{n-1}, t_{n}\}$ and for $\sigma \in B_{n}$ and $s \in S$, let $c_s(\sigma)$ count the number of instances of $s$ in a reduced expression for $w$.
 Then define
 \[ \delta_{s}(\sigma) = (-1)^{c_s(\sigma)}. \]

Like conjugacy classes, irreducible representations of $B_{n}$ are indexed by signed partitions $\sgnpart$. 

We define the irreducible characters $\chi^{\lambda, \mu}$ of $B_n$ as follows:
\begin{itemize}
    \item $\chi^{\lambda, \emptyset}$ is obtained from the irreducible character $\chi^{\lambda}$ of $S_n$ by pulling back along $\tau;$ in other words
    \[ \chi^{\lambda, \emptyset}(s_i) = \chi^{\lambda}(s_i), \hspace{5em} \chi^{\lambda, \emptyset}(t_n) = 1. \]
    \item $\chi^{\emptyset, \lambda}$ is obtained from the tensor product
    \[\chi^{\emptyset, \lambda} = \chi^{\lambda, \emptyset } \otimes \delta_{t_{n}}. \]
    \item $\chi^{\lambda, \mu}$ is defined by the induction product:
     \[\chi^{\lambda, \mu} = \chi^{\lambda, \emptyset} \cdot \chi^{\emptyset, \mu}:=  \ind_{B_{|\lambda|} \times B_{|\mu|}}^{B_{n}}( \chi^{\lambda, \emptyset}\times \chi^{\emptyset, \mu}). \]
\end{itemize}
The induction product further implies that one inherits branching rules for $\chi^{\lambda, \mu} \cdot \chi^{\nu, \psi}$ from the Type $A$ Littlewood-Richardson Rule by rewriting
\begin{align*} \chi^{\lambda, \mu} \cdot \chi^{\nu, \psi} &= (\chi^{\lambda, \emptyset} \cdot \chi^{\nu, \emptyset}) \cdot (\chi^{\emptyset, \mu} \cdot \chi^{\emptyset, \psi}) \\
& = \left( \sum_{\alpha}c_{\lambda,\nu}^{\alpha}\chi^{\alpha,\emptyset} \right) \cdot \left( \sum_{\alpha}c_{\mu,\psi}^{\beta}\chi^{\emptyset,\beta}  \right)\\
&= \sum_{\alpha, \beta} c_{\lambda,\nu}^{\alpha}c_{\mu,\psi}^{\beta} \chi^{\alpha, \beta},
\end{align*}
where $c_{\lambda,\nu}^{\alpha}, c_{\mu,\psi}^{\beta}$ are the standard Littlewood-Richardson coefficients. 

\begin{example} \rm
There are four 1-dimensional representations of $B_{n}$, shown in Table \ref{tab:1dimreps} below. 
\begin{table}[!h]
\centering
\setlength{\tabcolsep}{10pt} % Default value: 6pt
\renewcommand{\arraystretch}{1.25} % Default
\begin{tabular}{|l||l|}
\hline
Representation & Character of $\sigma \in B_n$ \\ \hline  \hline
 $\chi^{(n), \emptyset}$ & 1\\ \hline
  $\chi^{\emptyset, (n)}$ & $\delta_{t_{n}}(\sigma)$\\ \hline
  $\chi^{(1^{n}), \emptyset}$ & $\delta_{s_{i}}(\sigma)$\\ \hline
  $\chi^{\emptyset, (1^{n})}$ & $\delta_{t_{n}} \otimes \delta_{s_{i}}(\sigma)$. \\ \hline
\end{tabular}
\caption{The 1-dimensional representations of $B_{n}$}\label{tab:1dimreps}
\end{table} 
Note that $\chi^{(1^{n}), \emptyset}$ has character $\delta_{s_{i}}$ for any $i \in [n-1]$ because all such $s_{i}$ are conjugate. The representation $\chi^{(n), \emptyset}$ is the trivial representation and $\chi^{\emptyset, (1^{n})}$ is the sign representation.
\end{example}

\begin{example} \rm
Consider the character tables for $B_1$ (Table \ref{charactertableb1}) and $B_2$ (Table \ref{charactertableb2}), which will be useful in the proof of Proposition~\ref{reptheory:n=2}.
\begin{table}[!h]
\centering
\setlength{\tabcolsep}{10pt} % Default value: 6pt
\renewcommand{\arraystretch}{1.25} % Default
\begin{tabular}{|l||l|l|}
\hline 
 & $\cc_{(1), \emptyset} =  \{ 1 \}$& $\cc_{\emptyset, (1)} = \{ t_{1} \}$ \\ \hline \hline
$\chi^{(1),\emptyset}$& 1 & 1 \\ \hline
$\chi^{\emptyset, (1)}$ &1 & -1 \\ \hline 
\end{tabular}
\caption{The character table for $B_{1}$.}\label{charactertableb1}
\end{table}

\begin{table}[!h]
\centering
\setlength{\tabcolsep}{10pt} % Default value: 6pt
\renewcommand{\arraystretch}{1.25} % Default
\begin{tabular}{|l||l|l|l|l|l|}
\hline 
 & $\cc_{(1,1),\emptyset} $ & $ \cc_{(2),\emptyset} $ & $\cc_{(1),(1)} $ & $\cc_{\emptyset,(2)} $ & $ \cc_{\emptyset, (1,1)} $ \\ 
  & $= \{ 1 \} $ & $  = \{ s_1, t_2 s_1 t_2 \}$ & $ = \{ t_2, s_1 t_2 s_1 \}$ & $ = \{ s_1t_2, t_2s_1 \}$ & $ = \{ (s_1t_2)^2 = w_0 \}$ \\ \hline \hline
% & $\c_{(1,1),\emptyset} = \{ 1 \} $ & $ \c_{(2),\emptyset} =  \{ s_1, t_2 s_1 t_2 \}$ & $\c_{(1),(1)} =  \{ t_2, s_1 t_2 s_1 \}$ & $\c_{\emptyset,(2)} =  \{ s_1t_2, t_2s_1 \} $ & $ \c_{\emptyset, (1,1)} =\{ (s_1t_2)^2 \} $ \\ \hline 

$\chi^{(2),\emptyset}$& 1 & 1 & 1 & 1 & 1  \\ \hline

$\chi^{\emptyset, (1,1)}$ &1 & -1 & -1 & 1 & 1  \\ \hline

$\chi^{(1,1),\emptyset}$  & 1 & -1 & 1 & -1 & 1 \\ \hline

$\chi^{\emptyset, (2)}$  & 1 & 1 & -1 & -1 & 1 \\ \hline

$\chi^{(1),(1)}$  & 2 & 0 & 0 & 0 & -2 \\ \hline
\end{tabular}
\caption{The character table for $B_{2}$.}\label{charactertableb2}
\end{table}
\end{example}

\section{The $d=1$ case in Type $B$}\label{sec:d1}
Our first goal is to determine a presentation for the cohomology of $\bconf^1$ and $\bconflift^1$ with coefficients in $\Z$; later in Section \ref{sec:d3} we will change to having coefficients in $\Q$ (and indicate as much). In contrast to Type $A$, we will begin with $\bconflift^1$.
\subsection{Signed cyclic Heaviside functions and $\bconflift^1$}\label{sec:signedheaviside}
The space $\Y_{n+1}^{1}$ is comprised of $2^{n}n!$ contractible pieces, parametrized by arrangements of $p_{0},\cdots, p_{n}$ \emph{and their antipodal points} $-p_{0}, \dots, -p_{n}$ on $U(1)$. Given a point $\vec{p} = (p_{0}, \cdots, p_{n}) \in \Y_{n+1}^{1}$, write $C(\vec{p}) = C(p_{0}, \cdots, p_{n})$ as the arrangement of $\vec{p}$ \emph{with} antipodes on $U(1)$ and $-p_{i}$ as $p_{\overline{i}}$. By convention $\overline{\overline{i}} = i$. 
Recall that 
\[ [n] = \{ 1, \cdots, n \}, \hspace{2em} [n]_{0} = \{ 0, 1, 2, \cdots, n\}, \hspace{2em}[n]^{-} = \{ \overline{1}, \cdots, \overline{n} \}, \hspace{2em} [n]^{\pm} = [n] \cup [n]^{-}. \]
Further define
\[ [n]_{0}^{-} := \{ \overline{0}, \overline{1}, \cdots, \overline{n} \}, \hspace{2em} [n]^{\pm}_{0} := \{ 0, \overline{0}, 1, \overline{1}, \cdots, n, \overline{n}\}.\]

When considering $C(\vec{p})$, we will identify the point $p_i$ with $i$ and $-p_i = p_{\overline{i}}$ with $\overline{i}$, so that $C(\vec{p})$ has entries in $[n]^{\pm}_{0}$.
We take $B_{n+1}$ to be the group of signed permutations on $[n]^{\pm}_{0}$, so that $\sigma \in B_{n+1}$ acts on elements of $C(\vec{p})$ by sending $i \in [n]^{\pm}_{0}$ to $\sigma(i)$. We will also consider the restricted action of $B_n \leq B_{n+1}$, where $B_n$ is thought of as the subgroup of $B_{n+1}$ fixing $\overline{0}$ and $0$.
\begin{example} \rm
Let $\vec{p} = (x_0, x_1, x_2, x_3)$, where $x_j = e^{ i  \zeta_{j} }$ and $x_{\overline{j}} = e^{ i  \overline{\zeta_{j}} },$ where $\overline{\zeta_{j}} = \zeta_j + \pi$. Suppose that 
\[ 0 \leq \zeta_0 < \zeta_1 < \overline{\zeta_2} < \zeta_3 <  \overline{\zeta_0} <  \overline{\zeta_1} < \zeta_2 <  \overline{\zeta_3} < 2 \pi. \]
Then we can write $C(\vec{p}) = (0,1,\overline{2}, 3, \overline{0}, \overline{1}, 2, \overline{3})$, which is drawn below. Note that we are only interested in the relative order of the $x_j$, rather than their values.  
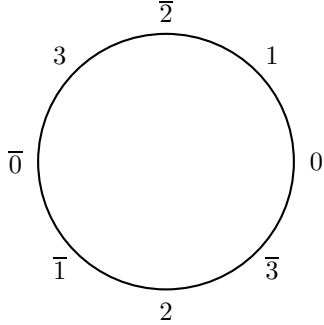
\begin{center}
\begin{tikzpicture}[thick][scale = .8]
 % State q2
  \draw[thick](0,0) circle(1.7);
% State q2
\node[] (A) at (2,0){$0$};
 
% State q1
\node[] (B) at (-2,0){$\overline{0}$};
 
% State q0
\node[] (C) at (0,2){$\overline{2}$};
\node[] (D) at (0,-2){$2$};

\node[] (E) at (1.41,1.41){$1$};
\node[] (F) at (-1.41,-1.41){$\overline{1}$};

\node[] (G) at (-1.41,1.41){$3$};
\node[] (H) at (1.41,-1.41){$\overline{3}$};
% Transition q2 to q0
%\draw (A) to[ bend right] node[left]{}(E) ;
 %\draw (E) to[bend right] node[left]{}(C) ;

% Transition q0 to q1
%\draw (C) to[bend left] node[right]{}(B);
 
% Transition q1 to q2
%\draw (B) to[bend left]node[below]{} (A);
 
\end{tikzpicture}
\end{center}
\end{example}

We define \emph{signed cyclic Heaviside functions} $y_{ijk}$ for distinct $i,j,k \in [n]^{\pm}_{0}$ as 
\[y_{ijk}(\vec{p}):=  \begin{cases} 1 & p_{i} < p_{j} < p_{k} \textrm{ counter-clockwise in }C(\vec{p}) \\
0 & \textrm{ otherwise.}
\end{cases} \]

Once again, the $y_{ijk}$ form a $\Z$-algebra with multiplication given by
\[y_{ijk}\cdot y_{qrs}(\vec{p}):=  \begin{cases} 1 & p_{i} < p_{j} < p_{k}\textrm{ and } p_{q} < p_{r} < p_{s}  \textrm{ counter-clockwise in }C(\vec{p}) \\
0 & \textrm{ otherwise.}
\end{cases} \]
The $B_{n+1}$ action on $\bconflift^1$ induces an action on the signed cyclic  Heaviside functions, where $\sigma \in B_{n+1}$ acts by
\[ \sigma \cdot y_{ijk} = y_{\sigma(i)\sigma(j)\sigma(k)}. \]

\begin{example} \rm
Consider the 8 representatives of $\Y_{3}^1$ and their evaluation via select cyclic Heaviside functions shown in Table \ref{tab:heavi}. Note that if $y_{ijk}$ is included, there is no need to include $y_{jik}$ since $y_{jik} = 1 - y_{ijk}$. Each entry in Table \ref{tab:heavi} indicates the value of a signed cyclic Heaviside function (columns) evaluated at a given representative in $\bconflift^1$ (rows). 
\begin{table}[!h]
\centering
\setlength{\tabcolsep}{10pt} % Default value: 6pt
\renewcommand{\arraystretch}{1.25} % Default
\begin{tabular}{|l||l|l|l|l|l|l|}
\hline
 $C(\vec{p})$ & $y_{0 \overline{0} 1}$ & $y_{0 \overline{0} 2}$ & $y_{012}$ & $y_{0 1 \overline{2}}$ & $y_{0\overline{1}2}$    & $y_{0\overline{1}\overline{2}}$  \\ 
 \hline \hline
% &&\\
$(0,1,2,\overline{0},\overline{1},\overline{2})$ & 0  & 0  & 1   & 1 & 0 & 1 \\ \hline
% &&\\
$(0,2,1,\overline{0}, \overline{2},\overline{1})$ & 0 & 0  & 0   & 1 & 0 & 0 \\ \hline
% &&\\ 
$(0,\overline{1},2,\overline{0},1,\overline{2})$ & 1  & 0  & 0  & 1 & 1 & 1 \\ \hline
% &&\\
$(0, 2,\overline{1},\overline{0}, \overline{2},1)$ & 1  & 0  & 0   & 0 & 0 & 1 \\ \hline
 %&&\\
$(0,1,\overline{2},\overline{0},\overline{1},2)$ & 0 & 1  & 1   & 1 & 1 & 0 \\ \hline
% &&\\
$(0,\overline{2},1,\overline{0},2,\overline{1})$ & 0  & 1  & 1   & 0 & 0 & 0 \\ \hline
% &&\\
$(0,\overline{1},\overline{2},\overline{0},1,2)$ & 1  & 1  & 1   & 0 & 1 & 1 \\ \hline
% &&\\
$(0,\overline{2},\overline{1},\overline{0},2,1)$ & 1  & 1  & 0   & 0 & 1 & 0 \\ \hline
\end{tabular}
\caption{Signed cyclic Heaviside functions evaluating the 8 representatives in $\Y_{3}^1$}\label{tab:heavi}
\end{table} 
\end{example}

\begin{prop}\label{prop:ungradedpt1}
The $B_{n+1}$ action on the set of connected components of $\bconflift^1$ 
is isomorphic to the coset action \[ B_{n+1}/\langle c \rangle, \]
where $c$ is a Coxeter element of $B_{n+1}$.
The $B_n$ action on the set of connected components of $\bconflift^1$ is simply transitive. 
\end{prop}
\begin{proof}
It is not difficult to see that the $B_{n+1}$ action on the connected components of $\bconflift^1$ is transitive. Consider a ``typical'' point $C(\vec{p}) = (0,1,2,\ldots,n,\overline{0},\overline{1},\ldots,\overline{n}) \in \bconflift^1$; it is fixed by the cyclic subgroup generated by the Coxeter element 
\[ c= (012\cdots n)^{-}. \]
This proves the first claim.

For the second, since $B_n$ fixes $0$ and $\overline{0}$, any $\sigma \in B_n$ will act on our typical point $C(\vec{p})$ as
\[ \sigma \cdot (0,1,2,\ldots,n,\overline{0},\overline{1},\ldots,\overline{n}) = (0,\sigma(1),\sigma(2),\ldots,\sigma(n),\overline{0},\sigma(\overline{1}),\ldots,\sigma(\overline{n})),  \]
and is thus fixed-point free. This $B_n$ action is also transitive, and so the second claim follows. 
\end{proof}

The next step is to give a presentation of $\bconfhomlift^1$ using the signed cyclic Heaviside functions as generators. First we determine relations that hold in $\bconfhomlift^1$. 
\begin{prop}\label{CHFrelations}
The following relations hold in $\bconfhomlift^1$ for $i,j,k,\ell \in [n]^{\pm}_{0}$:
\begin{equation*}
      (i) \hspace{.5em} y_{ijk}(1- y_{ijk}) = 0, \hspace{5em}   (ii)  \hspace{.5em} y_{ijk} =1 -y_{ikj}, \hspace{5em}   (iii)  \hspace{.5em}
\hspace{.5em} y_{\ \overline{ i  } j \, k} = y_{i \, \overline{ j \, k}},
\end{equation*}
\begin{equation*}
     (iv) \hspace{.5em} y_{ijk} - y_{ij\ell} + y_{ik \ell} - y_{jk\ell} = 0,  \hspace{5em}  (v) \hspace{.5em}  y_{ijk}y_{ik \ell }(1-y_{ij\ell}) + (1-y_{ijk})(1-y_{ik \ell })y_{ij\ell} = 0.
\end{equation*}
\end{prop}
\begin{proof}
Relations $(i)$ follows from the fact that $y_{ijk}$ and $(1-y_{ijk})$ have disjoint support. Relation $(ii)$ holds because $y_{ijk}$ and $y_{ikj}$ have disjoint support and $y_{ijk} + y_{ikj} = 1$. Relation $(iii)$ comes from noting that by definition of the antipode, $(\overline{i},j,k)$ appears in $\vec{p}$ if and only if $(i, \overline{j}, \overline{k})$ appears in $\vec{p}$. For relation $(iv)$, consider the six possible relative cyclic orderings of distinct elements $i,j,k$ and $\ell$, and their evaluations on $y_{ijk}, y_{ij\ell}, y_{ik\ell}$ and $y_{jk\ell}$, shown in Table \ref{tab:weirdrel}.
\begin{table}[!h]
\centering
\setlength{\tabcolsep}{10pt} % Default value: 6pt
\renewcommand{\arraystretch}{1.25} % Default
\begin{tabular}{|l||l|l|l|l||l|}
\hline
 & $y_{ijk}$ & $y_{ij\ell}$ & $y_{ik\ell}$ & $y_{jk\ell}$ &$y_{ijk} - y_{ij\ell} + y_{ik \ell} - y_{jk\ell}$  \\ 
 \hline \hline
% &&&&&\\
$(i, j, k, \ell)$ & 1  & 1  & 1   & 1 & 0 \\ \hline
% &&&&&\\
$(i,k,j,\ell)$ & 0  & 1  & 1   & 0 & 0 \\ \hline
% &&&&&\\
$(j,i,k, \ell)$ & 0  & 0  & 1   & 1 & 0 \\ \hline
% &&&&&\\
$(j, k, i, \ell)$ & 1  & 0  & 0   & 1 & 0 \\ \hline
% &&&&&\\
$(k,i,j,\ell)$ & 1  & 1  & 0   & 0 & 0  \\ \hline
% &&&&&\\
$(k,j,i,\ell)$ &0  & 0  & 0   & 0 & 0  \\
 \hline
\end{tabular}
\caption{Relation $(iv)$ on all possible cyclic orderings of elements $i,j,k,\ell$}\label{tab:weirdrel}
\end{table} 

 Relation $(v)$ follows because $y_{ijk}$, $y_{ik\ell }$ and $(1- y_{ij \ell})$ have disjoint support, implying that $(1-y_{ijk}),$ $(1-y_{ik \ell})$ and $y_{ij\ell}$ do as well.
\end{proof}
Note that relations $(i), (ii)$ and $(v)$ deal only with the two of the three coordinates in the $y_{ijk}$ (and will therefore be useful in the restriction to $\bconfhom^1$), whereas relations $(iii)$ and $(iv)$ use all three coordinates, and so are only useful to $\bconfhom^1$ in specific cases, discussed in Remark \ref{rm:specialcases}. 

\begin{remark}\label{rm:specialcases} \rm
Relation $(iv)$ has several noteworthy special cases that will be important in computing the presentation for $\bconfhomlift^1$ (Theorem \ref{thm:lift1pres}) and $\bconfhom^1$ (Theorem \ref{cor:bconfhom1}) in \S \ref{sec:bconf1}. 
%$\bconfhomlift^1$ in Theorem \ref{thm:lift1pres}

\begin{enumerate}
\item Taking $i =0$ gives a method of rewriting any $y_{jk\ell}$ as
\[ y_{jk\ell} = y_{0jk} - y_{0j\ell} + y_{0 k \ell}. \]
This allows us to write any signed cyclic Heaviside function in terms of $y_{ijk}$ with $i=0$, which we will use in the proof of Theorem \ref{thm:lift1pres} and when restricting from $\bconfhomlift^1$ to $\bconfhom^1$.
    \item Taking $j = 0, k = \overline{i}$ and $\ell = \overline{0}$ gives 
    \begin{align*} 0 &= y_{0 \overline{i}i} - y_{\overline{0}i 0} + y_{\overline{0}i\overline{i}} - y_{\overline{0}0\overline{i}} \\
    &=  y_{0 \overline{i}i} - y_{0\overline{0}i} + y_{0\overline{i}i} - y_{0\overline{0}i}\\
    & = 2 (y_{0 \overline{i} i} - y_{0 \overline{0}i}),\end{align*}
    where the second line follows by using Relation $(iii)$.
    This relation will prove to be a useful reduction in the proof of Theorem~\ref{thm:lift1pres}.
        \item  Taking $j = \overline{i}$ gives
    \begin{align*}
        0 &= y_{i\overline{i}k} - y_{i\overline{i}\ell} + y_{ik \ell} - y_{\overline{i}k\ell}  \\
        & = y_{i\overline{i}k} - y_{i\overline{i}\ell} + y_{ik \ell} - y_{i\overline{k}\overline{\ell}}. 
    \end{align*}
    This case is most relevant when $i = 0$, in which case
    \[ 0 = y_{0 \overline{0}k} - y_{0\overline{0}\ell} + y_{0k\ell} - y_{0 \overline{k}\overline{\ell}}.\]
    Similarly, using relations $(iii)$ and $(ii)$,
    \begin{align*}
        0 &= y_{0 \overline{0}k} - y_{0\overline{0}\overline{\ell}} + y_{0k\overline{\ell}} - y_{0 \overline{k}\ell} \\
        &= y_{0 \overline{0}k} - y_{\overline{0}0\ell} + y_{0k\overline{\ell}} - y_{0 \overline{k}\ell}\\
        &= y_{0 \overline{0}k} - (1-y_{0\overline{0}\ell}) + y_{0k\overline{\ell}} - y_{0 \overline{k}\ell}.\\
    \end{align*}
    In addition to appearing as a reduction in the proof of Theorem \ref{thm:lift1pres}, these relations will be instrumental in computing a presentation for $\bconfhom^1$ in Theorem~\ref{cor:bconfhom1} and will appear again in \S \ref{sec:d3} (Proposition~\ref{prop:actionongenerators}). \end{enumerate}
\end{remark}

Next, we show that the relations in Proposition \ref{CHFrelations} are sufficient to describe $\bconfhomlift^1$. In order to do so, we will employ a Gr\"obner basis lemma used by Dorpalen-Barry in \cite{dorpalen2021varchenko}.
For vectors
\begin{align*}
{\bf{(i,j,k)}} &:= (i_{1},j_{1},k_{1}),(i_{2},j_{2},k_{2}), \cdots, (i_{\ell},j_{\ell},k_{\ell}),\\
{\bf{a}} &:= (a_{1}, \cdots, a_{\ell})
\end{align*}
write
\[ {\bf{y}_{(i,j,k)}^{a}} :=  y^{a_{1}}_{i_{1}j_{1}k_{1}}y^{a_{2}}_{i_{2}j_{2}k_{2}} \cdots y^{a_{\ell}}_{i_{\ell}j_{\ell}k_{\ell}}. \]

Then given a polynomial 
\[ f = \sum c_{a}  {\bf{y}_{(i,j,k)}^{a}}, \]
with $c_a \in \Z$, the \emph{degree} of $f$ is 
\[ \deg(f):= \max \{ \sum_{i} a_{i}: c_{a} \neq 0 \} \]
and the \emph{degree-initial form} of $f$ is 
\[ \indeg(f):= \sum_{\substack{a \\ \sum_{i} a_{i} = \deg(f) }} c_{a} {\bf{y}_{(i,j,k)}^{a}}.  \]
In other words, $\indeg(f)$ picks off the top degree terms of $f$. Given a polynomial ideal $\mathcal{Q}$,  let $\indeg(\mathcal{Q})$ be the ideal generated by $\indeg(f)$ for $f \in \mathcal{Q}$.

A similar notion exists for any total ordering $\prec$ on monomials provided that $\prec$ is a \emph{well-ordering}, meaning that if $f \prec f',$ then for any monomial $g$, one has $fg \prec f'g$. Suppose $\prec$ is a well-ordering on monomials in $\Z[y_{ijk}]$ for $i,j,k \in [n]^{\pm}_{0}$. Then given a polynomial 
\[ f = \sum c_{a} {\bf{y}_{(i,j,k)}^{a}} \in \Z[y_{ijk}], \]
define $\inn_{\prec}(f)$ to be the $\prec$-leading \emph{monomial} in $f$. The ordering $\prec$ is said to be a \emph{degree order} if it is compatible with the natural degree ordering on $f$ in the sense that  $\inn_{\prec}(f) = \inn_{\prec}(\indeg(f))$.

Our goal is to define a degree order $\prec$ on the set of polynomials $\Z[y_{ijk}]$ for any distinct $i,j,k \in [n]^{\pm}_{0}$. To do so, first order $[n]^{\pm}_{0}$:   

 \begin{equation}\label{eq:order}
\{ 0 < \overline{0} < 1 < \overline{1} < \cdots < n < \overline{n}. \}     \end{equation} 

Given $y_{ijk}$ with $i,j,k \in [n]^{\pm}_{0}$, order the indices so that $i < j, k$ in \eqref{eq:order}; this is always possible because the $y_{ijk}$ are equivalent up to cyclic rotation. From there,  define $\prec$ by lexicographically ordering the indices in $y_{ijk}$ using the total order in \eqref{eq:order}. One can check that $\prec$ does indeed define a well-ordering on $\Z[y_{ijk}]$ for $i,j,k \in [n]^{\pm}_{0}$. We will use the ordering $\prec$ to argue that the relations in Proposition \ref{CHFrelations} are the only ones needed to generate $\bconfhomlift^1$. 

Given a set of polynomials $\I= \{ f_{i} \}_{i \in I}$ indexed by a set $I$, the polynomials which are \emph{not} divisible by the terms $\inn_{\prec}(f_{i})$ for each $f_{i}$ are called the \emph{ $\inn_{\prec}(\mathcal{I})$-standard monomials}. For example, if $\mathcal{I} = \{ y_{ijk}^{2} - y_{ijk} \}$, then 
\[ \inn_{\prec}( y_{ijk}^{2} - y_{ijk}) = y_{ijk}^{2}, \]
and so in this case $y_{ijk}$ is an $\inn_{\prec}(\mathcal{I})$-standard monomial. To foreshadow, when $\mathcal{I}$ is the set of relations in Proposition \ref{CHFrelations}, the set of $\inn_{\prec}(\mathcal{I})$-standard monomials will be a $\Z$-basis for $\bconflift^1$. 

Finally, recall that an \emph{ascending filtration} of a commutative ring $R$ is a nested sequence of $\Z$-submodules $F_{0} \subseteq F_{1} \subseteq \cdots$ such that for $f \in F_{i}$ and $f' \in F_{j}$, one has $f\cdot f' \in F_{i+j}$. The \emph{associated graded ring} of $R$ with respect to this filtration is then defined to be 
\[ \gr(R):= \bigoplus_{i \geq 0} F_{i}/F_{i-1}. \]
In our case, the ring we are interested in is $\bconfhomlift^1$ and the filtration is the natural filtration by degree, so that $F_{d}$ consists of polynomials of degree $d$ or less.

From these definitions we may now state the relevant lemma. Let $S = \Z[e_{1}, \cdots, e_{r}],$ where the $e_{i}$ are standard basis vectors. 

\begin{lemma}[Dorpalen-Barry, \cite{dorpalen2021varchenko} Lemma 8]\label{lemma:GDB}
Let $R$ be a free $\Z$-module of rank $r$. Given a surjection $\gamma: S \twoheadrightarrow R$, let $\mathcal{I} = \{ f_{i} \}_{i \in I} \subset S$ be a set of polynomials such that 
\begin{enumerate}
    \item Each $f_{i}$ is \emph{monic} (e.g. the coefficient of $\indeg(f_{i})$ is $\pm 1$),
    \item $\mathcal{I} \subset \ker(\gamma)$, and 
    \item the set of $\inn_{\prec}(\mathcal{I})$-standard monomials $\mathcal{N} = \{ m_{1}, \cdots, m_{t} \}$ has cardinality $t \leq r$.
\end{enumerate}
Then \begin{enumerate}
    \item $R \cong S/(\mathcal{I})$ as $\Z$-modules, where $(\mathcal{I})$ is the ideal generated by $\mathcal{I}$ and
    \item the cardinality of $\mathcal{N}$ is $r$, so that $\gamma(m_{1}), \cdots, \gamma(m_{r})$ is a $\Z$-basis for $R$.
\end{enumerate}
In the case that $\prec$ is a degree-ordering, then there is a further $\Z$-module isomorphism
\[ S / \indeg(\mathcal{I}) \cong \gr (R). \]
%where $(\indeg(\mathcal{I}))$ is the ideal generated by $\indeg(\mathcal{I})$. 
The Hilbert series for $\gr(R)$ is then given by 
\[ \hilb( \gr(R),t) = \sum_{i=1}^{r} t^{\deg(m_{i})}. \]
\end{lemma}

We will see that Lemma \ref{lemma:GDB} is precisely what we need to give a presentation for $\bconfhomlift^1$. 

Define the sets \[\underbrace{\{ y_{0\overline{0}1} \}}_{h_{1}}, \underbrace{\{ y_{0 \overline{0}2}, y_{012}, y_{01\overline{2}} \}}_{h_{2}}, \cdots, \underbrace{\{ y_{0 \overline{0}n}, y_{01n}, y_{0 1\overline{n}}, \cdots, y_{0(n-1)n}, y_{0 (n-1)\overline{n}} \}}_{h_{n}}, \] 
and let $\mathcal{N}$ be the set of monomials obtained by multiplying at most one term in each $h_{i}$. Intuitively, think of each $h_{i}$ as a hand and the elements in $h_{i}$ as its fingers; then $\mathcal{N}$ is the set of monomials obtained by picking at most one finger from each hand\footnote{The terminology of ``hands'' and ``fingers'' originates in H\'el\`ene Barcelo's thesis \cite[Thm 2.1]{barcelothesis} and was later used in Barcelo--Goupil \cite{barcelo1999lattices}, both in the context of describing an \emph{nbc-basis}. Such bases arise in the study of matroids. While we are not in the matroid (e.g. hyperplane) setting, because our description of $\mathcal{N}$ uses the hand/finger description, we may refer to it basis as an nbc-basis nonetheless.}.

\begin{theorem}\label{thm:lift1pres} Let $\mathcal{I}$ be the ideal generated by the relations in Proposition \ref{CHFrelations}. Then
\[ \bconfhomlift^1  = \Z \left[ y_{ijk}\right]/ \mathcal{I} \]
where $i,j,k$ are distinct elements in $[n]^{\pm}_{0}$. Further, the set $\mathcal{N}$ is a basis for $\bconfhomlift^1$.
\end{theorem}

\begin{proof}
Suppose $m$ is an $\inn_{\prec}(\mathcal{I})$-standard monomial. We make the following reductions:
\begin{itemize}
    \item By Proposition \ref{CHFrelations} $(i)$, we may assume $m$ is square free;
    \item Using Proposition \ref{CHFrelations} $(iv)$ and setting $i = 0$, we may also assume that $m$ is comprised of generators of the form $y_{0jk}$;
    \item Furthermore, using Proposition \ref{CHFrelations}  $(iv)$ and Remark \ref{rm:specialcases} (2) we have that $m$ cannot contain $y_{0\overline{i}i}$ for $i \in [n]$;
    \item  Using  \ref{CHFrelations} $(ii)$, we may assume further for $y_{0jk}$ that $j < k$ with respect to the ordering in \eqref{eq:order};
    \item By Remark \ref{rm:specialcases} (3) we have that $m$ cannot contain any generator of the form $y_{0\overline{j}k}$ or $y_{0 \overline{j}\overline{k}}$ for $j,k \in [n]$.
    \end{itemize}
Finally, we must use relation $(v)$ to show that $m$ is in fact in $\mathcal{N}$. Note that $(v)$ can be understood as a choice of three elements $i < j < k \in [n]^{\pm}_{0} \setminus \{ 0 \}$, corresponding to the three monomials $y_{0ij}, y_{0jk}$ and $y_{0ik}$. Applying relation \ref{CHFrelations} $(v)$ to these generators implies that the term $y_{0ik}y_{0jk}$ cannot divide $m$. We use this logic to underline terms that are not divisible by $m$ for $i,j,k \in [n]$:
\begin{itemize}
    \item $(i < j < k):$
    \[ 0 = y_{0ij}y_{jk} - y_{0ij}y_{0ik} - \underline{y_{0ik}y_{jk}} + y_{0ik}. \]
        \item $(i < j < \overline{k}):$
    \[ 0 = y_{0ij}y_{j\overline{k}} - y_{0ij}y_{0i\overline{k}} - \underline{y_{0i\overline{k}}y_{j\overline{k}}} + y_{0i\overline{k}}. \]
    \item ($\overline{0} < i < j$): %$\yj y_{0ij}$, which can be rewritten in terms of $\yi$ and $y_{0ij}$;
    \[  0 = \yi y_{0ij} - \yi \yj - \underline{y_{0ij} \yj} + \yj.\]
    
    \item ($\overline{0} < i < \overline{j}$): Noting that $y_{0\overline{0}\overline{j}} = y_{\overline{0}0j} = (1-\yj)$ by  \ref{CHFrelations} $(iii)$ and  \ref{CHFrelations} $(ii)$:
    \begin{align*}
        0 &= \yi \yijminus - \yi y_{0\overline{0}\overline{j}} - \yijminus y_{0\overline{0}\overline{j}} - y_{0\overline{0}\overline{j}}\\
        &= \yi \yijminus - \yi (1-\yj) - \underline{\yijminus} (1-\underline{\yj}) - (1-\yj).
    \end{align*}
   %     \[  0 = \yi \yijminus + \yi \yj + \underline{\yijminus \yj} - \yj\]
  %  $  \yijminus,$ which can be rewritten in terms of $\yi$ and $\yijminus$;
    \item $(i < j < \overline{j})$: Noting that $y_{0j\overline{j}} = 1-y_{0 \overline{j}j} = 1- \yj$ by  \ref{CHFrelations} $(iii)$ and  \ref{CHFrelations} $(ii)$:
    \begin{align*}
        0 &= y_{0ij}y_{0j\overline{j}} - y_{ij} \yijminus - y_{0j\overline{j}} \yijminus \\
        &=  y_{0ij}(1-\yj) - \underline{y_{ij} \yijminus} - (1-\yj) \yijminus.
    \end{align*}
%    \[ 0 = -y_{0ij}\yj - \underline{ y_{0ij}\yijminus} + \yj \yijminus - \yj \]
    \item $(i < \overline{j} < k):$ By  \ref{CHFrelations}  $(iv)$, Remark \ref{rm:specialcases} (3) one has
    $\yjkbad = \yj + \yk + \yjkminus - 1$, and so
    \begin{align*}
   0 &= \yijminus y_{\overline{j}k} -  y_{\overline{j}k} y_{ik} - \yijminus y_{ik} + y_{ik}\\
   &= \yijminus (\yj + \yk + \yjkminus - 1) -  (\yj + \yk + \underline{\yjkminus} - 1) \underline{y_{ik}} - \yijminus y_{ik} + y_{ik}.
   \end{align*} 
    \item $(\overline{i} < j < k):$ Again, we have 
    $\yijbad = \yi + \yj + \yijminus - 1,$ and $\yikbad = \yi + \yk + \yikminus - 1$. Hence
    \begin{align*}
        0 &= \yijbad y_{0jk} - \yijbad \yikbad - y_{0jk}\yikbad + \yikbad\\
        &= (\yi + \yj + \yijminus - 1)y_{0jk} - (\yi + \yj + \yijminus - 1) (\yi + \yk + \yikminus - 1)\\ 
        &- \underline{y_{0jk}}(\yi + \yk + \underline{\yikminus} - 1) + (\yi + \yk + \yikminus - 1).
    \end{align*}
\end{itemize}

It follows that $m \in \mathcal{N}$, because all of the terms violating the definition of $\mathcal{N}$ (e.g. picking multiple generators from the same set $h_i$) are in blue above. Thus the hypotheses in Lemma \ref{lemma:GDB} are satisfied, and the claim follows. 
\end{proof}

Note that $\prec$ is also a degree ordering, which by Lemma \ref{lemma:GDB} implies the following. 
\begin{cor}\label{cor:lifthom1gr}
The associated graded ring of $\bconfhomlift^1$ with respect to the filtration by degree has presentation 

\[ \gr (\bconfhomlift^1) \cong \Z[y_{ijk}]/ \indeg ( \I),\]
for distinct $i,j,k \in [n]_{0}^{\pm}$, where $\indeg ( \I)$ is generated by the relations
\[ (i) \hspace{.5em} y_{ijk}^{2} , \hspace{2em}  (ii) \hspace{.5em} y_{ijk} - y_{ij\ell} + y_{ik \ell} - y_{jk\ell}  \hspace{2em}   (iii) \hspace{.5em} y_{ \overline{i} \, j \, k} = y_{i \, \overline{j \, k}} \hspace{2em} \]
\[ (iv) \hspace{.5em} y_{ijk} = -y_{ikj} \hspace{2em} (v) \hspace{.5em} y_{ijk}y_{ik \ell } - y_{ijk} y_{ij \ell } - y_{ik \ell } y_{ij \ell }.  \]

Further,
\[ \hilb(\gr(\bconfhomlift^1),t) = (1+t)(3+t)(5+t) \cdots ((2n-1)+t).\]
\end{cor}
\begin{proof}
The relations generating $\indeg(\I)$ come from computing $\indeg(f_{i})$ for each $f_{i} \in \I$ in Theorem \ref{thm:lift1pres}. Note that expanding relation \ref{CHFrelations} $(v)$ gives 
\begin{align*}
     y_{ijk}y_{ik \ell }(1-y_{ij\ell}) + (1-y_{ijk})(1-y_{ik \ell })y_{ij\ell} &= y_{ijk}y_{ik\ell} - y_{ijk}y_{ik\ell}y_{ij\ell} + y_{ij} - y_{ijk}y_{ik\ell} - y_{ik\ell} y_{ij\ell} + y_{ijk}y_{ik\ell}y_{ij\ell} \\
     &=y_{ijk}y_{ik\ell} + y_{ij} - y_{ijk}y_{ik\ell} - y_{ik\ell} y_{ij\ell}. 
\end{align*}
Removing the single degree 1 term here gives a homogeneous degree 2 relation. The Hilbert series follows from a standard counting argument using the set $\mathcal{N}$.
\end{proof}

\begin{example} \rm
Suppose $n+1 = 3$. Then a basis for $H^*\Y_{3}^1$, separated by degree in the $y_{ijk}$ is 
\begin{align*}
\deg(0):& \ \ 1 \\
\deg(1): & \ \ y_{0\overline{0}1}, \  \ \   y_{012}, \ \ \  y_{01\overline{2}}, \ \ \ y_{0 \overline{0}2},\\
\deg(2): & \ \ y_{0\overline{0}1} y_{012},  \ \ \ y_{0\overline{0}1}y_{01\overline{2}}, \ \ \   y_{0\overline{0}1}y_{0\overline{0}2}
\end{align*}
The Hilbert series for $\gr(H^*\Y_3^1)$ is 
\[ 1 + 4t + 3t^2 = (1+t)(3+t). \]
\end{example}

\subsection{A presentation for $\bconfhom^1$}\label{sec:bconf1}
The $B_n$ module isomorphism induced from Proposition \ref{hiddenactionframework} (applied as in Example \ref{ex:typebhiddenaction}) recovers a  presentation for $\bconfhom^1$ and $\gr(\bconfhom^1)$. The isomorphism can be described as follows: for each Heaviside function $y_{ijk}$, fix $i = 0$ and define for $j,k \in  [n]^{\pm}$,
\begin{align*}
    z_{jk}:=& y_{0jk}\\
    z_{j}:= & y_{0\overline{0}j}.
\end{align*}
These $z_{jk}$ and $z_j$ inherit a $B_n$ action: $\sigma \in B_n$ acts by
\begin{align*}
   \sigma \cdot z_{jk}=& z_{\sigma(j)\sigma(k)},\\
   \sigma \cdot z_{j} =& z_{\sigma(j)}.
\end{align*}
The elements $z_{jk}$ and $z_j$ are (by construction) functionals on $\bconf^1$. The latter generator has a simple description:
\[ z_j(x_1, \cdots, x_n) = \begin{cases} 1 & x_j > 0\\
0& \textrm{ otherwise.}
\end{cases}\]
The $z_{jk}$ are a bit more complicated; if $j,k \in [n]$, then 
\[ z_{jk}(x_1, \cdots, x_n) = \begin{cases} 1 & x_{j} < x_k \\
0& \textrm{otherwise,}
\end{cases}\]
and 
\[ z_{j\overline{k}}(x_1, \cdots, x_n) = \begin{cases} 1 & x_{j} < \varphi(x_k) \\
0& \textrm{otherwise,}
\end{cases}\]
where $\varphi(x_j) = -x_j/|x_j|^2$. The generator $z_{\overline{j}k}$ can analogously be interpreted as 1 when $\varphi(x_j) < x_k$, and the generator $z_{\overline{j}\overline{k}}$ is 1 when $\varphi(x_j) < \varphi(x_k)$. Note that $\varphi(x_j) < \varphi(x_k)$ is \emph{not} equivalent to $-x_j < -x_k$.

We now obtain from Theorem \ref{thm:lift1pres} a presentation for $\bconfhom^1$.
\begin{theorem}
\label{cor:bconfhom1}
There is an isomorphism of $\Z$-modules 

\[\bconfhom^1 = \Z[z_{ij},\zi] / \J \]
for $i\neq j \in [n]^{\pm}$, where $\J$ is generated by the relations:
\[ (i) \hspace{.5em} z_{ij}(1-z_{ij}) \hspace{2em}  (ii) \hspace{.5em} \zi(1-\zi)  \hspace{2em} (iii) \hspace{.5em} \zi- \zj + z_{ij} - z_{\overline{i}\overline{j}}   \hspace{2em}   (iv) \hspace{.5em} \zi - (1-z_{\overline{i}})  \] 
and
\begin{align*}
    (v) \hspace{.5em} z_{ij}z_{j k }(1-z_{ik }) &+ (1-z_{ij}) (1-z_{j k })z_{i k }\\
    (vi) \hspace{.5em} z_{ij}z_i(1-z_j) &+ (1-z_{ij})(1-z_i)z_j  \\
    (vii) \hspace{.5em} z_{j}\zijminus(1-z_{ij}) &+ (1-z_{j})(1-\zijminus)z_{ij}.
\end{align*}
\end{theorem}
The ``extra'' relations $(vi)$ and $(vii)$ compared to $\bconfhomlift^1$ are necessary because unlike the $y_{ijk}$, there are now two types of generators, $z_j$ and $z_{jk}$.
\begin{proof}
The isomorphism between $\bconfhomlift^1$ and $\bconfhom^1$ implies that $\bconfhom^1$ has a $\Z$-basis consisting of monomials that are a product of at most one element from each set
\[\{z_1 \}, \{z_2, z_{12}, \zminus \}, \cdots \{ z_n, z_{1n}, \cdots, z_{(n-1)n}, z_{1\overline{n}}, \cdots, z_{(n-1)\overline{n}} \}. \]

That Relations $(i)$---$(v)$ hold in $\bconfhom^1$ follows immediately from the identification of variables discussed above. Relation $(vi)$ can be seen directly: for $i,j \in[n]$, the function 
$ z_{ij}z_i(1-z_j) $ is nonzero on $(x_1, \ldots, x_n)$ if $x_i < x_j$, $x_i> 0$, and $x_j < 0$, which is impossible. Similarly, $(1-z_{ij})(1-z_i)z_j$ is always 0, so the sum must be 0 as well. The relation is then closed under the action by $B_n$. 

An analogous argument holds for $(vii)$; $z_{j}\zijminus(1-z_{ij})$ is non-zero on $(x_1, \cdots, x_n)$ if $x_i > x_j$, $x_i < \varphi(x_j)$ and $x_j > 0$. However if $x_j > 0,$ this means that $\varphi(x_j) < 0,$ and so $x_i< 0$, contradicting the assumption that $x_i > x_j$. Analogously, the second summand in $(vii)$ is 0.
A similar argument as the proof of Theorem \ref{thm:lift1pres} (e.g. using the ordering induced from \eqref{eq:order} and Lemma \ref{lemma:GDB}) then shows that these relations are sufficient to generate $\bconfhom^1$. 
\end{proof}

Using Theorem \ref{cor:lifthom1gr}, we obtain a presentation for $\gr(\bconfhom^1)$ as well. 
\begin{cor}\label{cor:grbconfhom}
The associated graded ring of $\bconfhom^1$ with respect to the filtration by degree is given by
\[  \gr (\bconfhom^1) = \Z[z_{ij},z_{i}]/\indeg(\J)\]
for distinct $i,j \in [n]^{\pm}$ where $\indeg(\J)$ is generated by
\[ (i) \hspace{.5em} z_{ij}^2  \hspace{2em}  (ii) \hspace{.5em} \zi^2 \hspace{2em} (iii) \hspace{.5em} \zi- \zj + z_{ij} - \zijbadbad   \hspace{2em}   (iv) \hspace{.5em} \zi + z_{\overline{i}} \] 
\[
    (v) \hspace{.5em} z_{ij}z_{jk} - z_{ij}z_{ik} - z_{jk}z_{ik}\hspace{2em} (vi) \hspace{.5em} z_i z_{ij} - z_{ij}z_j - z_{i}z_j \hspace{2em} (vii) \hspace{.5em} z_j \zijminus - \zijminus z_{ij} - z_j z_{ij}. \]
\end{cor}

\begin{example}[$n=2$] \rm
The above restriction tells us that when $n=2$, the basis for $\gr(\bconfhom^1)$ is 
\begin{align*}
\deg(0) :& \ \ 1 \\
\deg(1): & \ \ z_1, \  \ \   z_{12}, \ \ \  z_{1\overline{2}}, \ \ \ z_2,\\
\deg(2): & \ \ z_1 z_{12},  \ \ \ z_1 z_{1\overline{2}}, \ \ \  z_1 z_2,
\end{align*}
and the Hilbert series of $\gr(H^* \bconfn{2}^1)$ is the same as that of  $\gr(H^*\Y_3^1)$.
\end{example}

\subsubsection{A bi-grading on $\bconfhom^1$}
Finally, we introduce a further filtration on $\gr(\bconfhom^1)$. 

\begin{prop}\label{prop:filbyz}
Let $P_{\ell}$ be the ideal in $\gr(\bconfhom^1)$ with monomials of degree $\ell$ or higher in the $z_{i}$ for $i \in [n]^{\pm}$. Then there is a descending filtration on $\gr(\bconfhom^1)$ that is stable under the $B_n$-action on $\gr(\bconfhom^1)$:
\[ P_{n} \subset P_{n-1} \subset \cdots \subset P_{1} \subset P_{0}. \]
\end{prop}
\begin{proof}
Consider first the polynomial ring $\Z[z_{ij}, z_i]$ for $i,j \in [n]^{\pm}$; it is clear that it has a filtration by $z_i$ degree, and that this filtration respects the action by $B_n$.
One can further check that the ideal $\indeg(\J)$ in Corollary \ref{cor:grbconfhom} also respects the filtration because each relation (and its $B_n$ orbit) is in either $P_{0}$, $P_{1}$ or $P_{2}$.
\end{proof}

\begin{definition}\label{def:G} \rm
Define the associated graded ring of $\gr(\bconfhom^1)$ with respect to the filtration in Proposition \ref{prop:filbyz} to be
\[  \G = \bigoplus_{0 \leq k \leq n} \G_k = \bigoplus_{0 \leq \ell \leq k \leq n} \G_{k,\ell}, \]
where $\G_k$ consists of monomials of degree $k$ in the variables $z_{ij}$ and $z_i$ for $i,j \in [n]^{\pm}$ and $\G_{k,\ell}$ consists of monomials in $\G_k$ which are degree $\ell$ in the $z_i$ variables.
\end{definition}

\begin{cor}\label{cor:bigrading}
As a bi-graded ring, $\G$ has a presentation 
\[ \Z[ z_{ij}, z_i ]/\mathcal{L} \]
for $i,j \in [n]^{\pm}$, where $\mathcal{L}$ is generated by $z_{ij}^{2} = z_i^{2} = 0$ and 
\[ (i) \hspace{.5em} z_{ij} + z_{ji} \hspace{1.5em} (ii) \hspace{.5em}z_{i} + z_{\overline{i}} \hspace{1.5em} (iii) \hspace{.5em} \zijminus  - \zijbad \hspace{2em} (iv) \hspace{.5em}z_{ij}\zijminus \hspace{1.5em} (v) \hspace{.5em} z_{ij}z_{jk} - z_{ij}z_{ik} - z_{jk}z_{ik}. \]
\end{cor}
\begin{remark} \rm
The fact that $\Q [B_n]$ is semisimple means that for any filtration stable under the action of $B_n$, passing to the associated graded ring will not change the isomorphism type of the representation. Hence we will study representations on $\G_{k,\ell}$ and use this to deduce information about $\bconfhom^1$ and $\bconfhom^3$.
\end{remark}
\begin{remark} \rm
The filtration in Proposition \ref{prop:filbyz} (using the identification $z_i \leftrightarrow y_{0\overline{0}i}$) does \emph{not} respect the lifted $B_{n+1}$ action, and therefore does not give a bi-grading on $\bconfhomlift^1$. 
\end{remark}

\section{The $d=3$ case in Type $B$}\label{sec:d3}
We now turn to the space $\bconf^3$, which was studied by Feichtner--Ziegler in \cite{feichtner2002orbit} and Xicot{\'e}ncatl in \cite{xicotencatl2000cohomology}; we will first review their work\footnote{In \cite{feichtner2002orbit}, Feichtner--Ziegler give a presentation of $\bconfhom^d$ for $d \geq 2$. However, their computation of the action of $B_n$ on the generators of $\bconfhom^d$ has an error \cite[Lemma 7(iv)]{feichtner2002orbit}. Xicot{\'e}ncatl also gives a presentation of $\bconfhom^d$, which agrees with our presentation; however his work does not explicitly compute the action of $B_n$ on the generators of $\bconfhom^d$. We will see in Proposition \ref{prop:actionongenerators} that the $B_n$-action on $\bconfhom^d$ is delicate, and so we include all the details of our computations for completeness.} in Section \ref{sec:toolstopology}. Then, we will work to understand the action of $B_n$ on the basis for $\bconfhom^3$ (Section \ref{tools:reptheory}). The presentation for $\bconfhom^3$ and its consequences are given in Section \ref{sec:presbconfhom}.

\subsection{Tools from Topology}\label{sec:toolstopology}

In \cite{feichtner2002orbit}, Feichtner--Ziegler use the space $\bconf^d$ to compute a presentation for the cohmology of the $\Z_2$ orbit configuration space $\conf_n^{\Z_2}(\SSS^d)$. Let \begin{align*} \Pi_{[n]}: \bconfn{n+1}^d &\longrightarrow \bconf^d \\
(x_{1},\cdots,x_{n+1}) &\longmapsto (x_{1}, \cdots, x_{n})
\end{align*}
be the map which ``forgets'' the last coordinate in $\bconfn{n+1}^d$, and define 
\[ Q_{n} =  \{ 0, x_{1}, \varphi(x_{1}),\cdots x_{n-1},\varphi(x_{n-1})\}, \]
where $\varphi(x_i) = -x_i/|x_i|^2$, as before.
This induces a locally trivial fiber bundle
\begin{equation}\label{eq:spectral} \R^{d} \setminus Q_n \longrightarrow \bconfn{n+1}^d \longrightarrow \bconf^d.\end{equation} Feichtner--Ziegler prove that the associated spectral sequence collapses in its second term whenever $d > 2$, and that the corresponding cohomology ring $\bconfhom^d$ is torsion-free with Hilbert series\footnote{They further prove (with a bit more work) that the $d=2$ case is also torsion free and satisfies \eqref{eq:hilbertseries}.}
\begin{equation}\label{eq:hilbertseries}
 \hilb(\bconfhom^d, t) =   (1+t^{d-1})(1+3t^{d-1}) \cdots (1+(2n-1)t^{d-1}).\end{equation}

Using the spectral sequence induced by \eqref{eq:spectral}, Feichtner--Ziegler provide a multiplicative generating set and $\Z$-module basis for $\bconfhom^d$. 

\begin{prop}[\cite{feichtner2002orbit}: Prop 8]\label{prop:nbc}
A $\Z$-linear basis for $\bconfhom^d$ is given by choosing one element from each set:
    \[ \Z \{ 1, \zone \} \cdot \Z \{ 1, \ztwo , z_{12}, \zminus, \} \cdots \Z \{ 1, z_{n}, z_{1n}, z_{1\overline{n}}, \cdots, z_{(n-1)n}, z_{(n-1)\overline{n}} \}.  \]
\end{prop}

A number of relations that hold in $\bconfhom^3$ can be established from the above topological framework. Recall that $t_{i} \in B_{n}$ is the element sending $i \mapsto -i$ and $j \mapsto j$ for all other $j$, and $s_i \in B_n$ is the Coxeter generator $(i,i+1)$. Let $(i,j)$ be the transposition swapping $i$ and $j$.

\begin{prop}\label{prop:FZrelations}
The following identities hold in $\bconfhom^3$ for distinct $i,j,k \in [n]^{\pm}$
\begin{enumerate}
    \item {\rm (Feichtner--Ziegler \cite[Lemma 7(iii)]{feichtner2002orbit}):} $t_{i} \cdot \zi = z_{\overline{i}} = - \zi;$
    \item $(i,j) \cdot z_{ij} = z_{ji} = -z_{ij}$;
    \item $z_{ij}^{2} = \zi^2 = 0.$
    \item {\rm (Feichtner--Ziegler \cite[Prop. 11]{feichtner2002orbit}):} $z_{ij}z_{jk} - z_{ik}z_{ij} - z_{ik}z_{jk}  = 0$
    \item {\rm (Feichtner--Ziegler \cite[Prop. 11]{feichtner2002orbit}):} $z_{ij}z_{i} - z_{ij}z_{j} - z_{i}z_{j} = 0.$
    \item  $\zj \zijminus  - z_{ij} \zijminus - \zj z_{ij} = 0.$
\end{enumerate}
\end{prop}
We give a general sketch of each of these relations for intuition.
\begin{proof}
\begin{enumerate}
    \item The first claim in Proposition \ref{prop:FZrelations} follows because the generator $\zi$ is the image of the (dual) fundamental class induced from the projection map
\begin{align*}
    \pi_{i}: \bconf^d &\longrightarrow S^{d-1}\\
     (x_{1}, \cdots, x_{n}) &\longmapsto \frac{x_{i}}{|x_{i}|}.
\end{align*}
The action of $t_{i}$ then restricts to the antipodal action on $S^{d-1}$.
\item  Analogously, the second claim in Proposition \ref{prop:FZrelations} can be understood via the projection map
\begin{align*}
    \pi_{ij}: \bconf^d &\longrightarrow S^{d-1}\\
     (x_{1}, \cdots, x_{n}) &\longmapsto \frac{x_{i} - x_{j}}{|x_{i}-x_{j}|}.
\end{align*}
The generator $z_{ij}$ is similarly defined as the image of $\pi_{ij}^{*}([S^{d-1}])$ (where $[S^{d-1}]$ is the dual fundamental class of $S^{d-1}$). Once again the action of $(i,j)$ can be traced back via $\pi_{ij}$, where it restricts to the antipodal map on $S^{d-1}$.
\item Again, the third claim comes from the fact that the generators $z_{ij}$ and $\zi$ are the images of $[S^{d-1}]$, and $[S^{d-1}]^{2} = 0$ in $H^{*}(S^{d-1})$.
\item Let $U_{ij}^{+}:= \{ (x_{1}, \cdots, x_{n}) \in \R^{3n}: x_{i} = x_{j} \},$
and consider the complement space $\mathcal{M}\{ U_{ij}^{+},U_{jk}^{+} \} \subset \R^{3n}.$
There is a natural inclusion 
\[ \bconf^{3}\hookrightarrow \mathcal{M}\{ U_{ij}^{+},U_{jk}^{+} \} ;\]
the induced map in cohomology must send the relation $z_{ij}z_{i} - z_{ij}z_{j} - z_{i}z_{j} = 0$ in  $H^{*}(\mathcal{M}\{ U_{ij}^{+},U_{jk}^{+} \})$ to $0$ in $\bconfhom^3$.
\item Let $U_{i}:= \{ (x_{1}, \cdots, x_{n}) \in \R^{3n}: x_{i} = 0 \},$ and again consider the complement and corresponding inclusion $\bconf^{3}\hookrightarrow \mathcal{M}\{ U_{ij}^{+},U_{i}, U_{j} \}  .$
Then (5) holds by the same logic as relation (4).
\item 
   Finally, we obtain the relation (6) using a similar argument as the last part of \cite[Prop 11.]{feichtner2002orbit}. It is sufficient to work in the case that $n=2$. Let    
\[ U_{12}^{-}:= \{ (x_{1},x_{2} \in (\R^{3}\setminus \{ 0 \})^{2}: x_{1} \neq x_{2}, x_{1} \neq \varphi(x_{2}) \}, \]
and $U_{1},U_{2}, U_{12}^{+}$ be as before.
Consider the map 
    \begin{align*}
        \Phi: \mathcal{M} \{ U_{1},U_{2}, U_{12}^{+} \} \longrightarrow& \mathcal{M} \{ U_{12}^{+},U_{12}^{-}, U_{2} \}\\
        (x_{1},x_{2}) \longmapsto& (x_{1}-x_{2}, x_{1}- \varphi(x_{2})).
    \end{align*}
    Let $d_{1}, d_{2}$ and $d_{12}$ be the respective generators of $U_{1}, U_{2}, U_{12}$. The same argument used in $(i)$ of the proof of \cite[Prop. 11]{feichtner2002orbit} shows that $\Phi^{*}(d_{1}) = z_{12}$. 
    
    To show that $\Phi^{*}(d_{2}) = \zminus$, note that 
    \[ \Phi \circ s_{2} (x_{1},x_{2}) = (x_{1}-\varphi(x_{2}),x_{1}-x_{2}) = s_{1} \circ \Phi(x_{1},x_{2}).\]
    Hence 
\[ \Phi^{*}(d_{2}) = \Phi^{*} \circ s_{1} (d_{1}) = s_{2} \circ \Phi^{*}(d_{1}) = s_{2}(z_{12}) = \zminus. \]

Finally, we show that $\Phi^{*}(d_{12}) = -c_{2}$.
Let $\sigma: \mathcal{M} \{ U_{1},U_{2}, U_{12}^{+} \} \to \mathbb{S}^{2}$ be the projection onto the second coordinate, and then the unit sphere. Thus $\sigma^{*}(c) = d_{2}$, where $c$ is the generator of $\mathbb{S}^{2}$. Let 
\begin{align*}
\overline{\theta}:  \mathcal{M} \{ U_{1},U_{2}, U_{12}^{+} \} \longrightarrow& \SSS^{2}\\
(x_{1},x_{2}) \longmapsto& \frac{1}{|x_{1}-x_{2}|} (x_{1}-x_{2}). 
\end{align*}
Then 
\begin{align*}
    \overline{\theta} \circ \Phi (x_{1},x_{2}) =& \frac{1}{|\varphi(x_{2}) - x_{2}|}(\varphi(x_{2})-x_{2}) \\
    =& \frac{1}{|\frac{-x_{2}}{|x_{2}|^{2}} (1 + \frac{1}{|x_{2}^{2}|})|}\cdot \frac{-x_{2}}{|x_{2}|^{2}} (1+ \frac{1}{|x_{2}^{2}|}) \\
    =& \frac{1}{(1+\frac{1}{|x_{2}^{2}|})}\cdot \frac{-x_{2}}{|x_{2}|^{2}} (1+ \frac{1}{|x_{2}^{2}|}) \\
    =& \frac{-x_{2}}{|x_{2}|^{2}}\\ 
    =& - \sigma(x_{1},x_{2}).
\end{align*}
Thus $\Phi^{*} \circ \overline{\theta}^{*} = - \sigma^{*}$
because the antipodal map in this context has degree $-1$.

Noting that $\overline{\theta}^{*}(c) = z_{12}$ implies that 
\[ \Phi^{*}(d_{12}) = \Phi^{*} \circ \overline{\theta}^{*}(c) = -\sigma^{*}(c) = -\ztwo,\]
and therefore
\[ \Phi^{*}( d_{12}d_{1} - d_{12}d_{2} - d_{1}d_{2}) = -\ztwo z_{12} + \ztwo \zminus - z_{12} \zminus. \]
\end{enumerate}
\end{proof}

\subsection{Tools from representation theory}\label{tools:reptheory}
The relations in Proposition \ref{prop:FZrelations} are insufficient to give an algebra presentation for $\bconfhom^3$---that is, they do not give all the relations among the generators of the algebra. In particular, we are missing information about how to rewrite the generators $\zijbad$ and $\zijbadbad$ in terms of the basis discussed in Proposition \ref{prop:nbc}. This is relevant in part because we would like to understand the action of $B_n$ on $\bconfhom^3$. Consider, for example, what is known based on Proposition \ref{prop:FZrelations} about the action of $B_2$ on the basis for $\bconfhomn{2}^{3}$, summarized in Table \ref{tab:b2action}. The teal entries indicate elements which are not in the basis.

\begin{table}[!h]
\centering
\setlength{\tabcolsep}{10pt} % Default value: 6pt
\renewcommand{\arraystretch}{1.25} % Default
\begin{tabular}{|l||l|l|l|l|l|l}
\hline
 &   $s_{1}$ & $t_{2}$ & $s_{1}t_{2}$
 & $(s_{1}t_{2})^2= -1$ \\ \hline  \hline
1&1&1&1&1\\ 
$\zone$&$\ztwo$&$\zone $& $\ztwo$&$-\zone$\\ 

$\ztwo $&$\zone $&$-\ztwo $&$-\zone$&$-\ztwo $\\ 

$z_{12}$&$-z_{12}$ &$\zminus$&\textcolor{teal}{$\bm{-\zbad}$}&\textcolor{teal}{$\bm{\zbadbad}$}\\ 

$\zminus $&\textcolor{teal}{$\bm{-\zbad}$}& $z_{12}$ &$-z_{12}$&\textcolor{teal}{$\bm{\zbad }$}\\ 

$\zone \ztwo $&$\zone \ztwo $&$-\zone \ztwo $&$-\zone \ztwo$&$\zone \ztwo $\\ 

$\zone z_{12}$&\textcolor{teal}{$\bm{-\ztwo z_{12}}$} $= -\zone z_{12} + \zone \ztwo $&$\zone \zminus $&\textcolor{teal}{$\bm{-\ztwo \zbad = -\zone \zbad}$} $- \zone \ztwo$& \textcolor{teal}{$\bm{-\zone \zbadbad}$}\\ 

$\zone \zminus $& \textcolor{teal}{$\bm{-\ztwo \zbad = -\zone \zbad}$} $- \zone \ztwo$ & $\zone z_{12}$&\textcolor{teal}{$\bm{-\ztwo z_{12}}$} $= -\zone z_{12} + \zone \ztwo $ &\textcolor{teal}{$\bm{-\zone \zbad} $}\\
\hline
\end{tabular}
\caption{The action of $B_{2}$ on the basis of $\bconfhomn{2}^{3}$. \textcolor{teal}{\bf{Bold teal}} indicates an element not in the basis.}\label{tab:b2action}
\end{table} 
Thus the next step is to develop tools to complete the presentation of $\bconfhom^3$. In order to understand the representations carried by $\bconfhom^3$ and $\bconfhomlift^3$, assume henceforth that both spaces have coefficients in $\Q$. Our subsequent computations will use the representation theory of the hyperoctahedral group, reviewed in \S \ref{sec:reptheorytypeB}.

\subsubsection{A recursion for $\bconfhomlift^3$}\label{sec:recursion}
We will now develop a recursion between the $B_{n}$-representations carried by $\bconfhom^3$ and $\bconfhomliftn{n}^3$; in addition to being of independent interest, this recursion will be instrumental to determining the presentation for $\bconfhom^3$.

Recall that the two spaces $\bconf^3$ and $\bconflift^3$ are homeomorphic, and therefore their cohomologies are isomorphic as $B_{n}$-representations. We will use this fact to understand $\bconfhomlift^3$ as a $B_{n+1}$-representation.
The space $\bconflift^3$ has a fiber sequence
\begin{equation}\label{Xfiber}
    SU_{2} \setminus \{ \pm p_{1}, \pm p_{2}, \cdots,  \pm p_{n}\} \longrightarrow \bconflift^3 \longrightarrow \bconfliftn{n}^3.
\end{equation} 
Note that there is a $B_n$-action on $F$, which comes from the $B_n$ action on $\{ \pm p_1, \cdots, \pm p_n\}.$ In particular, for $\sigma \in B_n$, the point $(p_1, \cdots, p_n, p_{n+1})$ and $(p_{\sigma(1)}, \cdots, p_{\sigma(n)}, p_{n+1})$ have homeomorphic fibers in the fiber sequence \eqref{eq:fiberseq}, and the action of $B_n$ permutes the punctures of $F$. It follows that $B_n$ acts on each space in \eqref{eq:fiberseq}, and is equivariant with respect to both maps.

Because $\bconf^3$ is simply connected and has homology concentrated only in even degrees, the same must be true for $\bconflift^3$. It follows that the spectral sequence collapses, yielding the isomorphism:
\begin{equation}\label{eq:fiberseq}
\bconfhomlift^3 \cong_{B_{n}} \bconfhomliftn{n}^3 \otimes H^{*}(F) \cong_{B_{n}} \bconfhom^3, \end{equation}
where $F$ is the fiber $ SU_{2} \setminus \{ \pm p_{1}, \pm p_{2}, \cdots,  \pm p_{n}\}$ from \eqref{Xfiber}.

Since $F$ is connected, $H^{0}(F) \cong \Q$ and carries the trivial representation. The only other non-trivial degree is $H^{2}(F)$, which is $2n-1$ dimensional. 

\begin{lemma}
As a $B_{n}$-representation, \[ H^{2}(F) \cong_{B_{n}} \chi^{(n-1),(1)} + \chi^{(n-1,1),\emptyset}. \]
\end{lemma}
\begin{proof}
Since $F$ is a punctured sphere, the generators of the homology can be understood as cycles around these punctures. Let $e_{i}$ be the cocycle dual to the cycle around the point $p_{i}$ and $\overline{e_{i}}$ the cocycle dual to the cycle around its antipodal point $-p_{i}$. The $B_{n}$ action on $H^{2}(F)$ is determined by how $B_{n}$ permutes these cocycles. It follows that 
 \begin{equation}\label{eq:fiberfraction} H^{2}(F) \cong \frac{\Q e_{1} \oplus \Q \overline{e_{1}} \oplus \dots \oplus \Q e_{n} \oplus \Q \overline{e_{n}}}{\Q(e_{1} + \overline{e_{1}} + \dots + e_{n} + \overline{e_{n}})}.  \end{equation}
 The denominator carries the trivial representation, which we shall denote by $\triv$. To compute the numerator, note that the subgroup fixing $e_{1}$ pointwise is $S_{1} \times B_{n-1}$, which has index $2n$ in $B_{n}$. Hence 
 \[\Q e_{1} \oplus \Q \overline{e_{1}} \oplus \dots \oplus \Q e_{n} \oplus \Q \overline{e_{n}} \cong_{B_{n}} \ind_{S_{1} \times B_{n-1}}^{B_{n}} \triv \cong_{B_{n}} \ind^{B_{n}}_{B_{1} \times B_{n-1}} \left( \ind^{B_{1} \times B_{n-1}}_{S_{1} \times B_{n-1}} \triv \right)  \]
 by transitivity of induction. The inner term on the left-hand-side can be expanded as
\[ \ind^{B_{1} \times B_{n-1}}_{S_{1} \times B_{n-1}} \triv = (\chi^{(1),\emptyset} \oplus \chi^{\emptyset,(1)}) \times \chi^{(n-1),\emptyset} = ( \chi^{(1),\emptyset} \times \chi^{(n-1),\emptyset}) \oplus ( \chi^{\emptyset,(1)} \times \chi^{(n-1),\emptyset} ). \]

Using the Type $B$ branching rules, it follows that 
\[ \ind_{B_{1} \times B_{n-1}}^{B_{n}} ( \chi^{\emptyset,(1)} \times \chi^{(n-1),\emptyset}) =  \chi^{\emptyset,(1)} \cdot \chi^{(n-1),\emptyset} = \chi^{(n-1),(1)}, \]
which is an $n$-dimensional representation of $B_{n}$. 

The other term is 
 \[ \ind_{B_{1} \times B_{n-1}}^{B_{n}} (\triv_{B_{1}} \times \triv_{B_{n-1}}) = \Q[B_{n}/ (B_{1} \times B_{n-1})], \]
 which is also $n$-dimensional and decomposes as $\chi^{(n),\emptyset} \oplus \chi^{(n-1,1), \emptyset}$.
 The numerator in \eqref{eq:fiberfraction} therefore has description 
 \[\Q e_{1} \oplus \Q \overline{e_{1}} \oplus \dots \oplus \Q e_{n} \oplus \Q \overline{e_{n}} = \chi^{(n),\emptyset} \oplus \chi^{(n-1,1), \emptyset} \oplus \chi^{(n-1),(1)} \]
 and so
 \begin{equation}\label{fibercomp} H^{2}(F) = \frac{\Q e_{1} \oplus \Q \overline{e_{1}} \oplus \dots \oplus \Q e_{n} \oplus \Q \overline{e_{n}}}{\Q(e_{1} + \overline{e_{1}} + \dots + e_{n} + \overline{e_{n}})} =  \chi^{(n-1,1), \emptyset} \oplus \chi^{(n-1),(1)}.  \end{equation}
\end{proof}

Plugging the description of $H^{*}(F)$ into \eqref{eq:fiberseq} yields the following recursion. 
\begin{cor}\label{cor:recursion} There is a isomorphism of $B_{n}$-modules
\[ H^{2j}\bconf^3 \cong_{B_{n}} H^{2j}\bconfliftn{n}^3 \oplus \Big( H^{2(j-1)}\bconfliftn{n}^3 \otimes V \Big) \]
for $0 \leq j \leq n$, where 
$V= \chi^{(n-1,1), \emptyset} \oplus \chi^{(n-1),(1)}$.
\end{cor}

\subsubsection{The case that $n=2$}\label{sec:reptheoryintuition}

Our strategy will be to compute the representation for $\bconfhomn{2}^3$ using Corollary \ref{cor:recursion}, and then use the fact that the cohomology generators $z_{ij}, z_i$ are compatible with projections between configuration spaces to deduce a presentation for $\bconfhom^3$.
In particular, if
\[\Pi_{[S]}: \bconfhom^3 \longrightarrow H^{*}\bconfn{|T|}^3 \]
is the map projecting $(x_{1}, \cdots, x_{n})$ to the points with index $T \subset [n]$, then Feichtner-Ziegler show in \cite[Lemma 6]{feichtner2002orbit} that in cohomology, $\Pi^{*}_{|T|}(z_{ij}) = z_{ij}$ for $i,j \in T$. 

Part of our upcoming argument will rely on Theorem \ref{ungradedrep}, which states that as an ungraded representation, 
\[ \bconfhom^3 \cong_{B_{n}} \Q [B_{n}]. \]
We will defer the proof\footnote{The proof of Theorem~\ref{ungradedrep} does not rely on the remaining results in \S \ref{sec:d3}. In particular, all subsequent work in \S \ref{sec:d3} will focus on computing the presentation of $\bconfhom^3$. Theorem~\ref{ungradedrep} will be proved using equivariant formality, which we will see follows from the Hilbert series for $\bconfhom^3$ given in \eqref{eq:hilbertseries}.} of Theorem \ref{ungradedrep} to Section \ref{sec:equivcohom}. One consequence of this fact is that $\bconfhom^3$ will contain exactly one copy of every 1-dimensional representation of $B_n$. 

Our goal in Proposition~\ref{reptheory:n=2} is to determine where each irreducible of $\Q[B_2]$ appears in $\bconfhomn{2}^3$. Denote by $\rho[j]$ a representation $\rho$ in cohomological degree $j$.

\begin{prop}\label{reptheory:n=2}
 As $B_{2}$ representations, $\bconfhomn{2}^3$ decomposes as
  \begin{align*}
    H^0 \bconfn{2}^3 &= \chi^{(2),\emptyset}\\
    H^{2} \bconfn{2}^3 &= \chi^{(1,1),\emptyset} + \chi^{\emptyset, (1,1)} + \chi^{(1),(1)}\\
    H^{4} \bconfn{2}^3 &= \chi^{\emptyset,(2)} + \chi^{(1),(1)}.
 \end{align*}
\end{prop}
\begin{proof}
By Corollary \ref{cor:recursion}, when $n=2$, 
\[ V = \chi^{(1,1), \emptyset} \oplus \chi^{(1),(1)}\]
and $H^{0}\bconfn{2}^{3} = \chi^{(2),\emptyset}$. Since $H^{0}\bconfn{1}^{3}= \chi^{(1),\emptyset}$, by Theorem \ref{ungradedrep} it follows that $H^2 \bconfn{1}^{3} = \chi^{\emptyset,(1)}$. 
Write the $B_{2}$ representation carried by $H^2\bconfliftn{2}^{3}$ as $\chi$; note that $\chi$ is 1-dimensional and must restrict to $\chi^{\emptyset, (1)}$. Inspection of Tables \ref{charactertableb1} and \ref{charactertableb2} then shows that $\chi$ must be either $\chi^{\emptyset, (1,1)}$ or $\chi^{\emptyset, (2)}$. 

Expanding the recursion in Corollary \ref{cor:recursion} gives
\begin{align*} \bconfhomn{2}^{3}&= (\chi^{(2),\emptyset}[0] + (\chi^{(1),(1)} +  \chi^{(1,1), \emptyset}[2]))\otimes (\chi^{(2),\emptyset}[0] + \chi [2])\\
&= \chi^{(2),\emptyset}[0] + (\chi^{(1),(1)} +  \chi^{(1,1), \emptyset} + \chi)[2] + (\chi^{(1),(1)} +  \chi^{(1,1), \emptyset}\otimes \chi)[4],
\end{align*}
where we use the fact that $\chi^{(1),(1)} \otimes \chi = \chi^{(1),(1)}$ no matter which 1-dimensional representation $\chi$ is.

Observe by Proposition \ref{prop:FZrelations} that $\zone \ztwo \in H^{4}\bconfn{2}^3$ is an eigenbasis for $\chi^{\emptyset,(2)}$. Since the total representation of $\bconfhomn{2}^{3}$ is the regular representation of $B_{2}$, it follows that $\chi^{\emptyset,(2)}$ must appear in cohomological degree 4 only, and so it must be that 
\[ \chi^{\emptyset,(2)} = \chi^{(1,1),\emptyset} \otimes \chi.\]
Again using the fact that the total representation in this context is $\Q[B_2]$, one deduces that $\chi=\chi^{\emptyset,(1,1)}$ and therefore
\[\bconfhomn{2}^{3} = \chi^{(2),\emptyset}[0] + (\chi^{(1),(1)} +  \chi^{(1,1), \emptyset} + \chi^{\emptyset,(1,1)})[2] + (\chi^{(1),(1)} +  \chi^{\emptyset,(2)})[4]. \]
\end{proof}

We will now determine (1) the
1-dimensional eigenspaces of $\bconfhomn{2}^3$ as $B_2$ representations and (2) a complete description of the $B_2$ action on the generators of $\bconfhomn{2}^3$, which can then be extended to the $B_n$ action on $\bconfhom^3$. Recall that since $\bconfhom^3 \cong \Q[B_n]$, each 1-dimensional representation of $B_n$ will occur once in $\bconfhom^3$.
\begin{prop}\label{prop:actionongenerators}
Let $V(\chi)$ be an eigenvector in $\bconfhomn{2}^3$ of a 1-dimensional representation $\chi$ of $B_{2}$. The four $1$-dimensional representations $\chi$ of $B_2$ occur in $\bconfhomn{2}^3$ as the span of the following elements:
\begin{table}[!h]
\centering
\setlength{\tabcolsep}{10pt} % Default value: 6pt
\renewcommand{\arraystretch}{1.25} % Default
\begin{tabular}{|l||l|} \hline
$V(\chi)$ & Eigenvector \\ \hline \hline
$V(\chi^{(2),\emptyset})$ & $1$ \\ \hline
$V(\chi^{(1,1),\emptyset})$ & $z_{12} + \zminus  + \zone$\\ \hline
$V(\chi^{\emptyset,(1,1)})$ & $z_{12} - \zminus  - \ztwo$ \\ \hline
$V(\chi^{\emptyset,(2)})$ & $\zone  \ztwo $ \\ \hline
\end{tabular}
\label{table:n=2eigenspaces}
\end{table}

This in turn determines the following relations:
\begin{align}
    \zbadbad &= z_{12} + \zone -  \ztwo \label{eqzbadbad} \\
    \zbad &= \zminus +\zone +  \ztwo. \label{eqzbad}
\end{align}
\end{prop}
\begin{proof}
The copy of $\chi^{(1),(1)}$ in $H^{2}\bconfn{2}^{3}$ has basis $\zone$ and $\ztwo$. By Proposition \ref{reptheory:n=2}, the remaining representations in $H^{2}\bconfn{2}^{3}$ are $\chi^{(1,1),\emptyset}$ and $\chi^{\emptyset,(1,1)}$.
For vectors $\alpha = (\alpha_{12}^{+},\alpha_{12}^{-}, \alpha_{1}, \alpha_{2}) \in \Q^{4}$ and $\beta = (\beta_{12}^{+},\beta_{12}^{-}, \beta_{1}, \beta_{2}) \in \Q^{4},$ write
\[ V(\repone) = \alpha_{12}^{+} z_{12} + \alpha_{12}^{-} \zminus  + \alpha_{1} \zone  + \alpha_{2} \ztwo  \]
and 
\[ V(\repsgn) = \beta_{12}^{+} z_{12} + \beta_{12}^{-} \zminus  + \beta_{1} \zone  + \beta_{2} \ztwo . \]
Our goal will be to solve for the undetermined coefficients above.

Starting with $V(\repone)$, we have that
\begin{align*}
    t_{2} \cdot V(\repone) &= \alpha_{12}^{+}\zminus   + \alpha_{12}^{-} z_{12} + \alpha_{1} \zone  - \alpha_{2} \ztwo  \\
    &= V(\repone) \\
    &=  \alpha_{12}^{+} z_{12} + \alpha_{12}^{-} \zminus  + \alpha_{1} \zone  + \alpha_{2} \ztwo . 
\end{align*}
This forces $\alpha_{12}^{+} = \alpha_{12}^{-}$ and $\alpha_{2} = 0$. Note that $\alpha_{12}^{+} \neq 0$, because otherwise $V(\repone)$ would not be linearly independent with $\langle \zone , \ztwo  \rangle$. Thus without loss of generality, scale $\alpha_{12}^{+} = \alpha_{12}^{-} = 1$; write $\alpha_{1}/\alpha_{12}^{+} = \alpha$, so that 
\[ V(\repone) = z_{12} + \zminus + \alpha \zone. \]
Acting by $s_{1}$ gives
\begin{align*}
    s_{1} \cdot V(\repone) &= -z_{12} -\zbad + \alpha \ztwo \\
    &= -V(\repone) \\
    &= -z_{12} -\zminus - \alpha \zone.
\end{align*}
This implies that 
\[ \zbad = \zminus + \alpha \zone + \alpha \ztwo. \]

By a similar argument, 
\begin{align*}
    t_{2} \cdot V(\repsgn) &= \beta_{12}^{+}\zminus + \beta_{12}^{-} z_{12} + \beta_{1} \zone - \beta_{2} \ztwo\\ 
    &= -V(\repsgn)\\
    &=  -\beta_{12}^{+} z_{12} - \beta_{12}^{-}\zminus- \beta_{1} \zone - \beta_{2}\ztwo. 
\end{align*}
Hence one can conclude that $\beta_{12}^{+} = -\beta_{12}^{-} \neq 0$ and $\beta_{1} =0$. Again, normalize by $\beta_{12}^{+} = 1$ and set $\beta = \beta_{2}/\beta_{12}^{+}$; therefore
\[ V(\repsgn) = z_{12} - \zminus + \beta \ztwo. \]
Acting again by $s_{1}$ gives 
\begin{align*}
    s_{1} \cdot V(\repsgn) &= -z_{12} + \zbad + \beta \zone\\ 
    &= -V(\repsgn) \\
    &=  - z_{12} + \zminus - \beta \ztwo,
\end{align*}
which implies that 
\begin{equation}\label{eq:zbad} \zbad = \zminus - \beta \zone - \beta \ztwo. \end{equation}
We therefore conclude that $\alpha = -\beta$, and so 
\[ V(\repone) = z_{12} + \zminus - \beta \zone. \]

Applying $w_{0}$ to both expressions gives:
\begin{align*}
 w_{0} \cdot V(\repone) &= \zbadbad + \zbad + \beta\zone= V(\repone) = z_{12} + \zminus - \beta \zone\\
  w_{0} \cdot V(\repsgn) &= \zbadbad - \zbad - \beta \ztwo= V(\repsgn) = z_{12} - \zminus + \beta \ztwo.
\end{align*}
Simplifying and adding the above two expressions gives 
\begin{align*} 2 \zbadbad &= 2 z_{12} - 2 \beta \zone + 2 \beta \ztwo.\\
\end{align*}

Finally, we must show that $\beta = -1$. By Proposition \ref{prop:FZrelations} (5),
\[ z_{12} \zone - z_{12}\ztwo - \zone\ztwo = 0\]
holds in $\bconfhomn{2}^3$.
Acting by $t_{2}$ gives
\[ \zminus \zone  + \zminus \ztwo + \zone \ztwo = 0. \]
One the other hand, acting by $s_{1}t_{2}s_{1}$ gives
\[ -\zbad \zone - \zbad \ztwo + \zone \ztwo = 0.\]
Using \eqref{eq:zbad}, and Propopsition \ref{prop:FZrelations} (3), we have
\begin{align*} -\zbad \zone - \zbad \ztwo + \zone \ztwo &= -(\zminus - \beta \zone - \beta \ztwo) \zone - (\zminus - \beta \zone - \beta \ztwo) \ztwo + \zone \ztwo \\
&= -\zminus \zone - \zminus \ztwo +(1+ 2\beta) \zone \ztwo.\end{align*}
It follows that 
\begin{align*}
0 &= \Big(\zminus \zone  + \zminus \ztwo + \zone \ztwo \Big) + \Big( -\zminus \zone - \zminus \ztwo +(1+ 2\beta) \zone \ztwo \Big) \\
&= (2+2\beta)\zone \ztwo.
\end{align*}
 Since $\zone \ztwo$ is a basis element, $\beta = -1$.
\end{proof}
The relations \eqref{eqzbadbad} and \eqref{eqzbad} are the final relations needed to determine a presentation for $\bconfhom^3$; note that they precisely match the relation (2) in Corollary \ref{cor:grbconfhom}.

\subsection{Presentation for $\bconfhom^3$}\label{sec:presbconfhom}
We are now ready to give a presentation for $\bconfhom^3$. As before, the action of  $\sigma \in B_n$ on the variables $z_{ij}$ and $z_i$ for $i,j \in [n]^{\pm}$ is given by
\[ \sigma \cdot \begin{cases} z_{ij} = z_{\sigma(i)\sigma(j)}\\
z_i = z_{\sigma(i),}
\end{cases}\]
with the convention that $\overline{\overline{i}} =i$.
\begin{theorem}\label{thm:3pres}
There is a $B_{n}$-equivariant ring isomorphism 
\[ \bconfhom^3 \cong \Z[z_{ij},z_{i}]/ \J', \]
for $i,j \in [n]^{\pm}$, where $\J'$ is generated by
\begin{itemize}
\item  $\zijbad = \zijminus +z_i +  z_j$ and
    \item The relations in Proposition \ref{prop:FZrelations}.
\end{itemize}
\end{theorem}
\begin{proof}
Proposition \ref{prop:FZrelations} and \ref{prop:actionongenerators} show that the relations in $\J'$ hold in $\bconfhom^3$. (The relation $\zijbad = \zijminus +z_i +  z_j$ comes from \eqref{eqzbad}; note that it also implies $\zijbadbad = z_{ij} + z_i - z_j$ from the action by $t_j$.)

Hence it is sufficient to prove these are the only relations needed to generate $\J'$. To do so we will employ a standard argument,
very similar in spirit to the proof of Theorem \ref{thm:lift1pres}, to show that 
    \[ \underbrace{\Z \{ 1, \zone \}}_{h'_1} \cdot \underbrace{\Z \{ 1, \ztwo , z_{12}, \zminus, \}}_{h'_2} \cdots \underbrace{\Z \{ 1, z_{n}, z_{1n}, z_{1\overline{n}}, \cdots, z_{(n-1)n}, z_{(n-1)\overline{n}} \}}_{h'_n}.  \]
forms a Gr{\"o}bner basis for $\bconfhom^3$. We will again use the lexicographic ordering on $z_{i}, z_{ij}$ for $i,j \in [n]^{\pm}$ induced by the ordering
 \[\overline{0}< 1 < \overline{1} < 2 < \overline{2} < \cdots < n < \overline{n}.\]
 Here we identify $\zi$ with $z_{\overline{0}i}$ so that for $i < j$ one has $\zi < \zj < z_{ij} < \zijminus < \zijbad < \zijbadbad$. The argument amounts to showing that any monomial with two terms from the same set $h'_i$ can be rewritten in terms of elements lower in the above ordering using the relations in $\J'$. Call such monomials \emph{broken-circuits}. The same argument as in Theorem \ref{thm:lift1pres} shows that these broken-circuits (the \underline{underlined} terms) can be rewritten using relations in $\J'$:
\begin{itemize}
    \item $0 = z_{ij}z_{j k } - z_{ij}z_{jk} - \underline{z_{jk}z_{ik}}$%$z_{ik}z_{jk}$
    \item $0 = z_{ij}z_{j \overline{k} } - z_{ij}z_{j\overline{k}} - \underline{z_{j\overline{k}}z_{i\overline{k}}}$ %$z_{i\overline{k}} z_{j \overline{k}}$
    \item $0= \zi z_{ij} - \underline{z_{ij} \zj} - \zi \zj$%$\zj z_{ij}$
    \item $0= \zi \zijminus + \underline{\zijminus \zj} + \zi \zj$ %$\zj \zijminus$
    \item $0=\zj \zijminus - \zj z_{ij} - \underline{z_{ij} \zijminus}$ %z_{ij} \zijminus$
    \item  
    $0=\zj \zijminus + \zk \zijminus + \zijminus z_{j\overline{k}} - \zj z_{ik} - \zk z_{ik} - \underline{z_{j\overline{k}} z_{ik}} - \zijminus z_{ik} $
    \item  %Rewriting $0= \zijbad z_{jk} - \zijbad \zikbad - z_{jk}\zikbad + \zikbad$ gives
        $0 = (\zi + \zj + \zijminus)z_{jk} - (\zi + \zj + \zijminus) (\zi + \zk + \zikminus)- \underline{z_{jk}}(\zi + \zk + \underline{\zikminus})$
\end{itemize}
for $i,j,k \in [n]$.
Note that each underlined term is the largest with respect to our ordering, which completes the proof.
\end{proof}
The presentation of $\bconfhom^3$ in Theorem \ref{thm:3pres} is identical to the presentation of $\G= \gr(\bconfhom^1)$ in Corollary \ref{cor:grbconfhom}, which gives the following.
\begin{cor}\label{cor:gr1and3}
There is an isomorphism of $B_{n}$-modules
    \[ \bconfhom^3 \cong \gr(\bconfhom^1) = \G.\]
\end{cor}

It turns out that Corollary \ref{cor:gr1and3} can be \emph{lifted} to the spaces $\bconfhomlift^3$ and $\bconfhomlift^1$:
\begin{cor}\label{cor:liftgr1and3}
There is an isomorphism of $B_{n+1}$-modules
    \[ \bconfhomlift^3 \cong \gr(\bconfhomlift^1)\]
\end{cor}
The proof of Corollary \ref{cor:liftgr1and3} requires additional tools that will be discussed in Section \ref{sec:equivcohom}. 

\begin{remark}\label{rmk:difpres} \rm Theorem \ref{thm:3pres} and its proof reveal an asymmetry in the space $\bconfhom^3$; in particular, only the generators $\zijminus, z_{ij}$ and $z_i$ for $1 \leq i < j \leq n$ are needed, but the presentation can be described more compactly (and elegantly) if the generator $\zijbad$ is also included. One alternative way to describe the presentation of $\bconfhom^3$ is to set $z_{ij}^{+}:= z_{ij}$ and $z_{ij}^{-}:= \zijminus$ with the action of $B_n$ given by Table \ref{table:actionsimplifiedpres}.

\begin{table}[!h]
\centering
\setlength{\tabcolsep}{10pt} % Default value: 6pt
\renewcommand{\arraystretch}{1.25} % Default
\begin{tabular}{|l||l|l|l|} \hline
  & $t_i$ & $t_j$  & $(i,j)$ \\ \hline \hline
  
 $z_{ij}^{+}$ &$z_{ij}^{-} + z_i + z_j$ & $z_{ij}^{-}$ & $-z_{ij}^{+}$\\ \hline
 
 $z_{ij}^{-}$ & $z_{ij}^{+} + z_i - z_j$ & $z_{ij}^{+}$& $-z_{ij}^{-} - z_i - z_j$\\ \hline
 
$z_i$&$-z_{i}$&$z_i$& $z_j$ \\ \hline
\end{tabular}
\caption{Action of the elements $t_i, t_j$ and $(i,j) \in B_n$ on the generators $z_{ij}^+, z_{ij}^{-}, z_i \in \bconfhom^3$. }\label{table:actionsimplifiedpres}
\end{table} 

One could then write the presentation of $\bconfhom^3$ as $\Z[z_{ij}^{+}, z_{ij}^{-}, z_i] / \mathcal{K},$
where $\mathcal{K}$ is generated by the \textbf{$B_{n}$-image} of the relations $(z_{ij}^{+})^{2} = (\zijminus)^{2} = (z_i)^{2}$ and
\[  (i) \hspace{.5em} z_{ij}^{+}z_{j k }^{+} - z_{ij}^{+}z_{jk}^{+} - z_{jk}^{+}z_{ik}^{+} \hspace{2em} (ii) \hspace{.5em}  \zi z_{ij}^{+} - z_{ij}^{+} \zj - \zi \zj \hspace{2em} (iii) \hspace{.5em}\zj z_{ij}^{-} - \zj z_{ij}^{+} - z_{ij}^{+} z_{ij}^{-}. \]

This is more or less a rewriting of Theorem \ref{thm:3pres}, but has the disadvantage of being less explicit and less obviously related to the presentation of $\bconfhom^1$. However, there are some advantages to this perspective, and we will return to it in Section \ref{sec:equivcohom}.

Similarly, the presentation of the ring $\G$ can be re-written
\[ \G = \Z[ z_{ij}^{+}, z_{ij}^{-}, z_i ]/ \mathcal{L}' \]
for $1 \leq i < j \leq n$, where $\mathcal{L}'$ is generated by $ (z_{ij}^{+})^{2} = (z_{ij}^{-})^{2} = (z_{i})^{2}$ and 
\[  (i) \hspace{.5em} z_{ij}^{+}z_{ij}^{-}\hspace{1.5em} (ii) \hspace{.5em} z_{ij}^{+} z_{j} - z_{ij}^{+}z_{i}  \hspace{1.5em} (iii)  \hspace{.5em} z_{ij}^{-} z_{j} + z_{ij}^{-} z_{i} \hspace{1.5em} (iv) \hspace{.5em} z_{ij}^{+}z_{jk}^{+} - z_{ij}^{+}z_{ik}^{+} - z_{jk}^{+}z_{ik}^{+}, \hspace{1.5em}\]
\[  (v) \hspace{.5em} z_{ij}^{-}z_{jk}^{+} - z_{ij}^{-}z_{ik}^{-} - z_{jk}^{+}z_{ik}^{-} \hspace{1.5em} (vi) \hspace{.5em} z_{ij}^{-}z_{jk}^{-} - z_{ij}^{-}z_{ik}^{+} - z_{jk}^{-}z_{ik}^{+} \hspace{1.5em} (vii)  \hspace{.5em} z_{ij}^{+}z_{jk}^{-} - z_{ij}^{+}z_{ik}^{-} - z_{jk}^{-}z_{ik}^{-}\]

Now the action by $B_n$ on the generators (see Table \ref{table:actionsimplifiedpresbigrading}) is significantly simpler.

 \begin{table}[!h]
\centering
\setlength{\tabcolsep}{10pt} % Default value: 6pt
\renewcommand{\arraystretch}{1.25} % Default
\begin{tabular}{|l||l|l|l|} \hline
  & $t_i$ & $t_j$  & $(i,j)$ \\ \hline \hline
 $z_{ij}^{+}$ &$z_{ij}^{-}$ &$z_{ij}^{-}$ & $-z_{ij}^{+}$\\ \hline
 $z_{ij}^{-}$ & $z_{ij}^{+}$ & $z_{ij}^{+}$& $-z_{ij}^{-}$\\ \hline
$z_i$&$-z_{i}$&$z_i$& $z_j$ \\ \hline
\end{tabular}

\caption{Action of the elements $t_i, t_j$ and $(i,j) \in B_n$ on the generators $z_{ij}^+, z_{ij}^{-}, z_i \in \G$. }\label{table:actionsimplifiedpresbigrading}
\end{table} 

\end{remark}

\section{$T$-equivariant cohomology}\label{sec:equivcohom}
There are several results whose proofs we have deferred thus far. In this section, we will build the tools to prove these theorems and more deeply understand the connections between the $d=1$ and $d=3$ cases for $\bconf^d$ and $\bconflift^d$. To do so, we will delve into the world of equivariant cohomology; background material on equivariant cohomology can be found in \S \ref{equivtopbackground}. Experts on the topic can proceed directly to \S \ref{equivtop:equivformality}.

\subsection{Background}\label{equivtopbackground}
Equivariant topology is a powerful tool by which to study a topological space $X$ with a group action by a Lie group $T$; see Anderson \cite{anderson2012introduction}, Proudfoot \cite{proudfootsymplectic}, and Tymoczko \cite{tymoczko2005introduction} for excellent introductions to the topic. Recently in \cite{moseley2017equivariant}, Moseley successfully used equivariant cohomology to study complements of hyperplane and subspace arrangements in real Euclidean space. Our approach will follow Moseley's lead.

Henceforth, the key objects will be:
\begin{itemize}
    \item $T$, a Lie group. For us, $T$ will almost exclusively be $U(1)$, the unit circle;
    \item $X$, a topological space that $T$ acts on; call $X$ a \emph{$T$-space} and let $X^{T}$ be the space fixed by $T$. For our purposes $X$ will be an oriented manifold. Call such a manifold a \emph{$T$-manifold}, and any submanifold of $X$ that is stable under the action of $T$ a \emph{$T$-submanifold};
    \item $ET$, a contractible space which $T$ acts upon freely; 
    \item $X_{T}:= ET \times_{T} X,$ is defined to be the quotient of $ET \times X$ by \[ (t \cdot e, x) \sim (e, t \cdot x) \]
    for $t \in T$, $e \in ET$ and $x \in X$. A nice special case is when $X$ is a point; then $X_{T}$ is simply the classifying space of $T$, written $BT$. See example \ref{ex:ECpt} for more details.
    \item $H_{T}^{*}(X):= H^{*}(X_{T})$ is the $T$-equivariant cohomology of $X$. This is precisely the ordinary cohomology for $X_{T}$. Intuitively, if $T$ acts freely on $X$ then $H^{*}_{T}(X) = H^{*}(X/T).$
    \item $[Y]_{T} \in H_{T}^{k}(X)$ is the class of a codimension $k$ oriented $T$-stable submanifold of $X$. 
\end{itemize}

To make the above more concrete, consider the following example.
\begin{example}\label{ex:ECpt} \rm
Let $X$ be a single point, denoted $\pt$, and $T = U(1)$. To determine $ET$, one must find a space that (1) $T$ acts upon freely and (2) is contractible. There is a natural free action of $T$ on $\mathbb{S}^{2n+1} \subset \C^{n+1}$ by rotation, where $e^{\pi i} \in U(1)$ acts on $(z_{0}, \cdots, z_{n})$ by 
\[ e^{\pi i} \cdot (z_{0}, \cdots, z_{n}) = (e^{\pi i}z_{0}, \cdots, e^{\pi i}z_{n}).  \]
While $\SSS^{2n+1}$ satisfies the first requirement, it is not contractible. To correct for this, note that $\SSS^{2n+1}$ can be successively embedded in $\SSS^{2n + 2\ell + 1}$ for $\ell \geq 0$. Following this procedure iteratively allows us to view $\SSS^{2n+1}$ inside of $\SSS^{\infty}$ (defined as the union of $\SSS^{2n+1}$ for all $n$) which \emph{is} a contractible space. Thus $ET = \mathbb{S}^{\infty}$. Since $X = \pt$, the space $\pt_{T}$ is simply
\[ \pt_{T} = \mathbb{S}^{\infty} \times_{T} \pt =  \mathbb{S}^{\infty}/\mathbb{S}^{1}. \]
In general, the quotient of $\SSS^{2n+1}$ by the rotation action of $\SSS^{1}$ gives complex projective space $\C \PP^{n}$, and in the limit
\[\SSS^{\infty}/\mathbb{S}^{1}=\C \mathbb{P}^{\infty}, \]
where $\C \mathbb{P}^{\infty}$ is infinite complex projective space. 
Hence $H^{*}_{T}(pt) = H^{*}(\C \mathbb{P}^{\infty}).$
The cohomology of $\C \mathbb{P}^{\infty}$ (with coefficients in $\Q$) is the univariate polynomial ring $\Q[u]$ where $u \in H^{2}(\C \mathbb{P}^{\infty}).$

Here, we understand $u$ as the image of $[0]_{T}$ under the induced map from the inclusion
\begin{align*} i: \pt &\xhookrightarrow{} \C  \\
\pt &\mapsto 0.
\end{align*}
\end{example}

Example \ref{ex:ECpt} illustrates one of the key differences between ordinary and equivariant cohomology: in the latter, $H^{*}(\pt) = \Q$ a (relatively) uninteresting ring. By contrast, $H^{*}_{T}(\pt) = \Q[u]$ has a richer structure.

Many of the important tools in ordinary cohomology also hold in the equivariant setting, including functoriality, cohomological ring structure, Poincare duality for smooth orientable spaces, the Leray-Serre spectral sequence, the K{\"u}nneth formula, and more (see \cite{tymoczko2005introduction}). This means, for starters, that the argument showing that $H^{*}(\pt) = \Q$ is a module over $H^{*}(X)$ can be used to show that $H^{*}_{T}(\pt)$ is a module over $H^{*}_{T}(X)$. In the case most relevant to us, where $T = U(1)$, this says that $H^{*}_{T}(X)$ is a $\Q[u]$-module.  Remark \ref{rmk:intersectionnumbers} discusses another ``standard'' fact.
\begin{remark}\label{rmk:intersectionnumbers} \rm
One property which we will use repeatedly is that if $T$-manifolds $Y,Z \subset X$ have empty intersection, then in $H_{T}^{*}(X)$, their product $[Y]_{T} \cdot [Z]_{T}$ is $0$. This is a special case of the more general fact that if $Y$ and $Z$ intersect transversely their product in equivariant cohomology will be an intersection number. The same statement is a standard fact in ordinary cohomology, and can be proved in much the same way for equivariant cohomology. See Dorpalen-Barry--Proudfoot--Wang \cite[\S 2.5]{DB-Moseley-Wang} for a proof of this fact over $\Z$, along with more intuition for the equivariant classes represented by submanifolds.
\end{remark}

The space $X_{T}$ is a fiber bundle over the classifying space $BT = ET/T$: 
\begin{equation}\label{eq:equivfiberseq}
 X \longrightarrow X_{T} \longrightarrow B_{T}.
\end{equation} 
In principal, one might try to compute the ordinary cohomology of $X_{T}$ using a Leray spectral sequence. In general this is very difficult, but we will see that in our case the relevant spectral sequence collapses. 

In particular, a space $X$ is \emph{equivariantly formal} (with respect to $T$) if the spectral sequence in \eqref{eq:equivfiberseq} collapses at the $E_{2}$ page; this means that $H^{*}_{T}(X)$ is a free module over $H_{T}(\pt)$. Importantly, if $X$ and $BT$ have cohomology concentrated only in even degrees, then $X$ is equivariantly formal with respect to $T$. All of the spaces we consider (e.g. $\bconf^3$ and $\bconflift^3$) will satisfy this hypothesis.

If $X$ is equivariantly formal with respect to $T$, there is an $H_{T}^{*}(\pt)$-module isomorphism
\[ H_{T}^{*}(X) = H^{*}(X) \otimes H_{T}^{*}(\pt). \]
Again, we almost exclusively care about the case that $T = U(1)$; the above then says that 
\[ H_{T}^{*}(X) = H^{*}(X) \otimes \Q[u].  \]
Equivariant formality thus allows us to recover information about $H^{*}(X)$ from $H^{*}_{T}(X)$ and vice-versa.
For instance, if $X$ is equivariantly formal, then any $\Q$-basis for $H^{*}(X)$ is an $H^{*}_{T}(\pt)$-basis for $H^{*}_{T}(X)$; if $T = U(1)$ and $v_{1}, \cdots, v_{k}$ is a basis for $H^{*}(X)$, then $v_{1}, \cdots, v_{k}, u$ is a $\Q$-basis for $H^{*}_{T}(X)$.

\begin{example}\label{ex:rthree} \rm

Let $T = U(1)$ and $X = \rthree.$ Then $T$ acts on $\rthree$ by rotation around the $x$-axis. Let $Z^{+}$ and $Z^{-}$ be the positive and negative $x$-axes, respectively. Once we fix an orientation of $\R^{3} \setminus \{ 0 \}$ and orient $Z^{+}$ and $Z^{-}$ outwards, we obtain two classes in $H^{2}_{T}(\R^{3} \setminus \{ 0 \})$, namely $[Z^{+}]_{T}$ and $[Z^{-}]_{T}$. 

Consider the map 
\begin{align*}
f: \Q[x,y]/(xy) &\longrightarrow H^{*}_{T}(\rthree) \\
  y  & \longmapsto [Z^{+}]_{T} \\
  x  & \longmapsto  [Z^{-}]_{T}. 
\end{align*}

By Remark \ref{rmk:intersectionnumbers}, since $Z^{+} \cap Z^{-} = \emptyset$, their corresponding classes in equivariant cohomology multiply as
\[ [Z^{+}]_{T} \cdot  [Z^{-}]_{T}  = 0,\]
and so $f$ is well-defined.
In fact, $f$ defines an isomorphism; this argument is written out in detail in  \cite[Prop 5.22]{proudfootsymplectic} and \cite[Example 2.4]{moseley2017equivariant}.

Since both $\rthree$ and $BT = \C \PP^{\infty}$ have cohomology concentrated in even degrees, $\rthree$ is equivariantly formal, and therefore
\[  H_{T}^{*} (\rthree) \cong \Q[x,y]/(xy) \cong \Q[u] \otimes H^{*}(\rthree). \]
We would like an explicit map between $ \Q[x,y]/(xy)$ and $\Q[u] \otimes H^{*}(\rthree)$. To this end, consider the $T$-equivariant projection
\begin{align*}\pi: \rthree &\longrightarrow \R^{2}\\
(x,y,z) &\longmapsto (y,z),
\end{align*}
so that $\pi^{-1}(0) = Z^{+} \sqcup -Z^{-}$ (when accounting for orientation). We are interested in the image of $u$ under the induced maps  
\[H^{*}_{T}(pt) \xrightarrow{i^{*}} H_{T}^{*}(\R^{2}) \xrightarrow{\pi^{*}} H^{*}_{T}(\rthree). \]
The first map sends $u$ to $[0]_{T}$. By the above discussion, $\pi^{*}([0]_{T}) = y-x$, and therefore
\[ \pi^{*} (i^{*}(u)) = y-x. \]
Note that in $H^{*}(\rthree),$ the classes $[Z^{+}]$ and $[Z^{-}]$ represent the same generator of $H^{*}(\rthree)$. This gives a (non-canonical) isomorphism:
\begin{align*} f': \Q[x,y]/(xy)   &\longrightarrow \Q[u] \otimes H^{*}(\rthree) \\
y-x &\longmapsto u \\
y &\longmapsto [Z^{+}].
\end{align*}
See \cite[Prop 5.22]{proudfootsymplectic} and \cite[Example 2.4]{moseley2017equivariant} for a more in-depth description of this example.
\end{example}

There are several consequences of equivariant formality proved by Moseley in \cite{moseley2017equivariant} that will play an important role in our analysis. 
\begin{prop}{\rm (Moseley \cite[Prop. 2.5, Prop 2.9]{moseley2017equivariant})}\label{moseleyequivconsequences}
Let $X$ be an equivariantly formal $T$-space. Then there are surjections
\[ \Phi_{0}: H_{T}^{*}(X) \to H^{*}(X)\] 
sending $u$ to $0$ and 
\[ \Phi_{1}: H_{T}^{*}(X) \to H^{*}(X^{T})\]
sending $u$ to $1$, which induce the ring isomorphisms
\begin{align*}
    H^{*}_{T}(X)/\langle u \rangle & \cong H^{*}(X),\\
    H^{*}_{T}(X)/\langle u-1 \rangle & \cong H^{*}(X^{T}).
\end{align*}
If a group $W$ acts on $X$ and commutes with the action of $T$, then there is a $W$-action on $H^{*}_{T}(X)$ which fixes $u$, and the maps $\Phi_{0}$ and $\Phi_{1}$ are $W$-equivariant. 
\end{prop}
One corollary to Proposition \ref{moseleyequivconsequences} is that there is a natural filtration on $H^{*}(X^{T})$ coming from the cohomological grading on $H^{*}_{T}(X)$. Note that $u$ has cohomological degree $2$ in $H^{*}_{T}(X)$.

\begin{definition}[Equivariant filtration]\label{def:equivariantfiltration} \rm
Suppose $X$ is an equivariantly formal $T$-space, so that $H^{*}_{T}(X)/\langle u-1 \rangle \cong H^{*}(X^{T})$.
The \emph{equivariant filtration} on $H^{*}(X^{T})$ is then
\[ F_0(X) \subset F_1(X) \subset \cdots \subset  H^{*}(X^{T}), \]
where $F_k(X)$ is the image under $\Phi_1$ of the classes in $H_{T}^{*}(X)$ of cohomological degree at most $k$.

Write the associated graded ring with respect to the equivariant filtration as $\gr_{T} (H^{*}(X^{T}))$.
\end{definition}

\begin{cor}{ \rm (Moseley, \cite[Cor. 2.7]{moseley2017equivariant}}\label{cor:associatedgradedinequivhom}
There is an isomorphism 
\[ \gr_{T}(H^{*}(X^{T})) \cong H^{*}(X). \]
\end{cor}
Corollary \ref{cor:associatedgradedinequivhom} has a similar flavor to Corollaries \ref{cor:gr1and3} and \ref{cor:liftgr1and3}. This connection will be formalized in Corollary \ref{cor:filtrationscoincide} and Theorem \ref{thm:liftedfiltrationscoincide}.

It is useful to see Proposition \ref{moseleyequivconsequences} and Corollary \ref{cor:associatedgradedinequivhom} in the context of Example \ref{ex:rthree}.
\begin{example} \rm
Continuing Example \ref{ex:rthree}, consider first the subspace $\rthree$ fixed by the rotation action of $T = U(1)$. Since $T$ rotates around the $x$-axis, 
\[ (\rthree)^T = \{ (x,0,0): x \in \R \setminus \{ 0 \} \} = Z^{+} \sqcup Z^{-}. \]
By Example \ref{ex:rthree} and Proposition \ref{moseleyequivconsequences}:
\begin{align}
H_{T}^{*}(\rthree) / \langle u \rangle & \cong \Q[y]/ (y^2) \cong H^{*}(\rthree), \label{eq:ex1}\\
 H_{T}^{*}(\rthree) / \langle u - 1 \rangle & \cong \Q[y]/ (y^2 - y) \cong H^{*}(Z^{+} \sqcup Z^{-}) \label{eq:ex2}.
\end{align}
The first isomorphism in \eqref{eq:ex1} follows because $f': y-x \mapsto u$ so we identify $y-x = 0$ in $H_{T}^{*}(\rthree) / \langle u \rangle$. The first isomorphism in \ref{eq:ex2} is because $y-x = 1$ in $H_{T}^{*}(\rthree) / \langle u-1 \rangle$, so 
\begin{align*}
    0 = xy = (y-1)y = y^2 - y.
\end{align*}
The second isomorphisms in both \eqref{eq:ex1} and \eqref{eq:ex2} are applications of Proposition \ref{moseleyequivconsequences}, but can be verified directly by computing $H^{*}(\rthree)$ and $H^{*}(Z^{+} \sqcup Z^{-})$. The former is standard; the latter is concentrated in degree 0, and can be thought of as the $\Q$-vector space of linear functionals on $Z^{+} \sqcup Z^{-}$. One can then understand $y$ as a single Heaviside function $y: Z^{+} \sqcup Z^{-} \to \Z$ given by
\[ y(p): = \begin{cases} 1 & p \in Z^{+}\\
0 & p \in Z^{-}.
\end{cases}\]
It follows naturally that $y^{2} = y$ and that $y$ is the only function needed to generate this space. (This line of reasoning should feel reminiscent of \S \ref{sec:d1}!) 

Finally, note that although $\Q[y]/(y^2 -y)$ is not graded, it has an ascending filtration by degree, and the associated graded ring with respect to this filtration is $\Q[y]/(y^{2})$. Corollary \ref{cor:associatedgradedinequivhom} says 
\[ \gr_T (H_{T}^{*}\rthree) \cong H^{*}(\rthree), \]
and so we see that {\bf{the two associated graded rings---e.g. coming from the equivariant filtration and the filtration by degree in the Heaviside functions---coincide}}. One way of interpreting this is that the presentation of $H^{*}(Z^{+} \sqcup Z^{-})$ using Heaviside functions is particularly natural with respect to equivariant cohomology. We will see this phenomena again in the case of $\bconfhom^1$ and $\bconfhomlift^1$.
\end{example} 

\subsection{Applications to $\bconf^3$ and $\bconflift^3$}\label{equivtop:equivformality}
Our goal is to apply Proposition \ref{moseleyequivconsequences} and Corollary \ref{cor:associatedgradedinequivhom} to $\bconf^3$ and its lift, $\bconflift^3$. To do so, we must identify a torus action and prove that it commutes with the respective actions of $B_n$ and $B_{n+1}$.

\begin{prop}\label{prop:U1acts}
There is an action by the circle group $T = U(1)$ on $\bconf^3$ and $\bconflift^3$ which commutes with the actions of $B_{n}$ and $B_{n+1}$ respectively, and has fixed spaces
\begin{align*}
    (\bconf^3)^{U(1)} &= \bconf^1\\
    (\bconflift^3)^{U(1)} &= \bconflift^1.
\end{align*}
\end{prop}
\begin{proof}
Identify $\R^{3}$ with $\C \oplus \R$, and let $\omega \in U(1)$ act on $(z,x) \in \C \oplus \R$ by 
\[ \omega \cdot (z,x) = (\omega z, x). \]
Let $|(z,x)|$ be the magnitude of $(z,x)$, thought of as a vector in $\R^3$. Importantly,
\[ |(\omega z,x)|^{2} = |(z,x)|^{2}.\]
This shows that the action by $\omega$ commutes with the action by $\varphi$:
\begin{align*}
    \omega \cdot (\varphi \cdot (z,x)) &= \omega \cdot \frac{-(z,x)}{|(z,x)|^{2}} = \frac{-(\omega z, x)}{|(z,x)|^{2}} = \varphi (\omega z, x).\end{align*}
The $S_{n}$-action on $\bconf^3$ by coordinate permutation commutes with the action by $U(1)$ in a more straightforward way, and hence the action of $B_n$ commutes with the $U(1)$ action. The fixed space of the $U(1)$ action on $\C \oplus \R$ is $\{ 0 \} \oplus \R$, from which it follows that 
\[ (\bconf^3)^{U(1)} = \{ (p_{1}, \cdots, p_{n}) \in (\{ 0 \} \oplus (\R \setminus \{ 0 \}))^{n}\subset (\C \oplus \R)^{n}: p_{i} \neq p_{j}, p_{i} \neq \varphi(p_{j}) \}  \cong \bconf^1.\]

To study the action of $U(1)$ on $\bconflift^3$, first embed $U(1)$ into $SU_{2}$ via the isomorphism
\[ U(1) \cong \left\{ \begin{pmatrix}
\omega & 0\\
0 & \overline{\omega} 
\end{pmatrix}: \omega = e^{2\pi i \theta} \right\}.\]
From this identification it follows that $U(1)$ acts on $SU_{2}$ on the \emph{right} by multiplication, inducing a diagonal action on $\bconflift^3$. (Recall that $SU_{2}$ acts on the left in $\bconflift^3$). The action by $U(1)$ commutes with the $\Z_{2}$ antipodal action on $\bconflift^3$ (e.g. left multiplication by $-1$) because $-1$ is in the center of $SU_{2}$. Since $U(1)$ acts diagonally, it must also commute with the action of $S_{n+1}$ on $\bconflift^3$ by coordinate permutation. Hence the $U(1)$ action commutes with the $B_{n+1}$-action on $\bconflift^3$. To determine the fixed space of the $U(1)$-action, note that $g \in SU_{2}$ commutes with every element of $U(1)$ if and only $g \in U(1)$.
Thus in $\bconflift^3$
\[(p_{1}, \cdots, p_{n}) \sim (p_{1}\omega, \cdots, p_{n} \omega) \]
for every $\omega \in U(1)$ if and only if $p_{i} \in U(1)$ forall $1 \leq i \leq n$, and so 
\[ (\bconflift^3)^{U(1)} = \bconflift^1. \]
\end{proof}
As a Corollary, we immediately obtain the following:
\begin{cor}\label{cor:applymoseley}
The spaces $\bconf^3$ and $\bconflift^3$ are equivariantly formal with respect to the action by $T = U(1)$. Hence
\begin{itemize}
    \item $H^{*}_{T}(\bconf^3)$ and $H^{*}_{T}(\bconflift^3)$ are $\Q[u]$ modules;
    \item There is an action of $B_{n}$ (resp. $B_{n+1}$) on $H^{*}_{T}(\bconf^3)$ (resp. $H^{*}_{T}(\bconflift^3)$) that fixes $u$;
    \item The maps $\Phi_{0}$ and $\Phi_{1}$ as in Proposition \ref{moseleyequivconsequences} are surjective and $B_{n}$ (resp. $B_{n+1}$) equivariant;
    \item There are isomorphisms
    \begin{align*}
    \gr_{T}(\bconfhom^1) &\cong_{B_{n}} \bconfhom^3\\
    \gr_{T}(\bconfhomlift^1) &\cong_{B_{n+1}} \bconfhomlift^3,
\end{align*}
where here the associated graded ring is with respect to the equivariant filtration.
\end{itemize}
\end{cor}
\begin{proof}
The equivariant formality follows from the fact that both $\bconf^3$ and $\bconflift^3$ have cohomology concentrated in even degrees. The remainder of the statement follows by applying Corollary \ref{cor:associatedgradedinequivhom}.
\end{proof}

The equivariant setting allows us to give a description of the total representations of $\bconfhom^3$ and $\bconfhomlift^3$ by applying Proposition~\ref{prop:ungradedpt1} and adapting techniques used in \cite{moseley2017equivariant} and \cite{moseley2017orlik}.

\begin{theorem}
\label{ungradedrep}
As ungraded representations
\begin{align*}
\bconfhom^3 & \cong_{B_{n}} \Q[B_{n}], \\
  \bconfhomlift^3 &  \cong_{B_{n+1}} \ind_{\langle c \rangle }^{B_{n+1}} \triv = \Q[B_{n+1}/\langle c \rangle ],
\end{align*}
where $c$ is a Coxeter element of $B_{n+1}$ and $\triv$ is the trivial representation.
\end{theorem}
\begin{proof}
By Proposition~\ref{prop:ungradedpt1}, $\bconflift^1$ is isomorphic to the $B_{n+1}$ action on the cosets $B_{n+1} / \langle c \rangle.$
%\[ \bconflift^1 \cong B_{n+1} / \langle c \rangle.\]
Since $\bconflift^1$ is a disjoint union of contractible pieces, $H^{0} \bconflift^1 =\bconfhomlift^1 $ and so passing to cohomology simply gives
\[ \bconfhomlift^1 \cong_{B_{n+1}} \Q[B_{n+1}/\langle c \rangle].  \]

Proposition~\ref{prop:ungradedpt1} then says that the $B_n$ action on the set of connected components of $\bconflift^1$ is simply transitive. Since $\bconflift^1$ is $B_n$-equivariantly homeomorphic to $\bconf^1$, this implies that $B_n$ also acts simply transitively on the set of connected components of $\bconf^1$. 
%\[\bconf^1 \cong B_n. \]
Thus in cohomology, 
\[ H^{0}\bconf^1 = \bconfhom^1 = \Q[B_n]. \]

Finally, passing to the associated graded ring $\gr_{T}(\bconfhomlift^1) = \bconfhomlift^3$ and $\gr_{T}(\bconfhom^1) = \bconfhom^3$ will not change the isomorphism type, so the claim follows.
\end{proof}

\subsection{Presentation of equivariant cohomology}
Having established some basic properties of $H_{T}^{*}(\bconf^3)$ and $H_{T}^{*}(\bconflift^3)$, we turn to computing their presentations. Our arguments will closely resemble Moseley \cite{moseley2017equivariant}. 

Recall that the equivariant formality of $\bconf^3$ implies that
\begin{enumerate}
    \item $H_{T}^{*}(\bconf^3) \cong \bconfhom^3 \otimes \Q[u]$ as a $\Q[u]$-module, where $u$ is the image of $H_{T}^{2}(\pt)$ under the map induced from $\bconf^3 \to \pt$, and 
    \item Any generating set for $\bconfhom^3$ can be lifted to a generating set of $H^{*}_{T}(\bconf^3)$ over $\Q[u]$.
\end{enumerate}

We will take a slight shortcut here and use the smaller generating set for $\bconfhom^3$ discussed in Remark \ref{rmk:difpres}: in particular, pick generators of $\bconfhom^3$ as $z_{ij}^{+}, z_{ij}^{-},$ and $z_{i}$ for integers $1 \leq i < j \leq n$. Since the maps $\Phi_{0}: H_{T}^{*}(\bconf^3) \to \bconfhom^3$ and $\Phi_{1}: H_{T}^{*}(\bconf^3) \to \bconfhom^1$ are $B_{n}$-equivariant by Corollary \ref{cor:applymoseley}, the $B_{n}$ action on the generators in $H_{T}^{*}(\bconf^3)$ is inherited from these spaces (see Table \ref{table:actionsimplifiedpres}). 

We would first like to relate these generators to $T$-submanifolds of $\bconf^3$. Recall the definition of the hyperplanes
\begin{align*}
    H_{ij}:=& \{(x_{1}, \cdots, x_{n}) \in \R^{n}: x_{i} \neq x_{j} \}\\
    H_i :=&  \{(x_{1}, \cdots, x_{n}) \in \R^{n}: x_{i} \neq 0 \}
\end{align*}
and that the action of $t_i$ on $\bconf^3$ is
\[ t_i \cdot (x_{1}, \cdots, x_i, \cdots, x_{n}) =(x_{1}, \cdots, \varphi(x_i), \cdots, x_{n}). \]

Let $w_{i} = (x_{i}, y_{i}, z_{i}) \in \rthree$. 
For integers $i,j \in [n]$, define the maps from $(\rthree)^n \to \R^3$ by
\begin{align*}
    \omega_{ij}: (w_{1}, \cdots, w_{n}) &\longmapsto w_{j} - w_{i},\\
  %  &\\
    \omega_{\ijminus}: (w_{1}, \cdots, w_{n}) &\longmapsto \varphi(w_{j}) - w_{i},\\
   % &\\
     \omega_{i}: (w_{1}, \cdots, w_{n}) &\longmapsto w_{i}.
\end{align*}

These maps are not linear. However, they are $T$-equivariant and have the property that  $(\omega_{ij})^{-1}(\R^{3} \setminus \{ 0 \}) = \R^{3n} - \left( H_{ij} \otimes \R^{3}\right)$ (and similarly for $\omega_{i}$). Note that
\[ t_{j} \cdot  (\omega_{ij})^{-1}(\R^{3} \setminus \{ 0 \}) =  (\omega_{\ijminus})^{-1}(\R^{3} \setminus \{ 0 \})\] 
and that
\[ \bconf^3 = \bigcap_{1 \leq i \neq j \leq n} \left( (\omega_{ij})^{-1}(\R^{3} \setminus \{ 0 \}) \cap (\omega_{\ijminus})^{-1}(\R^{3} \setminus \{ 0 \}) \right) \cap \bigcap_{1 \leq i \leq n} \left( \omega_{i}^{-1}(\R^{3} \setminus \{ 0 \}) \right).\]

For ease of notation, write $\omega_{J}$ to refer to an arbitrary $\omega$-map, so $\omega_{J}$ is the place-holder for $\omega_{ij}$ for $i,j \in [n]^{\pm}$ and $\omega_{i}$ for $i \in [n]$. Define
\begin{align*}
    Y_{J}^{+}:=& \omega_{J}^{-1}(Z^{+}) \\
    Y_{J}^{-}:=& \omega_{J}^{-1}(Z^{-}),
\end{align*}
where $Z^{+}$ (resp. $Z^{-}$) is the positive (resp. negative) part of the $x$-axis in $\R^{3} \setminus \{ 0 \}$. This induces classes
$\omega_{J}^{*}([Y^{+}]_{T})$ and $\omega_{J}^{*}([Y^{-}]_{T})$ in $H^{*}_{T}(\bconf^3)$. By Example \ref{ex:rthree}, and functoriality,
\[\omega_{J}^{*}([Y^{-}]_{T}) = \omega_{J}^{*}([Y^{+}]_{T}-u) = \omega_{J}^{*}([Y^{+}]_{T}) - u.\]

We are now ready to obtain relations in $H^{*}_{T}(\bconf^3)$. Our primary method will be to show that $Y_{J}^{\pm} \cap Y_{J'}^{\pm} = \emptyset$ for certain $J, J' \subset [n]^{\pm}$, which by Remark \ref{rmk:intersectionnumbers} will imply that the corresponding classes in equivariant cohomology multiply to 0. 

\begin{lemma}
\label{lem:eqrel}
For $1 \leq i < j \leq n$, let
\begin{align*} \Psi: \Q[z_{ij}^{+}, z_{ij}^{-}, z_{i},u] &\longrightarrow H^{*}_{T}(\bconf^3) %\\
\end{align*}
be the map sending 
\[   z_{ij}^{+} \mapsto \omega_{ij}^{*}([Z^{+}]_{T}) \hspace{2em} z_{ij}^{-} \mapsto \omega_{\ijminus}^{*}([Z^{+}]_{T}) \hspace{2em} z_{i} \mapsto \omega^{*}_{i}([Z^{+}]_{T}) \hspace{2em} u \mapsto u.\]
Then $\Psi$ is surjective and the following relations lie in $\ker(\Psi)$:
\begin{align}
    \label{ER0}0  &= z_{ij}^{+}(z_{ij}^{+} - u) = z_{ij}^{-}(z_{ij}^{-}-u) = z_{i}(z_{i}-u)\\
        \label{ER1}0 &= u^{-1}\left( z_{ij}^{+}z_{jk}^{+}(z_{ik}^{+}-u) - (z_{ij}^{+}-u)(z_{jk}^{+}-u)z_{ik}^{+} \right)\\
    \label{ER2}0 &=u^{-1}\left( z_{ij}^{+}z_{i}(z_{j}-u) - (z_{ij}^{+}-u)(z_{i}-u)z_{j} \right)\\
    \label{ER3}0 &=u^{-1}\left( z_{ij}^{+}z_{ij}^{-}(z_{i}-u) - (z_{ij}^{+}-u)(z_{ij}^{-}-u)z_{i} \right)
\end{align}
Note that relations \eqref{ER1}---\eqref{ER3} are all polynomials (i.e. the expression in the parentheses has a factor of $u$).
\end{lemma}

\begin{proof}
Because $\bconf^3$ is formal and $z_{ij}^{+}, z_{ij}^{-}$ and $z_{i}$ for $1 \leq i < j \leq n$ is a generating set for $\bconfhom^3$, it follows that $z_{ij}^{+}, z_{ij}^{-}$, $z_{i}$ and $u$ form a generating set for $H^{*}_{T}(\bconf^3)$. Hence $\Psi$ is surjective.

For \eqref{ER0}, 
    \[ \Psi(z_{J}(z_{J}-u)) =  \omega_{J}^{*}([Z^{+}]_{T}) \,\, \omega_{J}^{*}([Z^{-}]_{T}) = \omega_{J}^{*}(0) =0\]
    because $Z^{+} \cap Z^{-} = \emptyset$.
    
For $1 \leq i < j < k \leq n,$ consider 
\begin{align*} Y_{ij}^{+}  &= \{ (w_{1}, \cdots, w_{n}): x_{i} < x_{j} \}, \\
Y_{jk}^{+}  &= \{ (w_{1}, \cdots, w_{n}): x_{j} < x_{k} \},\\
Y_{ik}^{-}  &= \{ (w_{1}, \cdots, w_{n}): x_{k} < x_{i} \}.
\end{align*}
Then $Y_{ij}^{+} \cap Y_{jk}^{+} \cap Y_{ij}^{-} = \emptyset$ since we cannot have $x_{i} < x_{j} < x_{k} < x_{i}$. This implies that 
\[ \Psi(z_{ij}^{+} z_{jk}^{+} (z_{ik}^{+} - u)) = \omega^{*}_{ij}([Z^{+}]_{T}) \,\, \omega^{*}_{jk}([Z^{+}]_{T}) \,\,\omega^{*}_{ik}([Z^{-}]_{T}) = 0, \]
and analogously, 
\[ \Psi((z_{ij}^{+}-u) (z_{jk}^{+}-u) z_{ik}^{+}) = \omega^{*}_{ij}([Z^{-}]_{T}) \,\, \omega^{*}_{jk}([Z^{-}]_{T}) \,\, \omega^{*}_{ik}([Z^{+}]_{T}) = 0. \]
The expansion of the above expression has a factor of $u$; removing this factor recovers \eqref{ER1}.

Relations \eqref{ER2} and \eqref{ER3} follow similarly. For the former, $Y_{ij}^{+} \cap Y_{i}^{+} \cap Y_{j}^{-} = \emptyset$ since we cannot have $0 < x_i < x_j < 0$. By the same argument, $Y_{ij}^{-} \cap Y_{i}^{-} \cap Y_{j}^{+} = \emptyset$ and following the logic of \eqref{ER1}, we obtain \eqref{ER2}. 

For \eqref{ER3}, elements in $Y_{j}^{+} \cap Y_{i\overline{j}}^{+}$ must have $x_i < 0$ since $0 < x_j$ implies $\varphi(x_{j}) < 0$, From this it follows that $Y_{j}^{+} \cap Y_{i\overline{j}}^{+} \cap Y_{ij}^{-} = \emptyset$ since we cannot have $0 < x_j < x_i < 0$. Analogously $Y_{j}^{-} \cap Y_{i\overline{j}}^{-} \cap Y_{ij}^{+} = \emptyset$, and \eqref{ER3} follows. 
\end{proof}
We now use these relations to give a presentation for $H_{T}^{*}(\bconf^3)$.
\begin{theorem}\label{eqcohompres}
Let $\mathcal{K}'$ be the ideal generated by $B_{n}$-images of the relations in Lemma \ref{lem:eqrel}. Then 
\[ H^{*}_{T}(\bconf^3) = \Z[ z_{ij}^{+}, z_{ij}^{-}, z_{i},u: 1 \leq i<j \leq n] / \K'. \]
\end{theorem}
\begin{proof}
Lemma \ref{lem:eqrel} shows that $\Psi$ induces a surjective map 
\[ \overline{\Psi}: \Q[z_{J}, u] / \K' \to H_{T}^{*}(\bconf^3).\]
We would like to show $\overline{\Psi}$ is injective as well.

Recall that $\Phi_0$ is the map defined by sending $u$ to $0$. By Corollary \ref{cor:applymoseley}, applying $\Phi_{0}$ to $ H^{*}_{T}(\bconf^3)$ gives a surjective, $B_{n}$ equivariant map to $\bconfhom^3$. Applying $\Phi_{0}$ to $\Z[z_{J},u]$ via
\[ \Phi_{0}: \Z[z_{J}, u] \longrightarrow \Z[z_{J}]\]
we see that $\Phi_{0}(\K^{'}) = \mathcal{K}$ from Remark \ref{rmk:difpres}. (Note that the image of relations \eqref{ER1}, \eqref{ER2}, \eqref{ER3} under $\Psi_{0}$ is the negation of the relations $(i), (ii)$ and $(iii)$ in $\bconfhom^3$ given in Remark \ref{rmk:difpres}.)

It follows that we have a commutative diagram with exact bottom row and surjective columns:
\[ \begin{tikzcd}
&\K' \arrow{r}{} \arrow[swap]{d}{\Phi_{0}} & \Z[z_{J}, u] \arrow{r}{\Psi} \arrow{d}{\Phi_{0}} & H_{T}^{*}(\bconf^3) \arrow{r}{} \arrow{d}{\Phi_{0}}& 0 \\%
0\arrow{r}{}& \K \arrow{r}{}& \Z[z_{J}]\arrow{r}{} & \bconfhom^3 \arrow{r}{} & 0.
\end{tikzcd}
\]
The same diagram chase argument used by Moseley in \cite[Thm 4.6]{moseley2017equivariant} then applies to show that $\ker(\Psi) = \K'$.
\end{proof}
One key takeaway from Theorem \ref{eqcohompres} is that (as in Example \ref{ex:rthree}) the Heaviside function presentation of $\bconfhom^1$ is quite natural with respect to equivariant cohomology. 
\begin{cor}\label{cor:filtrationscoincide}
The filtration by Heaviside-like functions $z_{ij}, z_i$ for $i,j \in [n]^{\pm}$ and the equivariant filtration on $\bconfhom^1$ coincide. Thus
\[H^{*}_{T}(\bconf^{3})/\langle u \rangle = \gr_{T}(\bconfhom^1) \cong \gr(\bconfhom^1) \cong \bconfhom^3. \]
\end{cor}
\begin{proof}
This follows from noting that setting $u=1$ in $H^{*}_{T}(\bconf^3)$ recovers the presentation of $\bconfhom^1$ in Corollary \ref{cor:grbconfhom}. 
\end{proof}

\subsection{Lifting to $\bconflift^3$}
We would like make an analogous statement to Corollary \ref{cor:filtrationscoincide} for $H_{T}^{*}(\bconflift^3)$ by applying two key facts established earlier:
\begin{enumerate}
    \item there is a $B_n$-equivariant homeomorphism between $\bconf^1$ and $\bconflift^1$ and
    \item $\gr_{T}(\bconf^1)$ coincides with $\gr(\bconf^1)$.
\end{enumerate}
Using a similar argument to \cite[Remark 2.9]{moseley2017orlik}, this is enough to establish the desired connection. 
\begin{theorem}\label{thm:liftedfiltrationscoincide}
The filtration by signed cyclic Heaviside functions $y_{ijk}$ for $i,j,k \in [n]_{0}^{\pm}$ and the equivariant filtration on $\bconfhomlift^1$ coincide. Thus
\[H^{*}_{T}(\bconflift^{3})/\langle u \rangle = \gr_{T}(\bconfhomlift^1) \cong \gr(\bconfhomlift^1) \cong \bconfhomlift^3. \]
\end{theorem}
\begin{proof}
The key idea here is to recognize that the homeomorphism between $ \bconf^1 $ and $ \bconflift^1$ derived from Proposition \ref{hiddenactionframework} is a special case ($\ell=0$) of a more general family of homeomorphisms for $0 \leq \ell \leq n$:   
\begin{align*}
    f_\ell: \bconflift^1 &\longrightarrow \bconf^1\\
    (p_{0}, \cdots, p_{n})& \longmapsto (p_{\ell}^{-1}p_{0}, \cdots, p_{\ell}^{-1}p_{\ell-1}, p_{\ell}^{-1}p_{\ell+1}, \cdots p_{\ell}^{-1}p_{n}).
\end{align*}
In cohomology, this induces a family of maps $f_\ell^*$ sending $z_{ij}$ to $y_{\ell ij}$ and $z_i$ to $y_{\ell \overline{\ell}i}$. By Corollary \ref{cor:associatedgradedinequivhom}, in $\bconf^1$ the filtration by Heaviside-like functions coincides with the filtration arising from equivariant cohomology. It follows that for a fixed $\ell$, the same must be true of $\bconfhomlift^1$ with respect to the filtration arising from the $y_{\ell ij}$ and the filtration arising from $H_{T}^{*}(\bconflift^3)$. Since the latter filtration is stable under the action of $B_{n+1}$, it must coincide with the filtration by \emph{all} signed cyclic Heaviside functions (e.g. allowing the indices $i,j, k \in [n]_{0}^{\pm}$ to vary.)
\end{proof}

Since we have already obtained a description of $\gr(\bconfhomlift^1)$ by Corollary \ref{cor:lifthom1gr}, we obtain a description of $\bconfhomlift^3$ as a Corollary. 

\begin{cor}\label{thm:bpres3lift}
The ring $\bconfhomlift^3$ has presentation
\[ \bconfhomlift^3 \cong \Z[y_{ijk}]/ \I',\]
for distinct $i,j,k \in [n]_{0}^{\pm}$, where $\I'$ is generated by the relations
\[ (i) \hspace{.5em} y_{ijk}^{2} , \hspace{2em}  (ii) \hspace{.5em} y_{ijk} - y_{ij\ell} + y_{ik \ell} - y_{jk\ell}  \hspace{2em}   (iii) \hspace{.5em} y_{ \overline{i} \, j \, k}- y_{i \, \overline{j \, k}} \hspace{2em} \]
\[ (iv) \hspace{.5em} y_{ijk} = y_{jik}  \hspace{2em} (v) \hspace{.5em} y_{ijk}y_{ik \ell } - y_{ijk} y_{ij \ell } - y_{ik \ell } y_{ij \ell }.  \]
\end{cor}

\section{Connection to the Mantaci-Reutenauer algebra}\label{sec:MRalg}
Finally, we will complete the Type $B$ story by studying the family of $B_{n}$-representations arising in the
\emph{Mantaci--Reutenauer algebra} (Definition~\ref{def:MR}), introduced in \cite{mantaci1995generalization}. Our primary goal is to prove Theorem~\ref{thm:BRiso} relating the topological spaces from \S \ref{sec:d1}--\ref{sec:equivcohom} to these representations. In order to do so, we will need to define a family of idempotents in the Mantaci-Reutenauer algebra (\S \ref{sec:vazidem}) and the representations they generate (\S \ref{sec:vazrep}), as well as decompose the ring $\G$ further by \emph{signed set compositions} (\S \ref{sec:decompsignedset}).

 An \emph{integer composition} of $n$ is a sequence $(a_{1}, \cdots, a_{\ell})$ where $a_{i} \in [n]$ and $|a_{1}| + \cdots + |a_{\ell}| = n$. Similarly 
 a \emph{signed integer composition} of $n$ will be a sequence $(a_{1}, \cdots, a_{\ell})$ where $a_{i} \in [n]^{\pm}$ and again $|a_{1}| + \cdots + |a_{\ell}| = n$. 

\begin{definition}\label{def:MRdes} \rm
For $\sigma$ in one-line notation, the \emph{Mantaci-Reutenauer descent set} of $\sigma$ is
\[ \mrdes(\sigma) := \begin{cases}i \in [n-1]:  &|\sigma_{i}| > |\sigma_{i+1}| \textrm{ and } \sigma_{i} \textrm{ and } \sigma_{i+1} \textrm{ have the same sign or} \\
&\sigma_{i} \textrm{ and } \sigma_{i+1} \textrm{ have opposite signs.} \end{cases}
\]
\end{definition}
Note that $\mrdes(\sigma)$ partitions $\sigma$ into $\ell:= |\mrdes(\sigma)|+1$ many blocks, $b_{1}, \cdots, b_{\ell}$ of size $m_{1}, \cdots, m_{\ell}$. By construction every element in $b_{i}$ will have the same sign; let 
\[ \sgn(m_{i}) = \begin{cases} m_{i} & \textrm{ if }b_{i} \subset [n] \\
\overline{m_{i}} & \textrm{ if }b_i \subset [n]^{-}.
\end{cases}\]
Then the \emph{shape} of $\sigma \in B_n$ is the signed integer composition
\[ \sh(\sigma):= (\sgn(m_{1}), \cdots, \sgn(m_{\ell})). \]
\begin{example} \rm
If $\sigma = (3,4,\overline{1},\overline{5},\overline{2}) \in B_5$, then $\mrdes(\sigma) = \{ 2, 4 \}$, which partitions $\sigma$ into ordered blocks $(\{3,4\}, \{\overline{1}, \overline{5} \}, \{\overline{2}\})$. Therefore $\sh(\sigma) = (2, \overline{2}, \overline{1})$.
\end{example}
\begin{remark} \rm
Recall that for $\sigma \in S_n$, the \emph{descent set} of $\sigma$ is  $\Des(\sigma) = \{ i \in [n-1]: \sigma_i > \sigma_{i+1}\}$. Because $S_n \leq B_n$, both $\mrdes(\sigma)$ and $\sh(\sigma)$ are well-defined for $\sigma \in S_n$. In particular, when $\sigma \in S_n$, $\mrdes(\sigma) = \Des(\sigma)$ and
$\sh(\sigma)$ will be an unsigned integer composition of $n$.
\end{remark}

\begin{definition}\label{def:MR} \rm
The \emph{Mantaci-Reutenauer algebra} is the algebra $\Sigma'[B_{n}]$ generated by $Y_{\alpha}$ in $\Q[B_{n}]$ where
\[ Y_{\alpha}:= \sum_{\substack{\sigma \in B_{n}\\ \sh(\sigma) = \alpha}} \sigma. \]
\end{definition}

The dimension of $\Sigma'[B_{n}]$ is $2 \cdot 3^{n-1}$; for other bases of $\Sigma'[B_{n}]$, see \cite{aguiar2004peak}.

\subsection{The Vazirani idempotents}\label{sec:vazidem}
In \cite{vazirani}, Vazirani introduces a complete family of orthogonal idempotents for $\Sigma^{'}[B_n]$. These idempotents generalize a family of idempotents $\{ \frake_{\lambda} \}_{\lambda \vdash n}$ introduced by Garsia--Reutenauer in \cite{garsia} for $\Sigma[S_n]$; see \cite[\S 3]{douglass2018decomposition} for complete details on both definitions. 

We will use an equivalent construction of Vazirani's idempotents given by Douglass--Tomlin in \cite[Prop. 2.5]{douglass2018decomposition}. This will require the following objects and maps:
\begin{itemize}
\item Let $\mathcal{C}(n)$ and $\mathcal{SC}(n)$ be the set of unsigned and signed integer compositions of $n$, respectively.
    \item For $p = (p_{1}, \cdots, p_{\ell}) \in \mathcal{SC}(n)$, let $|p| := (|p_{1}|, \cdots, |p_{\ell}|) \in \mathcal{C}(n)$. 
    \item For $p = (p_{1}, \cdots, p_{\ell}) \in \mathcal{SC}(n)$, define 
    \[ \widehat{p_{i}} = \sum_{j=1}^{i} |p_{j}| \in [n],\]
    and the map $\ \ \widehat{} \ \ $ sending signed compositions to subsets of $[n-1]$:
    \begin{align*}
        \widehat{} \ : \mathcal{SC}(n) &\longrightarrow 2^{[n-1]}\\
        p = (p_1, \cdots, p_\ell) & \longmapsto \widehat{p} = (\widehat{p_{1}}, \cdots, \widehat{p_{\ell-1}}).
    \end{align*}
    If we restrict the domain of $\ \ \widehat{} \ \ $ to ordinary compositions $\mathcal{C}(n)$, then $ \ \ \widehat{} \ \ $ defines a bijection, and allows us to identify the sets
    \[ \bigg\{ \sigma \in S_n: \sh(\sigma) = p \in \mathcal{C}(n) \bigg\}
    \quad = \quad
    \bigg\{ \sigma \in S_n: \Des(\sigma) = \ \widehat{p} \  \bigg\}.\]

    \item A set composition of $[n]$ is a partition of $[n]$ into disjoint blocks $(b_1, \cdots, b_\ell)$. Let $\Lambda_{[n]}$ be the set of such compositions.
    Define 
    \begin{align*}
        \Lambda: \mathcal{SC}(n) &\longrightarrow \Lambda_{[n]}\\
        p = (p_1, \cdots, p_{\ell}) &\longmapsto \Lambda(p)= (\Lambda(p_1), \cdots, \Lambda(p_{\ell})),
    \end{align*}
    where $\Lambda(p_i):=  \{ \widehat{p_{i-1}} +1, \cdots, \widehat{p_{i}}\}.$
    Note that $\widehat{p}$ and $\Lambda(p)$ differ because the former is a subset of $[n-1]$ while the latter is a genuine set composition of $[n]$.
\item Recall that a \emph{signed partition} of $n$ is a pair of partitions $\lambda = \sgnpart$ such that $\lambda^{+}$ is a partition of $n_1 \leq n$, $\lambda^{-}$ is a partition of $n_2 \leq n$, and $n_1 + n_2 = n$. Let $\ell(\lambda^{+})$ (resp. $\ell(\lambda^{-})$) be the number of parts in $\lambda^{+}$ (resp. $\lambda^{-}$), and $\ell(\lambda) = \ell(\lambda^{+}) + \ell(\lambda^{-})$. 
  \item For $p \in \mathcal{SC}(n)$, let $\overleftarrow{p}$ be the reordering of the parts of $p$ into a signed partition $(\lambda^{+}, \lambda^{-})$ so that $\lambda^{+}$ consists of all positive parts of $p$ in decreasing order and $\lambda^{-}$ consists of all negative parts of $p$ in decreasing order. 
\item For a set $A \subset [n-1]$, define an element in the Type $A$ Descent algebra $\Sigma[S_n]$:
    \[ X_{A}:= \sum_{\substack{w \in S_n\\ \Des(w) \subset A}} w . \]
Varying over all subsets $A \subset [n-1]$ gives a basis for $\Sigma[S_n]$.

    \item For $m > 0$, the \emph{Reutenauer idempotent} $r_{m} \in \Sigma[S_{m}] \subset \Q[S_{n}]$ is
    \[ r_{m} = \sum_{A \subset [m-1]} \frac{(-1)^{|A|}}{|A|+1} X_{A}. \]
    Note that the definition of $r_m$ can be extended to any ordered subset $J \subset [n]$ by replacing $[m]$ by $J$, in which case we will write $r_{J} \in \Q[S_{J}]$. Furthermore, since $S_n \leq B_n$, one also has $r_J \in \Q[B_n].$ 
    \item For $J \subset [n]$ the element $w_{0,J} \in B_n$ is the product
    \[ \prod_{i \in J} t_i. \]
    In other words, $w_{0,J}$ acts like $-1$ on $J$ and 1 off $J$.
    \item For $J \subset [n]$, define 
    \[ \epsilon_{J}^{\pm} := \frac{1}{2} (1 \pm w_{0,J}). \]
\end{itemize}
\begin{example}\label{ex:partition} \rm
Suppose $p = (2, \overline{2}, 1)$. Then $|p| = (2,2,1)$, $\widehat{p} = \{ 2, 4, 5 \}$ and $\Lambda(p) = \{ \{ 1,2 \}, \{3, 4 \}, \{ 5 \} \}.$ Finally, 
\[ \overleftarrow{p} =  \big( \underbrace{(2,1)}_{\lambda^{+}},\underbrace{(2)}_{\lambda^{-}} \big), \] 
so $\ell(\lambda^{+}) = 2$, $\ell(\lambda^{-}) = 1$.
\end{example}
We are at last ready to define the Vazirani idempotents.
\begin{definition} \label{def:vazidem} \rm (Douglass-Tomlin \cite[Prop 2.5]{douglass2018decomposition}).
Given $p = (p_1, \cdots, p_\ell) \in \mathcal{SC}(n)$, define the element 
\[ I_p := X_{\widehat{p}} \cdot \epsilon^{\zeta_{1}}_{\Lambda(p_1)}\cdot r_{\Lambda(p_1)} \cdots  \epsilon^{\zeta_{\ell}}_{\Lambda(p_\ell)}\cdot r_{\Lambda(p_\ell)},  \]
where $\zeta_i$ is the sign of $p_i$.
Then for each signed partition $(\lambda^{+}, \lambda^{-})$, the \emph{Vazirani idempotent} is
\[ \frakg_{\sgnpart}:= \sum_{\substack{p \in \mathcal{SC}(n) \\ \overleftarrow{p} = \sgnpart}} \frac{1}{ \ell(\lambda)!} I_p.\]
\end{definition}

The Vazirani idempotents extend the Garsia-Reutenauer idempotents in $\Sigma[S_{n}]$, which are defined as:

\begin{equation}\label{eq:GRidem}
\frake_{\lambda} = \frac{1}{\ell(\lambda)!} \sum_{\substack{p \in \mathcal{C}(n)\\ \overleftarrow{p} = \lambda}} X_{\widehat{p}} \cdot r_{\Lambda(p_1)} \cdots r_{\Lambda(p_\ell)}. \end{equation}

In fact, we can make the relationship between the $\frakg_{\sgnpart}$ and $\frake_{\lambda}$ precise. In \cite{aguiar2004peak}, Aguiar--Bergeron--Nyman study the surjection $\tau: B_n \to S_n$ that forgets the signs of $\sigma \in B_n$. For example, if $\sigma = (\overline{2},1,\overline{3})$, then $\tau(\sigma) = (2,1,3)$.

\begin{theorem}{\rm (Aguiar--Bergeron--Nyman \cite[Prop 7.5]{aguiar2004peak}). }
The map $\tau$ is an algebra homomorphism from
\[\tau: \Q[B_n] \to \Q[S_n]. \]
When restricted to $\Sigma^{'}[B_n]$, the map $\tau$ surjects onto $\Sigma[S_n]$.
\end{theorem}

This will allow us to precisely relate the $\frakg_{\sgnpart}$ to the $\frake_{\lambda}$; the author is grateful to M. Aguiar for suggesting this line of inquiry.
\begin{prop}\label{sgnmap}
For a signed partition $\sgnpart$, one has
\[ \tau(\frakg_{\sgnpart}) = \begin{cases} \frake_{\lambda^{+}} & \lambda^{-} = \emptyset \\
0 & \text{otherwise.}
\end{cases}\]
\end{prop}
\begin{proof}
Since $\tau$ is an algebra homomorphism, it is enough to consider how $\tau$ maps each term in $I_p$ for $\overleftarrow{p} = \sgnpart$. In particular, for any $\Lambda(p_i)$, since $\tau(w_{0,P_i}) = 1$, it follows that 
\[ \tau(\epsilon_{\Lambda(p_i)}^{\zeta_i}) = \begin{cases} 1 & \zeta_i = +\\
0 & \zeta_i = -.
\end{cases}\]
Thus $\tau(I_p) = 0$ if $p$ has any negative parts, from which it follows that $\tau(\frakg_{\sgnpart}) = 0$ if $\lambda^{-} \neq \emptyset$. On the other hand, if $\overleftarrow{p} = (\lambda^{+}, \emptyset)$, then $p \in \mathcal{C}(n)$, $|p| = p$ and $\tau(\epsilon^{+}_{P_i}) = 1$ for all $i$. Hence $\tau(\frakg_{(\lambda^{+}, \emptyset)}) = \frake_{\lambda^{+}}$ by \eqref{eq:GRidem}.
\end{proof}

In \cite{garsia}, Garsia and Reutenauer show that the definition of the Type $A$ Eulerian idempotents given in \eqref{eq:typeaeulerian} is equivalent to 

\begin{equation}\label{eq:alttypeaeulerian} \frake_k = \sum_{\substack{\lambda \vdash n \\
\ell(\lambda) = k-1}} \frake_{\lambda}. \end{equation}

Our analog of the Eulerian idempotents will take inspiration from \eqref{eq:alttypeaeulerian}: 
\begin{definition}\label{def:frakg} \rm
For $0 \leq k \leq n$, define the idempotent in $\Sigma'[B_n]$
\begin{equation}\label{eq:Bidem} \frakg_k:= \sum_{\substack{\lambda = \sgnpart\\ \ell(\lambda^{+}) = k}} \vazidem.
\end{equation}
\end{definition}

Proposition \ref{sgnmap} implies the following relationship between the $\frakg_k$ and the $\frake_k$.
\begin{cor}
The map $\tau: \Sigma'[B_n] \to \Sigma[S_n]$ sends 
\[ \tau(\frakg_{k}) = \begin{cases} 0 & k=0\\
\frake_{k-1} & k > 0.
\end{cases} \]
\end{cor}
\subsubsection{Representations generated by $\frakg_{\sgnpart}$ and $\frakg_k$} \label{sec:vazrep}
The idempotents in Definitions \ref{def:vazidem} and \ref{def:frakg} generate families of $B_{n}$-representations.
\begin{definition}\label{def:vazrep} \rm
For any signed partition $\sgnpart$ of $n$, define the $B_{n}$-representation
\[ \vazrep:= \vazidem \Q[B_n]. \]
Further, for $0 \leq k \leq n$ define the $B_{n}$-representation
\[ G_n^{(k)}:= \frakg_k \Q[B_n]. \]
\end{definition}

Douglass--Tomlin proved in \cite{douglass2018decomposition} that the characters of the $\vazrep$ could be described as induced representations from one-dimensional characters of certain centralizers. In order to describe and analyze these induced characters, we will temporarily extend scalars from $\Q$ to $\C$.

 Recall from Definition~\ref{ex:conjclasses} that conjugacy classes in $B_n$ are determined by cycle type and are parametrized by signed partitions $\sgnpart$. Our first goal is to construct a standard element for each conjugacy class $\cc_{\sgnpart}$ of $B_n$ which we will call $\sigma_{\sgnpart}$. We will then express the representation $\vazrep$ as an induced representation from the centralizer $Z_{\sgnpart}$ of $\sigma_{\sgnpart}$. 

Fix the signed partition $\lambda = \sgnpart$; we first obtain $\sigma_{\sgnpart}$ as follows. As before, let $\widehat{\lambda_{i}} = \sum_{j = 1}^{i}|\lambda_{j}|$ and define $\Lambda_i:=\Lambda(\lambda_i):=  \{ \widehat{\lambda_{i-1}} + 1,\widehat{\lambda_{i-1}} +2, \cdots, \widehat{\lambda_{i}}\}.$ If $\lambda_i \in \lambda^{+}$, define the positive $\lambda_{i}$-cycle 
\[ c_{i}:= ((\widehat{\lambda_{i-1}} + 1)\widehat{\lambda_{i-1}} \cdots \widehat{\lambda_{i}}). \]
By construction $c_{i}$ has order $|\lambda_{i}|$.

To construct a ``standard'' negative cycle, recall that for any $J \subset [n]$, the element $w_{0,J} \in B_n$ is the product of $t_i$ for $i \in P$ and that for $j_1, j_2, \cdots, j_\ell \in [n]$, 
\[ (j_1 \cdots j_\ell)^{-} \]
is the negative cycle in $B_n$ where $j_i \mapsto j_{i+1}$ for $1 \leq i \leq \ell - 1$ and $j_\ell \mapsto \overline{j_1}$. Then for $\lambda_i \in \lambda^{-}$, define the negative $|\lambda_{i}|$ cycle
\[ d_{i}: = \begin{cases} c_{i} \cdot w_{0,\Lambda_{i}} & |\lambda_{i}| \text{ is odd} \\
((\widehat{\lambda_{i-1}} + 1)\widehat{\lambda_{i-1}} \cdots \widehat{\lambda_{i}})^{-} & |\lambda_{i}| \text{ is even.} 
\end{cases} \]
The cycle $d_i$ has order $2|\lambda_i|$ in $B_n$.
The signed permutation 
\[ \sigma_{\sgnpart}:= c_{1} \cdots c_{\ell(\lambda^{+})} d_{\ell(\lambda^{+})+1} \cdots d_{\ell(\lambda)  - 1}  \]
will be our standard representative of $\cc_{\sgnpart}$. 

 Finally, suppose $\lambda_i = \lambda_{i+1}$ and both $\lambda_i, \lambda_{i+1}$ occur in either $\lambda^{+}$ or $\lambda^{-}$. Define $\delta_i$ to be a particular choice of permutation swapping the blocks $\Lambda_i$ and $\Lambda_{i+1}$.  Specifically, 
\[\delta_i: j \mapsto \begin{cases} j & j \not \in \Lambda_i \cup \Lambda_{i+1} \\
j + |\lambda_i| & j \in \Lambda_i\\
j - |\lambda_i| & j \in \Lambda_{i+1}.
\end{cases} \]

The centralizer $Z_{\sgnpart}$ of $\sigma_{\sgnpart}$ is then generated by 
\begin{itemize}
    \item $c_i$ and $w_{0, \Lambda_{i}}$ for every $\lambda_i \in \lambda^{+}$
    \item $d_i$ for every $\lambda_i \in \lambda^{-}$;
    \item $\delta_i$ for every pair $\lambda_i = \lambda_{i+1}$ for $1 \leq i \leq \ell(\lambda)- 1$. (Again, we require that $\lambda_i, \lambda_{i+1}$ are both in $\lambda^{+}$ or both in $\lambda^{-}$.)
\end{itemize}

\begin{example} \rm
Suppose $\lambda = \sgnpart = ((2,1),(2,2))$. Then
the relevant elements of $B_7$ (written in one-line notation) are
\begin{align*}
    c_1 &= (2,1,3,4,5,6,7)\\
    w_{0, \{ 1, 2\}} &= (\overline{1},\overline{2},3,4,5,6,7)\\
    c_2 &= (1,2,3,4,5,6,7) \\
    w_{0, \{ 3\}} &= (1,2,\overline{3},4,5,6,7)\\
    d_3 &= (1,2,3,5,\overline{4},6,7)\\
    d_4 &= (1,2,3,4,5,7,\overline{6})\\
    \sigma_{((2,1),(\overline{2},\overline{2})}&= (2,1,3,5,\overline{4},7,\overline{6})\\
    \delta_{3} &= (1,2,3,6,7,4,5).
\end{align*}
\end{example}
Since $c_i, d_i, w_{0, J}$ and $\delta_i$ generate $Z_{\sgnpart}$, any representation of $Z_{\sgnpart}$ is determined by its values on these elements. 

Define $\omega_{k}:= e^{2 \pi i / k}$. The representation of $Z_{\sgnpart}$ we are interested in is 1-dimensional, that is, a map $Z_{\sgnpart} \to \C$.

\begin{definition} \rm
Let $\rho_{\sgnpart}$ be the character of $Z_{\sgnpart}$ given by
\[  \rho_{\sgnpart}(\sigma):= \begin{cases}
\omega_{|\lambda_i|} & \sigma = c_{i} \text{ for }\lambda_i \in \lambda^{+}\\
\omega_{2|\lambda_{i}|} & \sigma = d_{i} \text{ for } \lambda_i \in \lambda^{-}\\
1 & \sigma = w_{0,\Lambda_{i}} \text{ for } 1 \leq i \leq \ell(\lambda^{+})\\
1 & \sigma = \delta_{i} \text{ for } 1 \leq i \leq \ell(\lambda) -1 \text{ and } \lambda_{i} = \lambda_{i+1}.
\end{cases} \]
\end{definition}

The representation $\rho_{\sgnpart}$ is of interest because of the following theorem. 
\begin{theorem}[Douglass--Tomlin  \cite{douglass2018decomposition}]
There is an isomorphism of $B_{n}$-representations
\[ \vazrep = \ind_{Z_{\sgnpart}}^{B_{n}} \rho_{\sgnpart}. \]
\end{theorem}

Our aim going forward is to connect the representations $\vazrep$ to the topological spaces we have been studying.

\subsubsection{Decomposition of $\G$ by signed set partitions}\label{sec:decompsignedset}
Recall from Definition \ref{def:G} that $\G$ is the associated graded ring of $\gr(\bconfhom^1) \cong \bconfhom^3$ with respect to the filtration described in Proposition \ref{prop:filbyz}. We will use the presentation of $\G$ described by Remark \ref{rmk:difpres} to show that $\G$ admits an even finer decomposition by \emph{signed set compositions}, defined below.

\begin{definition}\rm
Let $J_1, J_2 \subset [n]$ be any partition of $[n]$. Then $\alpha = (\alpha^{+}, \alpha^{-})$ is a \emph{signed set partition} of $[n]$ if $\alpha^{+}$ is a set partition of $J_1$ and $\alpha^{-}$ is a set partition of $J_2$. \\

\noindent Let $\Lambda_{[n]}^{\pm}$ be the collection of all such signed set partitions of $[n]$.
\end{definition}
Note that we allow either $\alpha^{+}$ or $\alpha^{-}$ to be empty. The blocks in $\alpha^{+}$ are said to be \emph{positive} while the blocks in $\alpha^{-}$ are said to be \emph{negative}. Let $\ell(\alpha^{+})$ (respectively, $\ell(\alpha^{-})$) be the number of parts of $\alpha^{+}$ (respectively, $\alpha^{-}$).

Write $E = \Z[z_{ij}^{+}, z_{ij}^{-}, z_{i}]$ for $1 \leq i< j \leq n.$ In order to decompose $E$ by $\Lambda^{\pm}_{[n]}$, we introduce the \emph{type map} of a monomial $f \in E$.

\begin{definition} \rm
From $f$, create a graph $G(f) = (V,E)$, with $V = [n]$ and $E$ described by: 
\begin{itemize}
    \item For every $z_{ij}^{+}$ or $z_{ij}^{-}$ appearing in $f$, there is an edge between $i$ and $j$
    \item For every $z_{i}$ appearing in $f$, there is a loop at the vertex $i$. 
\end{itemize}
Define the $\type_{n}(f) \in \Lambda_{[n]}^{\pm}$ to be the signed set partition with blocks comprised of the connected components of $G(f)$; if a connected component contains a loop, it is a negative block, and otherwise it is a positive block.
\end{definition}

\begin{example} \rm
Suppose $n=8$ and $f = (z_{12}^{-})^{2}z_{5}z_{56}^{-}z_{7}^{3}$. Then 
\[  \type_{8}((z_{12}^{-})^{2}\cdot z_{5}z_{56}^{-} \cdot z_{7}^{3}) = (\underbrace{(\{ 1,2 \}, \{ 3 \}, \{ 4\}, \{ 8 \})}_{(\type_{8}(f))^{+}}, \underbrace{( \{5,6 \}, \{ 7\})}_{(\type_{8}(f))^{-}}).  \]
\end{example}

The map $\type_{n}$ is  well-defined on monomials, but not polynomials.
Note that $\type_{n}$ is surjective but not injective. (To see that it is surjective, note that for any $\alpha \in \Lambda^{\pm}_{[n]}$, one can easily construct an $f \in E$ with $\type(f) = \alpha$.)

\begin{definition} \rm
For a signed set partition $\alpha = (\alpha^{+}, \alpha^{-}) \in \Lambda^{\pm}_{[n]}$, define
\[ E_{\alpha}:= \Z [\{ \text{monomials } f \in E: \type_{n}(f) = \alpha\}]. \]
\end{definition}

Note that by the above description, there is a vector space decomposition 
\[  E = \bigoplus_{\alpha \in \Lambda_{[n]}^{\pm}} E_{\alpha}.\]
We would like to show that this decomposition descends to the quotient $\G$ as well. 

Recall that \[ \G = \Z[ z_{ij}^{+}, z_{ij}^{-}, z_i ]/ \mathcal{L}' \]
for $1 \leq i < j \leq n$, where $\mathcal{L}'$ is generated by $ (z_{ij}^{+})^{2} = (z_{ij}^{-})^{2} = (z_{i})^{2}$ and 
\[  (i) \hspace{.5em} z_{ij}^{+}z_{ij}^{-}\hspace{1.5em} (ii) \hspace{.5em} z_{ij}^{+} z_{j} - z_{ij}^{+}z_{i}  \hspace{1.5em} (iii)  \hspace{.5em} z_{ij}^{-} z_{j} + z_{ij}^{-} z_{i} \hspace{1.5em} (iv) \hspace{.5em} z_{ij}^{+}z_{jk}^{+} - z_{ij}^{+}z_{ik}^{+} - z_{jk}^{+}z_{ik}^{+}, \hspace{1.5em}\]
\[  (v) \hspace{.5em} z_{ij}^{-}z_{jk}^{+} - z_{ij}^{-}z_{ik}^{-} - z_{jk}^{+}z_{ik}^{-} \hspace{1.5em} (vi) \hspace{.5em} z_{ij}^{-}z_{jk}^{-} - z_{ij}^{-}z_{ik}^{+} - z_{jk}^{-}z_{ik}^{+} \hspace{1.5em} (vii)  \hspace{.5em} z_{ij}^{+}z_{jk}^{-} - z_{ij}^{+}z_{ik}^{-} - z_{jk}^{-}z_{ik}^{-}.\]

\begin{prop}
With $\G$ as above, there is a vector space decomposition 
\[  \mathcal{L}' = \bigoplus_{\alpha \in \Lambda^{\pm}_{[n]}} \mathcal{L}'\cap E_{\alpha}, \]
inducing the decomposition 
\[ \G = \bigoplus_{\alpha \in \Lambda^{\pm}_{[n]}} \G_{\alpha}, \]
where $\G_{\alpha}:= E_{\alpha} / (\mathcal{L}' \cap E_{\alpha})$.
\end{prop}
\begin{proof}
It is sufficient to check that each relation in $\mathcal{L}'$ has summands in a single $E_{\alpha}$ for some $\alpha \in \Lambda_{[n]}^{\pm}$.
Relation $(i)$ has type $\alpha^{+} = \{ \{ i, j \}, \  \{ k \}_{k \neq i,j} \}$ and $ \alpha^{-} = \emptyset$.
Each summand in relations $(ii)$ and $(iii)$ has type $\alpha^{+} = \{  \{ k \}\}_{k \neq i,j}$ and $\alpha^{-} = \{ i, j \}$, while each summand in relations $(iv)$---$(vii)$ has type $\alpha^{+} = \{ i, j, k \} , \{  \{ \ell \}\}_{\ell \neq i,j,k},$ and
$\alpha^{-} = \emptyset.$  
\end{proof}

The decomposition of $\G$ by signed set partitions is also compatible with the bi-grading on $\G$.
\begin{prop}\label{prop:bigradingGalpha}
For $0 \leq \ell \leq k \leq n$,
\[ \G_{k,\ell} = \bigoplus_{\substack{\alpha = (\alpha^{+}, \alpha^{-}) \in \Lambda_{[n]}^{\pm}\\ \ell(\alpha^{+}) = n-k, \\ \ell(\alpha^{-}) = \ell}} \G_{\alpha} \]
\end{prop}
\begin{proof}
Fix $\alpha \in \Lambda_{[n]}^{\pm}$ with $\ell(\alpha^{+}) = n-k $ and $\ell(\alpha^{-}) = \ell$. Consider $f \in \G_{\alpha}$. 
Factor $f$ by the blocks of $\alpha$, and write $f_{i}$ as the factor of $f$ with support $\alpha_{i} \in \alpha$. Note that if $\alpha_{i} \in \alpha^{+}$, then $f_{i} \in \G_{|\alpha_{i}|-1, 0}$. Similarly, if $\alpha_{j} \in \alpha^{-}$, then  $f_{j} \in \G_{|\alpha_{j}|, 1}$.

Hence
\begin{align*}
   f=f_{1} \cdots f_{n-k} \cdot f_{n-k+1} \cdots f_{n-k+\ell} \in& \underbrace{\left( \G_{|\alpha_{1}|-1,0} \cdot \G_{|\alpha_{2}|-1,0}\cdots \G_{|\alpha_{n-k}|-1,0}\right)}_{\alpha^{+}} \cdot \underbrace{\left(\G_{|\alpha_{n-k+1}|,1} \cdots \G_{|\alpha_{n-k+\ell}|,1} \right)}_{\alpha^{-}} \\
   &\subseteq \G_{|\alpha^{+}| - (n-k),0} \cdot \G_{|\alpha^{-}|, \ell} \\
   & \subseteq \G_{k, \ell},
\end{align*}
because $|\alpha^{+}| + |\alpha^{-}| = n$.

For the other containment, consider any monomial $f \in \G_{k,\ell}$ and let $\alpha = \type_{n}(f)$. Using the above notation, we have that $f_{i} \in \G_{|\alpha_{i}|-1, 0}$ if $\alpha_{i} \in \alpha^{+}$ and $\G_{|\alpha_{i}|, 1}$ if $\alpha_{i} \in \alpha^{-}$. Since $f \in \G_{k,\ell}$, we must therefore have $\ell(\alpha^{-}) = \ell$, and 
\begin{align*}
         \deg \left( \prod_{\substack{1 \leq i \leq n \\ \alpha_{i} \in \alpha^{+}}} f_{i} \right) &= |\alpha^{+}|-\ell(\alpha^{+}),\\
    \deg \left( \prod_{\substack{1 \leq j \leq n \\ \alpha_{j} \in \alpha^{-}}} f_{j} \right) &= |\alpha^{-}|.
\end{align*}
Since $\deg(f) = k$, 
\[ k = |\alpha^{+}| - \ell(\alpha^{+}) + |\alpha^{-}| = n-\ell(\alpha^{+}), \]
from which it follows that $\ell(\alpha^{+}) = n-k$.
\end{proof}

\subsection{Induced representations}
Using the decomposition of $\G$ by $\Lambda^{\pm}_{[n]}$, we may begin to describe the pieces $\G_{\alpha}$ as induced representations. 

Given $\alpha = (\alpha^{+}, \alpha^{-})= (\alpha_1, \cdots, \alpha_\ell) \in \Lambda^{\pm}_{[n]}$, let $[\alpha] \in \mathcal{SC}(n)$ be the signed integer composition 
\[ [\alpha] = (\# \alpha_1, \cdots, \# \alpha_\ell), \]
where $\# \alpha_i\in [n]^{\pm}$ is the size of the set $\alpha_i$ with $\# \alpha_i \in [n]$ if $\alpha_i \in \alpha^{+}$ and in $[n]^{-}$ if $\alpha_i \in \alpha^{-}$. As before, $\overleftarrow{[\alpha]} = (\overleftarrow{[\alpha]^{+}}, \overleftarrow{[\alpha]^{-}})$ is the signed integer partition obtained by ordering the parts of $[\alpha]$ in decreasing order. 

Consider the $B_n$ orbit of $\G_{\alpha}$, which will be indexed by signed partitions $\sgnpart$:

\[ \G_{\sgnpart}:= \bigoplus_{\substack{ \alpha \in \Lambda^{\pm}_{[n]}\\  \overleftarrow{[\alpha]} = \sgnpart}} \G_{\alpha}. \]

The first step is to compute the characters of $\G_{n,1} = \G_{(0, (n))}$ and $\G_{n-1,0} = \G_{((n), 0)}$ which will form a base-case for the other pieces of $\G_{\sgnpart}$.

\begin{theorem}\label{thm:gn1induce}
As $B_{n}$-representations
\begin{align*}
    \G_{n,1} = \ind_{Z_{(\emptyset, (n))}}^{B_{n}} \rho_{(\emptyset, (n))} = G_{(\emptyset, (n))}, \hspace{5em}
    \G_{n-1,0} = \ind_{Z_{((n), \emptyset)}}^{B_{n}} \rho_{((n), \emptyset)} = G_{((n), \emptyset)}.
\end{align*}
\end{theorem}
To prove Theorem \ref{thm:gn1induce}, we will first prove that $\G_{n,1}$ is isomorphic to a different induced representation, and then appeal to Lemma \ref{lemma:coxeterinduction}. Both the proof of Theorem \ref{thm:gn1induce} and Lemma \ref{lemma:coxeterinduction} will use techniques developed by Berget \cite{berget2018internal} and in the case of Lemma \ref{lemma:coxeterinduction}, by Douglass--Tomlin \cite{douglass2018decomposition}. 

\begin{lemma}\label{lemma:coxeterinduction}
Let $\eta$ be the $n$-cycle $(12\cdots n) \in S_n$. Then as $B_{n}$-representations, 
\[ \ind_{\langle \eta, -1 \rangle}^{B_n} \chi =  \ind_{Z_{(\emptyset, (n))}}^{B_{n}} \rho_{(\emptyset, (n))} \]
where $\chi(\eta) = \omega_n$ and $\chi(-1) = -1$.
\end{lemma} 
\begin{proof}
Note that $Z_{(\emptyset, (n))}$ is the cyclic group generated by a Coxeter element $c$ of $B_n$ and 
\[ \rho_{(\emptyset, (n))}(c) = \omega_{2n}. \]

When $n$ is odd, we use similar methods to Berget in \cite[Corollary 9.2]{berget2018internal}. In this case, $-\eta$ is a Coxeter element with eigenvalue $-e^{2 \pi i / n}$, implying that $\langle \eta, -1 \rangle = Z_{(\emptyset, (n))}$. Because $n$ is odd, $-e^{2 \pi i / n}$ has order $2n$ and is in fact a primitive $2n$-th root of unity. Thus the representation $\chi$ of $\langle \eta, -1 \rangle$ coincides with the representation $\rho_{(\emptyset, (n))}$ of $Z_{(\emptyset, (n))}$. 

When $n$ is even, we use a result by Douglass-Tomlin \cite[Prop 4.1]{douglass2018decomposition}, which states that for even $n$,
\[ \ind_{\langle \eta,-1 \rangle}^{B_{n}} \chi =  \ind_{Z_{(\emptyset, (n))}}^{B_{n}} \rho_{(\emptyset, (n))}.\]
Note that unlike the odd case, the above is a statement about the induced representations rather than the representations $\chi$ and $\rho_{(\emptyset, (n))}$. 
\end{proof}

We may now prove Theorem \ref{thm:gn1induce}.

\begin{proof}[Proof of Theorem \ref{thm:gn1induce}]
Define $v:= z_{1}z_{12}^{+}z_{23}^{+} \cdots z_{(n-1)n}^{+} \in \G_{n,1}$, and let $V_{n}$ be the $S_{n}$-module generated by $v$ (cf. proof of Berget \cite[Thm. 9.1, Cor. 9.2]{berget2018internal}.) 
We claim that as an $S_{n}$-representation, $V_{n} \cong_{S_{n}} \lie_{n}$, the multilinear component of the free Lie algebra on $n$ generators (see Example \ref{ex:lie}). This follows from comparing the presentation in Remark \ref{rmk:difpres} with the presentation of $\lie_n$ discussed in Example \ref{ex:lie}. In particular, sending $v$ to $z_{12}z_{23} \cdots z_{(n-1)n} \in H^{(d-1)(n-1)}\conf_n(\R^d)$ (for $d \geq 3$, odd) induces an $S_n$ equivariant isomorphism since $z_{i}z_{ij}^{+} = z_{j}z_{ij}^{+}$ in $V_n$ (see Remark \ref{rmk:difpres}, relation $(ii)$ of $\mathcal{L}'$). Thus $V_{n} \cong \ind_{\langle w \rangle}^{S_{n}} e^{2 \pi i/n}$, where $\eta$ is an $n$-cycle in $S_{n}$.

Note that $w_{0} = -1$ acts on $z_{ij}^{+}$ trivially in $\G$, and $w_{0}z_{i} = -z_{i}$. Hence for any monomial in $\G_{n,1}$, $w_{0}$ acts as $-1$. Because $w_{0}$ is central, it follows that the $\langle S_{n}, w_{0} \rangle $ module generated by $v$ is isomorphic to
\[ \ind_{\langle \eta, w_{0} \rangle}^{\langle S_{n}, w_{0} \rangle} \chi, \]
where $\chi(\eta) = e^{2 \pi i/n}$ and $\chi(w_{0}) = -1$. Write this module as $V'_{n}$, and note that
\[ \dim(V'_{n}) = \dim(V_{n}) = (n-1)! \]

Now consider the $B_{n}$-module generated by $v$. Inspection shows that one obtains all of $\G_{n,1}$ by acting on $v$ by $B_{n}$. Hence there is a surjection of $B_{n}$-modules:
\[ \ind_{\langle S_{n}, w_{0} \rangle}^{B_{n}} V'_{n} \to \G_{n,1}, \]
where 
\[ \dim \left( \ind_{\langle S_{n}, w_{0} \rangle}^{B_{n}} V'_{n}\right) = \frac{n!2^{n}}{2 n!}\dim(V'_{n}) = 2^{n-1} (n-1)! = \dim(\G_{n,1}). \]
The dimension count of $\G_{n,1}$ follows from picking a cyclic ordering on the set $[n]$ which will determine the indices of the $n-1$ generators $z_{i_{1}i_{2}}^{\zeta_{1}}z_{i_{2}i_{3}}^{\zeta_{2}} \cdots z_{i_{n-1}i_{n}}^{\zeta_{n-1}}$, and then choosing $\zeta_{j} \in \{ + , - \}$ for $0 \leq j \leq n-1$. (The choice of $z_i$ is irrelevant by the relations in $\G$.)

Thus this surjection is in fact an isomorphism. By transitivity of induction, 
\begin{align*} \ind_{\langle \eta, -1 \rangle}^{B_{n}}\chi  &= \ind_{\langle S_{n}, -1 \rangle}^{B_{n}} \ind_{\langle \eta, -1 \rangle}^{\langle S_{n}, -1 \rangle}(\chi) \\
&= \ind_{\langle S_{n}, -1 \rangle}^{B_{n}}(V'_{n})\\ 
&= \G_{n,1}.
\end{align*}
The claim then follows by Lemma \ref{lemma:coxeterinduction}.

An identical argument (without the need for Lemma \ref{lemma:coxeterinduction}) shows the claim for $\G_{n-1,0}$, where now $-1$ acts as trivially on every monomial in $G_{n-1,0}$.
\end{proof}
\begin{remark} \rm
As suggested by the proof of Theorem \ref{thm:gn1induce}, the space $\G_{n,1}$ seems to be a good candidate for a Type $B$ analog of $\lie_n$ in the sense that 
\begin{itemize}
    \item $\G_{n,1}$ can be described as an induced representation from a Coxeter element of $B_n$ and 
\item One can show that as a $B_{n-1}$-representation, $\G_{n,1}$ restricts to the regular representation $\Q[B_{n-1}].$
\end{itemize}
A natural next step would be to extend some of the other properties of $\lie_n$ in Type $A$; see for instance \cite{aguiarmahajan}.
\end{remark}

\begin{remark} \rm
In the case that $n$ is odd, the $B_{n}$-representation carried by  $\G_{n,1}$ coincides with a representation studied by Berget in \cite{berget2018internal} coming from the internal zonotopal algebra of the Type $B$ hyperplane arrangement. In the case that $n$ is even, the $B_n$-module structure is almost the same, except that the element $w_0$ acts trivially in Berget's representation while it acts as $-1$ on $\G_{n,1}$.
\end{remark}

We are almost ready to prove the main result of this section (Theorem \ref{thm:BRiso}), but need one last lemma. For $\sgnpart$, write 
\begin{align}
    \lambda^{+} &= (1^{m_{1}}, \cdots, n^{m_{n}})\\
    \lambda^{-} &= (1^{q_{1}}, \cdots, n^{q_{n}}),
\end{align}
where $m_i, q_j \geq 0$. Then fix $\alpha \in \Lambda_{[n]}^{\pm}$ so that  $[\alpha] = \lambda = \sgnpart$, and 
\[ \alpha_i = \{ \widehat{\lambda_{i-1}} + 1,\widehat{\lambda_{i-1}}, \cdots \widehat{\lambda_{i}} \}.  \]
Consider $H_{\alpha}$, the set-wise stabilizer of $\G_{\alpha}$. One can write explicitly the isomorphism type of $H_{\alpha}$: 
\[ H_{\alpha} = \prod_{\substack{1 \leq i \leq n\\ m_i \neq 0}} B_{i} \wr S_{m_{i}} \times \prod_{\substack{1 \leq j \leq n \\ q_j \neq 0}} B_{j} \wr S_{q_{j}}.\]
By construction, $Z_{\sgnpart}$ is a subgroup of $H_{\alpha}$.
\begin{lemma}\label{lemma:indGalpha}
For $\alpha \in \Lambda_{[n]}^{\pm}$ as described above, there is an $H_{\alpha}$-isomorphism
\[ \G_{\alpha} = \ind_{Z_{\sgnpart}}^{H_{\alpha}} \varphi_{\sgnpart}. \]
\end{lemma}
\begin{proof}
Recall that in the character description of $\varphi_{\sgnpart}$, we have $\varphi_{\sgnpart}(\delta_{i}) = 1$, where $\delta_{i}$ swaps blocks of the same size that are both in the positive (resp. negative) part of $\alpha$. Hence in considering the induced representation of $\varphi_{\sgnpart}$ from $Z_{\sgnpart}$ to $H_{\alpha}$, it is enough to understand the representation on the distinct, commuting factors of $H_{\alpha}$ coming from each $\alpha_{i}$. 

Similarly, the action by $\delta_{i}$ on $\G_{\alpha}$ is trivial, so again to understand the $H_{\alpha}$-representation on $\G_{\alpha}$ it is sufficient to describe the distinct commuting factors of $H_{\alpha}$.

Write $a_i:= \# \alpha_i$. In the case of $\ind_{Z_{\sgnpart}}^{H_{\alpha}} \varphi_{\sgnpart}$, the $B_{a_i}$-representation coming from the block $\alpha_i$ in $\alpha^+$ is precisely  
 \[ \ind_{Z_{((a_i), \emptyset)}}^{B_{a_i}} \varphi_{((a_i), \emptyset)} \]
 and for the block of size $\alpha_j$ in $\alpha^{-}$ the corresponding $B_{a_j}$-representation is
 \[ \ind_{Z_{(\emptyset, (a_j))}}^{B_{a_j}} \varphi_{(\emptyset, (a_j))}.\]
By Theorem \ref{thm:gn1induce}, 
\begin{align*}
    \G_{a_i-1,0} &= \ind_{Z_{(a_i), \emptyset)}}^{B_{a_i}} \varphi_{((a_i) \emptyset)}\\
    \G_{a_j,1} &= \ind_{Z_{(\emptyset, (a_j))}}^{B_{a_j}} \varphi_{(\emptyset, (a_j))}.
\end{align*}
Hence the representations agree on each commuting factor of $H_{\alpha}$, and therefore must agree on all of $H_{\alpha}$.

\end{proof}

\begin{theorem}\label{thm:BRiso}
As $B_{n}$-representations, 
\[  \G_{\sgnpart} = \ind_{Z_{\sgnpart}}^{B_{n}} \varphi_{\sgnpart}, \]
and therefore 
\[  \G_{\sgnpart} \cong \vazrep.\]
\end{theorem}

\begin{proof}
By construction, 
\[ \ind_{H_{\alpha}}^{B_{n}} \G_{\alpha} = \G_{\sgnpart}.\]
By transitivity of induction and Lemma \ref{lemma:indGalpha},
\[ \ind_{Z_{\sgnpart}}^{B_{n}} \varphi_{\sgnpart} = \ind_{H_{\alpha}}^{B_{n}} \left( \ind_{Z_{\lambda}}^{H_{\alpha}} \varphi_{\sgnpart} \right) = \ind_{H_{\alpha}}^{B_{n}} \G_{\alpha} = \G_{\sgnpart}.\]
\end{proof}

\subsection{Consequences}
Theorem \ref{thm:BRiso} has several implications for the representations generated by the idempotents $\frakg_k$,
\[ G_n^{(k)}= \frakg_k \Q[B_n]. \]

\begin{theorem}\label{cor:mainiso}There is a $B_n$-representation isomorphism 
\[ G_{n}^{(n-k)} \cong \G_{k} = H^{2k}\bconf^3. \]
\end{theorem}
\begin{proof}
Note that 
\[ \frakg_{n-k} = \sum_{\substack{\sgnpart \\ \ell(\lambda^{+}) = n-k}} \vazidem \]
and 
\[ \G_{k} = \bigoplus_{\ell = 0}^{k}  \G_{k,\ell}  = \bigoplus_{\substack{ (\alpha^+, \alpha^-) \in \Lambda_{[n]}^{\pm}\\ \ell(\alpha^{+}) \ = \ n-k}} \G_{\alpha}. \]
The claim follows by applying Theorem \ref{thm:BRiso}.
\end{proof}

As a further consequence, this implies that the $G_{n}^{(n-k)}$ have a lift to $B_{n+1}$. 
\begin{theorem}\label{cor:liftedmainiso}
The representations $G_{n}^{(k)}$ lift to $B_{n+1}$, where they are described by $H^{2k}\bconflift^3$.
\end{theorem}
\begin{proof}
By Corollary \ref{cor:gr1and3}, $\G_k \cong H^{2k}\bconf^3$, which lifts to $H^{2k}\bconflift^3$.
\end{proof}

\subsection{Further questions}
We conclude with some lingering questions that our investigation has brought to light:
\begin{enumerate}
\item In Type $A$, the Whitehouse representations have the form $\Q[S_{n+1}]f_{n+1}^{(k)}$ for certain idempotents $f_{n+1}^{(k)} \in \Q[S_{n+1}]$. Is there a family of idempotents in $\Q[B_{n+1}]$ that generate representations isomorphic to the graded pieces of $\bconfhomlift^3$?
\item The Type $A$ Eulerian idempotents generate a commutative subalgebra of $\Sigma[S_n]$ spanned by elements with the same descent number (e.g. descent set size). The idempotents $\frakg_{k}$ also generate a commutative subalgebra of $\Sigma'[B_{n}]$. Is there a combinatorial description of this subalgebra?
\item The Whitehouse representation $F_{n+1}^{(0)}$ has other combinatorial-topological interpretations, for instance related to the homology of the space of trees by work of Robinson--Whitehouse  \cite{ROBINSON1996245}, as well as the homology of subposets of the partition lattice by work of Sundaram \cite{sundaram1999homotopy}. Do the Type $B$ lifts have analogous interpretations? 
\item Can the results of this paper be extended to the complex reflection groups $G(r,1,n) \cong \Z_r \wr S_n$? 
\end{enumerate}
\subsection*{Acknowledgements}
The author is very grateful to Vic Reiner for guidance and encouragement at every stage in this project, to Sheila Sundaram for insightful questions that served as the initial inspiration for this work, to Monica Vazirani for sharing her undergraduate thesis, and to Marcelo Aguiar, Fran\c cois Bergeron, Patty Commins, Christophe Hohlweg, Allen Knutson, Nick Proudfoot, Franco Saliola and Dev Sinha for helpful discussions. The author is supported by the NSF Graduate Research Fellowship (Award Number DMS-0007404).
\bibliographystyle{abbrv}
\bibliography{bibliography}

\end{document}